\documentclass[a4paper,11pt]{article}

\usepackage{amsmath}
\usepackage{amsfonts}
\usepackage{amsthm}
\usepackage{amssymb}
\usepackage{a4wide}
\usepackage{mathrsfs}
\usepackage{epsfig}
\usepackage{palatino}
\usepackage{tikz,fp,ifthen}
\usepackage{float}
\usepackage[utf8]{inputenc}

\newcommand{\pdfgraphics}{\ifpdf\DeclareGraphicsExtensions{.pdf,.jpg}\else\fi}
\usepackage{graphicx}

\usepackage{color}
\definecolor{citegreen}{rgb}{0,0.6,0}
\definecolor{refred}{rgb}{0.8,0,0}
\usepackage[colorlinks, citecolor=citegreen,linkcolor=refred]{hyperref}

\usepackage{mathtools}

\numberwithin{equation}{section}

\theoremstyle{plain}

\newtheorem{definition}{Definition}[section]
\newtheorem{notation}[definition]{Notations}
\newtheorem{remark}[definition]{Remark}
\newtheorem{ass}[definition]{Assumptions}
\newtheorem{corollary}[definition]{Corollary}
\newtheorem{proposition}[definition]{Proposition}
\newtheorem{theorem}[definition]{Theorem}
\newtheorem{lemma}[definition]{Lemma}

\usepackage{xparse}
\let\olddefinition\remark
\RenewDocumentCommand{\remark}{o}{%
  \IfNoValueTF{#1}
    {\olddefinition}
    {\olddefinition[#1]}%
  \normalfont
}

\numberwithin{equation}{section}

\newcommand{\defl}{\mathrel{=\!\!\mathop:}}
\newcommand{\defr}{\mathrel{\mathop:\!\!=}}
\newcommand{\Div}{\mathrm{div}}
\newcommand{\Id}{\mathrm{Id}}

\def\E{\mathbb{E}}

\def\N{\mathbb{N}}
\def\R{\mathbb{R}}
\renewcommand{\setminus}{\mathbin{\backslash}}

\begin{document}
\pdfgraphics 

\title{Short Time Existence for Coupling of \\Scaled Mean Curvature Flow and Diffusion}

\author{Helmut Abels
\footnote{Fakult\"at f\"ur Mathematik, Universit\"at Regensburg, Universit\"atsstrasse 31, 
93053 Regensburg, Germany \newline
\hspace*{1.6em} helmut.abels@ur.de (corresponding author), felicitas.buerger@ur.de, harald.garcke@ur.de}
 \and Felicitas Bürger
\footnotemark[1]
\and Harald Garcke
\footnotemark[1]
}

\maketitle

\begin{abstract}
We prove a short time existence result for a system consisting of a geometric evolution equation for a hypersurface and a parabolic equation on this evolving hypersurface. More precisely, we discuss a mean curvature flow scaled with a term that depends on a quantity defined on the surface coupled to a diffusion equation for that quantity. The proof is based on a splitting ansatz, solving both equations separately using linearization and a contraction argument. Our result is formulated for the case of immersed hypersurfaces and yields a uniform lower bound on the existence time that allows for small changes in the initial value of the height function.
\end{abstract}



\section{Introduction}
We prove a short time existence result for the coupling of a scaled mean curvature flow describing the evolution of a surface and a diffusion equation for a quantity on this surface. More precisely, we investigate the system 
\begin{subequations}\label{eq_Intro_GLS}
\begin{align}
V &= \big(G(c) - G'(c)c\big) H, \label{eq_Intro_GLS1} \\
\partial^\square c &= \Delta_\Gamma \big( G'(c) \big) + cHV \label{eq_Intro_GLS2} 
\end{align}
\end{subequations}
defined on an evolving closed hypersurface $\Gamma$, whose normal velocity and mean curvature are given by $V$ and $H$, respectively. The function $c: \Gamma \rightarrow \R_{\geq 0}$ describes a quantity defined on this surface and $\partial^\square c$ is its normal time derivative (see Remark \ref{IMM_evolvHF_rho_bem}). Finally, $\Delta_\Gamma$ denotes the Laplace-Beltrami operator on $\Gamma$. We will often use the short notation 
\begin{align*}
g(c) \defr G(c)-G'(c)c,
\end{align*}
which appears in the right hand side of \eqref{eq_Intro_GLS1}. The function $G: \R_{\geq 0} \rightarrow \R$ can be interpreted as a (Gibbs) energy density, as a solution $(\Gamma,c)$ to the system \eqref{eq_Intro_GLS} reduces the energy
\begin{align}\label{eq_Intro_Energy}
\mathrm{E}\big(\Gamma(t),c(t)\big) \defr \int_{\Gamma(t)} G\big(c(t)\big) \, \mathrm{d}\mathcal{A}.
\end{align}
This follows easily using the transport theorem and the Gauß theorem for closed surfaces (see \cite[Proposition 2.58 and Proposition 2.48]{Buerger}):
\begin{align*}
\frac{\mathrm{d}}{\mathrm{d}t} \mathrm{E}(\Gamma,c)
&= \frac{\mathrm{d}}{\mathrm{d}t}  \int_\Gamma G(c) \, \mathrm{d}\mathcal{A} \\
&= \int_{\Gamma(t)} G'(c) \partial^\square c - G(c)HV \, \mathrm{d}\mathcal{A} \\
&= \int_{\Gamma(t)} G'(c) \Delta_\Gamma \big(G'(c)\big) + \big(G'(c)c - G(c)\big) HV \, \mathrm{d}\mathcal{A} \\
&= - \int_{\Gamma(t)} \big| \nabla_\Gamma G'(c) \big|^2 + V^2 \, \mathrm{d}\mathcal{A}
\leq 0.
\end{align*}
Hence, a solution $(\Gamma,c)$ of \eqref{eq_Intro_GLS} can never increase the energy functional $\mathrm{E}$. As long as the geometry of the system changes, i.e. $V \neq 0$, the energy will actually decrease. Also, assuming $G''$ to be positive, a non-uniform distribution of the quantity described by $c$ in general also leads to an actual decrease of the energy. Later on, we impose certain conditions on the function $G$ that guarantee parabolicity for our system (see Assumptions \ref{Vor}(i)). Also, they imply an actual decrease of the energy: Because a closed hypersurface $\Gamma(t)$ can not have vanishing mean curvature $H$ everywhere, equation \eqref{eq_Intro_GLS1} together with the parabolicity conditions yields that $V$ cannot be identically zero. \\
Furthermore, a solution $(\Gamma,c)$ to the system \eqref{eq_Intro_GLS} conserves the mass of the quantity described by the function $c$: Using again the transport theorem and the Gauß theorem for closed surfaces, we obtain
\begin{align*}
\frac{\mathrm{d}}{\mathrm{d}t} \, \int_{\Gamma(t)} c \, \mathrm{d}\mathcal{A}
= \int_{\Gamma(t)} \partial^\square c - cHV \, \mathrm{d}\mathcal{A}
= \int_{\Gamma(t)} \Delta_\Gamma \big(G'(c)\big) \, \mathrm{d}\mathcal{A}
= 0.
\end{align*}
So, from a physical point of view, we are interested in a system consisting of an evolving closed hypersurface $\Gamma$ and a concentration $c: \Gamma \rightarrow \R_{\geq 0}$ that can vary in space and time but fulfills mass conservation on $\Gamma$ and the development of this system, tending to increase the energy \eqref{eq_Intro_Energy}. \\
Mathematically, we discuss a parabolic PDE on an evolving hypersurface, where the evolution of the geometry is not given but part of the problem. To our knowledge, there is not yet much literature on this interesting coupling. For the one-dimensional curve case there exist some first results: Pozzi and Stinner investigate the numerical approximation of such a coupled problem and develop (semi-discrete) finite element schemes for the curve shortening flow (in \cite{PozziStinner1}) and the elastic flow (in \cite{PozziStinner2}) coupled with a diffusion equation on the curve. Barrett, Deckelnick and Styles consider a slightly more general version of the problem from \cite{PozziStinner1}, enhance the numerical analysis and end up with a fully discrete scheme (see \cite{BarrettDeckelnickStyles}). 
For higher dimensions, studying finite element methods for coupled problems is a difficult task. A first error analysis for the case of two-dimensional, closed surfaces has been achieved by Kovács, Li, Lubich and Power Guerra in \cite{KovacsLiLubichPowerGuerra}, leading to a FEM semi-discretization for regularized versions of geometric evolution equations. Kovács and Lubich extend these ideas and obtain a full-discretization, again for regularized versions of geometric evolution equations (see \cite{KovacsLubich_LinearlyImplicitFDOfSurfaceEvolution}). Both results apply to the coupling of a regularized mean curvature flow and a diffusion equation. In the later work \cite{KovacsLiLubich_AConvergentAlgorithmForForcedMCF}, Kovács, Li and Lubich finally prove a result without regularization and present a fully discrete FE algorithm for the coupled problem of mean curvature flow and a diffusion equation for two-dimensional closed surfaces. For the case of two-dimensional surfaces that can be represented as the graph of some function, Deckelnick and Styles investigate the problem from \cite{BarrettDeckelnickStyles} and derive a fully discrete finite element scheme (see \cite{DeckelnickStyles_FEEA}). \\
The considerations in \cite{PozziStinner1}, \cite{BarrettDeckelnickStyles}, \cite{DeckelnickStyles_FEEA} and \cite{KovacsLiLubich_AConvergentAlgorithmForForcedMCF} are of special interest to us because the problem statements therein are very similar to ours, discussing the coupling of a mean curvature flow-type equation and a diffusion equation. In contrast to all these previous contributions that concern modifications $V=H+f(c)$ of the mean curvature flow resulting from an additive term $f(c)$ (with $f(c)=c$ in the case of \cite{KovacsLiLubich_AConvergentAlgorithmForForcedMCF}), we deal with a multiplicatively scaled version $V=g(c)H$, $g(c) \defr G(c)-G'(c)c$ of the mean curvature flow. This seems more natural to us, as it arises from the physical situation explained above. Also, while the diffusion equations in the previous literature all are semilinear, i.e., $\partial^\square c = \alpha \Delta_\Gamma c + \text{l.o.t.}$ with a constant $\alpha>0$, our second equation $\partial^\square c = G''(c) \Delta_\Gamma c + \text{l.o.t.}$ is quasilinear. Be reminded that \cite{PozziStinner1} and \cite{BarrettDeckelnickStyles} only consider the one-dimensional case of closed curves and \cite{DeckelnickStyles_FEEA} and \cite{KovacsLiLubich_AConvergentAlgorithmForForcedMCF} restrict to the case of two-dimensional surfaces, represented as graph of a function or being closed, respectively. Our results however apply to closed hypersurfaces of arbitrary dimension. Finally, all four of the mentioned contributions address numerical analysis exclusively whereas this work is purely analytic and yields a short-time existence result. We also refer to the recent contribution of Elliott, Garcke and Kovács in \cite{ElliottGarckeKovacs} who analyze a finite element approximation of \eqref{eq_Intro_GLS} relying on the existence result presented in this work. In a forthcoming paper, we will discuss several properties of solutions to \eqref{eq_Intro_GLS}, placing emphasis on to what extent the hypersurface in our setting qualitatively evolves as for the usual mean curvature flow. For properties of the mean curvature flow, we point out the famous result of Huisken that convex, closed surfaces shrink to round points (see \cite{Huisken}) and recommend Mantegazza (\cite{Mantegazza}) for further literature and details. \\
Our system of equations \eqref{eq_Intro_GLS} is defined on an evolving hypersurface so that usual analytic methods can not be applied directly. But as we only consider the case of codimension $1$, the evolving hypersurface can be parameterized over a fixed reference surface via a real valued parameterization called height function $\rho$. Then, transforming the system \eqref{eq_Intro_GLS} onto the fixed reference surface yields a system consisting of an equation for the height function $\rho$ and another one for the transformed concentration $\tilde{c}$ (cf. Section \ref{ChapReformFixDom}). \\
Both equations are of second order, as the mean curvature and the Laplace-Beltrami operator are (quasilinear and linear, respectively) differential operators of second order. Due to $HV \sim H^2$, second order derivatives of the height function occur quadratically such that the system is fully non-linear. But as these derivatives of the height function appear in the equation for the concentration only, both equations remain quasilinear when considered separately. Suitable assumptions on the energy density function $G$ ensure parabolicity of the system (see Assumptions \ref{Vor}(i)). Hence, we consider a system of two parabolic, quasilinear differential equations of second order that are, of course, coupled. \\
From a mathematical point of view, this coupling makes the problem interesting and challenging. The proof of short-time existence uses a spitting ansatz: As a first step, we solve the first equation for $\rho$ with an arbitrary concentration $\tilde{c}$ in Section \ref{ChaplokExh} and then, we solve the second equation for $\tilde{c}$ where we insert the solution function $\rho_{\tilde{c}}$ from the first equation in Section \ref{ChaplokExc}. The approach for solving both equations has the same structure, relying, as usual for parabolic, not fully linear equations, on a linearization and a contraction argument. Nevertheless, the second order derivatives of the height function occuring in the equation for the concentration necessitate handling the second equation more carefully than the first, where the concentration only appears in lower order terms. Also, the quadratic occurence of these derivatives makes it clear that we have to use solution spaces that form an algebra with pointwise multiplication. Sobolev spaces do not have this property in general. Instead, we will work with little Hölder spaces, which in particular implies that we solve the transformation of the system \eqref{eq_Intro_GLS} in a classical sense.  The combined result is given in Section \ref{ChaplokExAnalyt}. It is formulated for the case of immersed hypersurfaces and yields a uniform lower bound on the existence time that allows for small changes in the initial value of the height function. \\

The present paper is based on the dissertation \cite{Buerger} of the second author.

\section{Preliminaries}

\subsection{Function Spaces}\label{HölderSpaces_Def}

For Banach spaces $X$ and $Y$, an open subset $U \subset Y$ and a natural number $k \in \N_{> 0}$, we use 
\begin{align*}
C^0(U,X) \phantom{xx} \text{and} \phantom{xx} C^k(U,X)
\end{align*}
to describe the continuous and the $k$-times continuously Fréchet-differentiable functions ${f: U \rightarrow X}$. An index $b$ means that the function itself and all its Fréchet-derivatives up to order $k$ are bounded as functions on $U$. For an open subset $W \subset \R^d$, $d \in \N_{>0}$,
\begin{align*}
C^0(\overline{W},X) \phantom{xx} \text{and} \phantom{xx} C^k(\overline{W},X)
\end{align*}
denote the functions that themselves and all their Fréchet-derivatives up to order $k$ are continuously extendable onto $\overline{W}$. Again, an index $b$ indicates boundedness and then, these spaces form Banach spaces with the usual norms. If $X=\R$, we omit the image space and for an arbitrary subset $V \subset X$, we define
\begin{align*}
C^k(\overline{W},V) \defr \big\{ f \in C^k(\overline{W},X) \, \big| \, f(x) \in V \text{ for all } x \in \overline{W} \big\}.
\end{align*}

\begin{definition}[Hölder Spaces on the Closure of Open Sets]\label{definition_Hölder}$\phantom{x}$ \\
For $\alpha \in (0,1)$ and $R \in (0,\infty]$, we define the seminorm
\begin{align*}
[f]_{h^\alpha(\overline{W},X)}^R \defr \sup_{\substack{x, y \in \overline{W}\\0<|x-y|<R}} \frac{\|f(x)-f(y)\|_X}{|x-y|^\alpha}
\end{align*}
as well as the little Hölder space
\begin{align*}
h^\alpha(\overline{W},X) \defr \left\{ f \in C^0_b(\overline{W},X) \, \Big| \, [f]_{h^\alpha(\overline{W},X)}^\infty < \infty \text{ and } \lim_{R \rightarrow 0} [f]_{h^\alpha(\overline{W},X)}^R = 0 \right\}.
\end{align*}
\end{definition}

Together with the norm
\begin{align*}
\|f\|_{h^\alpha(\overline{W},X)} \defr \|f\|_{C^0(\overline{W},X)} + [f]_{h^\alpha(\overline{W},X)}^\infty,
\end{align*}
it forms a Banach space. For $k \in \N_{>0}$, the little Hölder spaces of higher order
\begin{align*}
h^{k+\alpha}(\overline{W},X)
\end{align*}
are the functions in $C^k_b(\overline{W},X)$ whose highest order derivatives lie in $h^\alpha(\overline{W},X)$. They are endowed with the natural norm to form Banach spaces. For short notation, we use 
\begin{align*}
h^s_b(\overline{W},X)
\end{align*}
for $s \in \R_{\geq 0}$, meaning $C^s_b(\overline{W},X)$ if $s \in \N_{\geq 0}$ and $h^s(\overline{W},X)$ else. Note, that we assume a Hölder regular function to fulfill not only a local, but a uniform Hölder condition! \\
On $C^1 \cap h^{k+\alpha}$-embedded hypersurfaces $M \subset \R^{d+1}$ as introduced in Definition \ref{hypersurface_Def} below, we define Hölder functions 
\begin{align*}
h^{k+\alpha}(M,X) &\defr \left\{ f: M \rightarrow X \, \Big| \, \forall p \in M{:} \, \exists \text{ loc. param.} \, (\gamma,W){:} \, p \in \gamma(W), f \circ \gamma \in h^{k+\alpha}\big(\overline{W},X\big) \right\}
\end{align*}
with the help of local parameterizations $(\gamma,W)$. If $M$ is closed, we obtain the characterization 
\begin{align*}
h^{k+\alpha}(M,X) = \left\{ f: M \rightarrow X \, \Big| \, f \circ \gamma \in h^{k+\alpha}\big(\overline{W},X\big) \text{ for all loc. param.} \, (\gamma,W) \right\}
\end{align*}
as for the case of continous or continously differentiable functions. Also, if $M$ is closed, we can restrict to a finite set $(\gamma_l,W_l)_{l=1,...,L}$ of local parameterizations with
\begin{align*}
h^{k+\alpha}(M,X) = \left\{ f: M \rightarrow X \, \Big| \, f \circ \gamma_l \in h^{k+\alpha}\big(\overline{W_l},X\big) \text{ for all } l=1,...,L \right\}
\end{align*}
and then $h^{k+\alpha}(M,X)$ forms a Banach space with the norm
\begin{align*}
\|f\|_{h^{k+\alpha}(M,X)} &\defr \sum_{l=1}^L \|f \circ \gamma_l\|_{h^{k+\alpha}(\overline{W_l},X)}.
\end{align*}

\subsection{Generators of Semigroups}

\begin{definition}
Let $A: \mathcal{D}(A) \subset X \rightarrow X$ generate an analytic $C^0$-semigroup $\big(T(t)\big)_{t \geq 0}$ in a Banach space $X$. For $\beta \in (0,1)$, we define
\begin{align*}
\mathcal{D}_A(\beta) &\defr \left\{ x \in X \, \Big| \, \sup_{0<s \leq 1} s^{1-\beta} \big\| A T(s) x \big\|_X < \infty \text{ and } \lim_{s \searrow 0} s^{1-\beta} A T(s) x = 0 \right\}
\end{align*}
with 
\begin{align*}
\|x\|_{\mathcal{D}_A(\beta)} \defr \|x\|_X + \sup_{0<s \leq 1} s^{1-\beta} \big\| A T(s) x \big\|_X
\end{align*}
and for $T>0$, we set
\begin{align*}
\big( h^\beta([0,T],X) \times \mathcal{D}(A) \big)_+ &\defr \big\{ (f,x) \in h^\beta([0,T],X) \times \mathcal{D}(A) \, \big| \, Ax + f(0) \in \mathcal{D}_A(\beta) \big\}
\end{align*} 
with 
\begin{align*}
\| (f,x) \|_{(h^\beta([0,T],X) \times \mathcal{D}(A))_+} \defr \|f\|_{h^\beta([0,T],X)} + \|x\|_{\mathcal{D}(A)} + \|Ax+f(0)\|_{\mathcal{D}_A(\beta)}.
\end{align*}
\end{definition}

The space $\mathcal{D}_A(\beta)$ is given as real interpolation space by
\begin{align*}
\mathcal{D}_A(\beta) = \big(X,\mathcal{D}(A)\big)_\beta.
\end{align*}
As any operator generating an analytic $C^0$-semigroup is sectorial in the sense of \cite[Definition 2.0.1]{LunardiASG}, a proof of this representation can be found in \cite[Proposition 2.2.2]{LunardiASG}.

In the appendix (see Proposition \ref{L_bij_allg}), we state that $Au_0 + f(0) \in \mathcal{D}_A(\beta)$ is the suitable compability condition such that $h^\beta([0,T],X)$ is of maximal regularity for the initial value problem
\begin{align*}
\partial_t u - Au &= f \quad \text{ in } (0,T), \\
u(0) &= u_0.
\end{align*}
Moreover, we discuss differential operators $A$ acting on little Hölder spaces and formulate a condition that guarantees them to generate  analytic $C^0$-semigroups (see Proposition \ref{HG_elliptic_semigroup_M}).

\subsection{Geometric Setting}\label{GeometricSetting}

In the following, let $d,n \in \N_{>0}$ and $r,s \in \R_{\geq 0}$. As in Section \ref{HölderSpaces_Def}, we set $h^s \defr C^s_b$ if $s \in \N_{\geq 0}$.

\begin{definition}[Hölder-continuous local Parameterization]\label{Def_parameterization}$\phantom{x}$ \\
Let $M \subset \R^n$. A pair $(\gamma,W)$ is called a ($d$-dimensional) $h^{1+s}$-local parameterization of $M$ if $W \subset \R^d$ is an open, bounded and convex subset and $\gamma \in h^{1+s}(\overline{W},\R^n)$ is an embedding such that $\gamma(W) \subset M$ is an open subset with $\gamma(\overline{W}) \subset M$. Choosing the local parameterization $(\gamma,W)$ sufficiently small means that $\gamma(W) \subset M$ is sufficiently small.
\end{definition}

In contrast to the usual literature on submanifolds, we restrict to bounded and convex sets  and assume the corresponding local parameterizations to be well-defined on the closure of these sets. This is possible w.l.o.g., because we can always achieve these properties by choosing the sets smaller.

\begin{definition}[Embedded Hypersurface]\label{hypersurface_Def}$\phantom{x}$ \\
A subset $M \subset \R^{d+1}$ is called an $h^{1+s}$-embedded (closed) hypersurface if 
\begin{enumerate}
\item[(i)] $M$ is a $d$-dimensional $h^{1+s}$-embedded submanifold,
i.e., if for every point $p \in M$ there exists a $d$-dimensional $h^{1+s}$-local parameterization $(\gamma_p,W_p)$ of $M$ with $p \in \gamma_p(W_p)$,
\item[(ii)] $M$ is orientable such that there exists a continuous unit normal $\nu_M$, i.e., a continuous vector field $\nu_M: M \rightarrow \R^{d+1}$ with $|\nu_M(p)| = 1$ and $\nu_M(p) \perp T_pM$ for all $p \in M$ and
\item[(iii)] $M$ is connected (and compact) as subset of $\R^{d+1}$.
\end{enumerate}
\end{definition}

In this work, a hypersurface never contains a boundary. A unit normal automatically fulfills $\nu_M \in h^{s}(M,\R^{d+1})$. 

\begin{remark}\label{hypersurface_remark}
If $M$ is a closed hypersurface, i.e., compact as a subset of $\R^{d+1}$, it suffices to use finitely many local parameterizations $(\gamma_l,W_l)_{l=1,...,L}$ to cover it. W.l.o.g., we can assume the existence of further open subsets $U_l \subset M$ with $\overline{U_l} \subset \gamma_l(W_l)$ and
\begin{align*}
M \subset \bigcup_{l=1}^L U_l.
\end{align*}
\end{remark} 

\begin{definition}[Immersed Hypersurface]$\phantom{x}$ \\
Let $M \subset \R^{d+1}$ be an $h^{1+s}$-embedded (closed) hypersurface and let $\theta: M \rightarrow \R^{d+1}$ be an $h^{1+s}$-immersion, i.e., $\theta \in h^{1+s}(M,\R^{d+1})$ such that its differential $\mathrm{d}_p\theta: T_pM \rightarrow \R^{d+1}$ is injective for all $p \in M$. Then, $\Sigma \defr \theta(M) \subset \R^{d+1}$ is called an $h^{1+s}$-immersed (closed) hypersurface with reference surface $M$ and global parameterization $\theta$. 
\end{definition}

Just as for the embedded case, an immersed hypersurface never contains a boundary. Moreover, we remark that we do not use any topological structure on the immersed hypersurface $\Sigma$ itself but only consider the topology on the (embedded) reference surface $M$. \\
As locally any immersion is an embedding, for every point $p \in M$ there exists an open neighborhood $U \subset M$ such that $\theta(U) \subset \Sigma$ is an \textit{embedded patch}, i.e., an embedded hypersurface. 
Every locally defined term for embedded hypersurfaces thus can easily be defined also for immersed hypersurfaces, simply defining it on the embedded patches. To avoid confusion in points of self-intersection, we always use the pullback onto the reference surface $M$. \\ 
As for every embedded patch $\theta(U)$ the restriction $\theta_{|U}$ is an embedding and thus the differential $\mathrm{d}_p\theta: T_pM=T_pU \rightarrow T_{\theta(p)}\theta(U)$ is a linear isomorphism, the tangent space of $\Sigma$ at $\theta(p)$ is given by $T_p\Sigma \defr \mathrm{d}_p\theta(T_pM)$ for every $p \in M$. Furthermore one can show that orientability transfers from the embedded hypersurface $M$ to the immersed hypersurface $\Sigma = \theta(M)$, meaning that there exists a unit normal $\nu \in h^{s}(M,\R^{d+1})$ with $|\nu(p)|=1$ and $\nu(p) \perp T_p\Sigma$ for all $p \in M$ (see \cite[Proposition 2.27]{Buerger}). 

\begin{definition}[Evolving Hypersurface]\label{evolvHF_rho} $\phantom{x}$ \\
Let $\Sigma = \bar{\theta}(M) \subset \R^{d+1}$ be an $h^{3+s}$-embedded / immersed closed hypersurface with unit normal $\nu_\Sigma$ and let $T \in (0,\infty)$. Furthermore, let 
$\rho \in h^{1+r}\big([0,T],h^{s}(M)\big) \cap h^{r}\big([0,T],h^{2+s}(M)\big)$ with $\|\rho\|_{C^0([0,T] \times M)}$ sufficiently small. We define 
\begin{align*}
\theta_\rho: [0,T] \times M \rightarrow \R^{d+1}, \phantom{x} \theta_\rho(t,p) \defr \bar{\theta}(p) + \rho(t,p)\nu_\Sigma(p).
\end{align*}
Then, with $\Gamma_\rho(t) \defr \theta_\rho(t,M)$,
\begin{align*}
\Gamma_\rho \defr \big\{ \{t\} \times \Gamma_\rho(t) \, \big| \, t \in [0,T] \big\}
\end{align*}
is called the $h^{1+r}$-$\phantom{.}h^{2+s}$-evolving embedded / immersed hypersurface parameterized via the height function $\rho$ with reference surface $M$ and global parameterization $\theta_\rho$. 
\end{definition}

We have $\theta_\rho \in h^{1+r}\big([0,T],h^{s}(M,\R^{d+1})\big) \cap h^{r}\big([0,T],h^{2+s}(M,\R^{d+1})\big)$ by construction. Also, $\theta_\rho(t,\cdot): M \rightarrow \R^{d+1}$ is an embedding / immersion for all $t \in [0,T]$: This follows with the usual arguments concerning tubular neighborhoods of hypersurfaces for the embedded case (cf. \cite[Section 2.3]{PS} or \cite[Section III.3.2]{BoyerFabrie}) and is proven in the appendix (see Lemma \ref{theta_imm}) for the immersed case. \\
In the following remark, we introduce some basic notation for evolving closed hypersurfaces parameterized via height functions and list some important regularity properties.

\begin{remark}\label{IMM_evolvHF_rho_bem}
Let $\Gamma_\rho$ be an $h^{1+r}$-$\phantom{.}h^{2+s}$-evolving immersed closed hypersurface parameterized via a height function $\rho$ as in Definition \ref{evolvHF_rho} with reference surface $M \subset \R^{d+1}$ and global parameterization $\theta_\rho: [0,T] \times M \rightarrow \R^{d+1}$. Moreover, let $\Sigma = \bar{\theta}(M)$ be the corresponding immersed reference surface with unit normal $\nu_\Sigma$. We use the notation $\theta_{\rho(t)} \defr \theta_\rho(t,\cdot)$ and $\Gamma_\rho(t) \defr \Gamma_{\rho(t)} \defr \theta_{\rho(t)}(M)$ for all $t \in [0,T]$. Given a sufficiently small local parameterization $(\gamma,W)$ of $M$, we define 
\begin{align*}
\gamma_{\rho(t)} \defr \theta_{\rho(t)} \circ \gamma
\end{align*}
such that 
\begin{align*}
\gamma_\rho &\in h^{1+r}\big( [0,T], h^{s}(\overline{W},\R^{d+1})\big) \cap h^{r}\big( [0,T], h^{2+s}(\overline{W},\R^{d+1})\big)
\end{align*}
holds. We use $g^{\rho(t)}_{ij} \defr \partial_i \gamma_{\rho(t)} \cdot \partial_j \gamma_{\rho(t)}$ for the first fundamental form and $g_{\rho(t)}^{ij}$ for its inverse. 
There exists a vector field 
\begin{align*}
\nu_\rho \in h^{r}\big( [0,T], h^{1+s}(M,\R^{d+1})\big)
\end{align*} such that $\nu_\rho(t,\cdot)$ is a continuous unit normal to $\Gamma_\rho(t)$ for all $t \in [0,T]$ (see \cite[Proposition 2.51]{Buerger}). 
As spatial derivatives are defined locally, we employ the usual definitions on the embedded patches of $\Gamma_\rho$ and then use a pullback onto the reference surface $M$ via the para\-meterization $\theta_\rho$: 
For functions $f \in C^0\big([0,T],C^1(M,\R)\big)$ and $F \in C^0\big([0,T],C^1(M,\R^{d+1})\big)$, we define the surface gradient and surface divergence by
\begin{align*}
\nabla_\rho f \defr \big(\nabla_{\Gamma_\rho}(f \circ \theta_\rho^{-1}) \big) \circ \theta_\rho 
\phantom{xx} \text{ and } \phantom{xx}
\Div_\rho F \defr \big(\Div_{\Gamma_\rho}(F \circ \theta_\rho^{-1} \big) \circ \theta_\rho,
\end{align*}
respectively, and for $f \in C^0\big([0,T],C^2(M,\R)\big)$ we use the Laplace-Beltrami operator
\begin{align*}
\Delta_\rho f &\defr \big( \Delta_{\Gamma_\rho} (f \circ \theta_\rho^{-1} \big) \circ \theta_\rho.
\end{align*}
Their representations with respect to a sufficiently small local parameterization $(\gamma,W)$ of $M$ are given by
\begin{align*}
\nabla_\rho f \circ \gamma
&= \sum_{i,j=1}^d g_\rho^{ij} \, \partial_i (f \circ \gamma) \, \partial_j \gamma_\rho, \phantom{xxx}
\Div_\rho F \circ \gamma
= \sum_{i,j=1}^{d} g_\rho^{ij} \, \partial_i (F \circ \gamma) \cdot \partial_j \gamma_\rho \phantom{xx} \text{ and} \\
\Delta_\rho f \circ \gamma 
&= \sum_{i,j=1}^{d} g_\rho^{ij} \, \partial_i \partial_j (f \circ \gamma) + \sum_{k,l=1}^d g_\rho^{ij} \, \partial_i\big( g_\rho^{kl} \partial_l \gamma_\rho\big) \cdot \partial_j \gamma_\rho \, \partial_k(f \circ \gamma).
\end{align*}
From these formulas it is clear that $f \in h^\tau\big([0,T],h^\sigma(M,\R)\big)$ and $F \in h^\tau\big([0,T],h^\sigma(M,\R^{d+1})\big)$ for $\tau, \sigma \in \R_{\geq 0}$ with $\tau \leq r$, $\sigma \leq 2+s$ and $\sigma \geq 1$ (or even $\sigma \geq 2$ if necessary), fulfill 
\begin{align*}
\nabla_\rho f &\in h^\tau\big([0,T],h^{\sigma-1}(M,\R^{d+1})\big), \\
\Div_\rho F &\in h^\tau\big([0,T],h^{\sigma-1}(M,\R)\big) \text{ and } \\
\Delta_\rho f &\in h^\tau\big([0,T],h^{\sigma-2}(M,\R)\big).
\end{align*}
We use a similar notation to express the dependence on the height function for the mean curvature 
\begin{align*}
H(\rho) \defr H_\rho \defr -\Div_\rho \nu_\rho \in h^{r}\big([0,T],h^{s}(M)\big)
\end{align*}
and the total and normal velocity of the hypersurface
\begin{align*}
V^{\text{tot}}_\rho \defr \partial_t \theta_\rho &\in h^{r}\big([0,T],h^{s}(M,\R^{d+1})\big) \text{ and } \\
V_\rho \defr V^{\text{tot}}_\rho \cdot \nu_\rho &\in h^{r}\big([0,T],h^{s}(M)\big),
\end{align*}
respectively. 
Finally, for $f \in h^{1+\tau}\big([0,T],h^{\sigma}(M)\big) \cap h^{\tau}\big([0,T],h^{1+\sigma}(M)\big)$ with $\tau, \sigma \in \R_{\geq 0}$, $\tau \leq r$ and $\sigma \leq s$, the normal time derivative is given by
\begin{align*}
\partial^\square f = \partial_t f - V^{\text{tot}}_\rho \cdot \nabla_\rho f
\end{align*}
and thus $\partial^\square f \in h^\tau\big([0,T],h^{\sigma}(M,\R)\big)$ holds. 
\end{remark}

\subsubsection*{Reformulation onto a Fixed Domain}\label{ChapReformFixDom}

We wish to reformulate the system \eqref{eq_Intro_GLS} in a way that enables us to prove the existence of short-time solutions. For this, we assume that $\Gamma_\rho$ is an evolving immersed closed hypersurface parameterized via a height function $\rho$ as in Definition \ref{evolvHF_rho} with reference surface $M$ and global parameterization $\theta_\rho: [0,T] \times M \rightarrow \R^{d+1}$, $\theta_\rho(t,z) \defr \bar{\theta}(z) + \rho(t,z)\nu_\Sigma(z)$, where $\Sigma = \bar{\theta}(M)$ is the immersed reference surface with unit normal $\nu_\Sigma$. Our considerations here are restricted to the embedded case, but they transform easily to the immersed case using the embedded patches. We introduce the function 
\begin{align*}
u \defr c \circ \theta_\rho: [0,T] \times M \rightarrow \R
\end{align*}
to describe the pullback of the concentration. Assuming $\rho$ to be sufficiently small in an appropriate sense yields that
\begin{align*}
a(\rho) \defr \frac{1}{\nu_\Sigma \cdot \nu_\rho}
\end{align*}
is well-defined with $\frac{1}{2} \leq a(\rho) \leq C$ (see Remark \ref{a_pos}). Using the definitions and notation from Remark \ref{IMM_evolvHF_rho_bem}, the total velocity of the surface is given by $V^{\text{tot}}_\rho  = \partial_t \theta_\rho = \partial_t \rho \, \nu_\Sigma$ and we obtain 
\begin{align*}
V \circ \theta_\rho &= V_\rho = V^{\text{tot}}_\rho \cdot \nu_\rho = \partial_t \rho \, \nu_\Sigma \cdot \nu_\rho = \frac{\partial_t \rho}{a(\rho)} \phantom{x} \text{ and } \\
\big( \partial^\square c \big) \circ \theta_\rho &= \partial^\square u = \partial_t u - V^{\text{tot}}_\rho \cdot \nabla_\rho u = \partial_t u - \partial_t \rho \, \nu_\Sigma \cdot \nabla_\rho u
\end{align*}
for the normal velocity of the surface and the normal time derivative of the concentration. So, finally, the formulation of the system \eqref{eq_Intro_GLS} on the fixed domain $[0,T] \times M$ is given by
\begin{subequations}\label{eq_lokEx_gls}
  \begin{align}
    \partial_t \rho &= g(u) a(\rho) H(\rho), \label{eq_lokEx_gls1} \\
    \partial_t u  &= \Delta_\rho G'(u) + \partial_t \rho \, \nu_\Sigma \cdot \nabla_\rho u + u H(\rho) V_\rho \nonumber \\ 
    &= \Delta_\rho G'(u) + g(u) a(\rho) H(\rho) \, \nu_\Sigma \cdot \nabla_\rho u + g(u)H(\rho)^2 u. \label{eq_lokEx_gls2}
  \end{align}
\end{subequations}

\section{Short-Time Existence}

The topic of this section is the existence of short-time solutions to our system of equations \eqref{eq_lokEx_gls}.
As a start, several regularity properties of functionals are stated which will be useful throughout the whole chapter. Then, we list the conditions under which our short-time existence result holds (see Assumptions \ref{Vor}) and introduce the notations that will be used (see Notations \ref{Not2}). With this preparatory work, we can move on to the actual proof of short-time existence. As explained in the introduction, a splitting ansatz is applied: In Section \ref{ChaplokExh}, the first equation \eqref{eq_lokEx_gls1} for the height function $\rho$ is discussed. For an arbitrary concentration $u$, we obtain a unique short-time solution $\rho_u$ of this equation, which is then inserted into the second equation \eqref{eq_lokEx_gls2} for the concentration $u$. Section \ref{ChaplokExc} deals with the existence of short-time solutions to this reduced system, i.e., the second equation with inserted $\rho_u$. The combined result on short-time existence can be found in Section \ref{ChaplokExAnalyt}.

\begin{notation}\label{Not1}
Let $s \in \R_{>0} \setminus \N$ and let $\Sigma = \bar{\theta}(M) \subset \R^{d+1}$ be an $h^{2+s}$-immersed closed hypersurface. We define $X_s \defr h^s(M)$, $Y_s \defr h^{1+s}(M)$, $Z_s \defr h^{2+s}(M)$ and for constants $R^\Sigma>0$ and $R^c>0$ 
\begin{align*}
\begin{alignedat}{3}
U^h_{s,1} &\defr \big\{ \rho \in Y_s \, \big| \, \|\rho\|_{C^1(M)} < 2R^\Sigma \big\}, \phantom{xxx} 
&&U^c_s \defr \big\{ u \in Y_s \, \big| \, \|u\|_{Y_s} < 2R^c \big\}.
\end{alignedat}
\end{align*}
\end{notation}

We recall the notation and some properties for surfaces parameterized via height functions in the following lemma. 

\begin{lemma}\label{theta_h_wohldef}
Let $s \in \R_{>0} \setminus \N$ and let $\Sigma = \bar{\theta}(M) \subset \R^{d+1}$ be an $h^{2+s}$-immersed closed hypersurface with unit normal $\nu_\Sigma$. We use Notations \ref{Not1}. There exists a sufficiently small $R^\Sigma>0$ such that for all $\rho \in U^h_{s,1}$
\begin{align*}
\theta_\rho: M \rightarrow \R^{d+1}, \, \theta_\rho(z) \defr \bar{\theta}(z) + \rho(z)\nu_\Sigma(z)
\end{align*}
is an $h^{1+s}$-immersion and $\Gamma_\rho \defr \theta_\rho(M)$ is an $h^{1+s}$-immersed closed hypersurface. In particular, for any sufficiently small local parameterization $(\gamma,W)$ of $M$ and  
\begin{align*}
\gamma_\rho \defr \theta_\rho \circ \gamma,
\end{align*}
$(\gamma_\rho,W)$ is a local parameterization of an embedded patch of $\Gamma_\rho$. \\
Moreover, $\big( \partial_1 \gamma_{\rho \, |x}, ..., \partial_d \gamma_{\rho \, |x}, \nu_\Sigma \circ \gamma_{|x} \big) \subset \R^{d+1}$ are linearly independent for every $x \in \overline{W}$, where 
\begin{align*}
\partial_i \gamma_\rho = \partial_i (\bar{\theta} \circ \gamma) + \partial_i (\rho \circ \gamma) (\nu_\Sigma \circ \gamma) + (\rho \circ \gamma) \partial_i (\nu_\Sigma \circ \gamma)
\end{align*}
holds. 
\end{lemma}
\begin{proof}
On account of Proposition \ref{theta_imm}, it remains to show that 
\begin{align*}
\big( \partial_1 \gamma_{\rho \, |x}, ..., \partial_d \gamma_{\rho \, |x}, \nu_{\Sigma \, |\gamma(x)} \big) \subset \R^{d+1}
\end{align*}
are linearly independent for every $x \in \overline{W}$. For this, fix $x \in \overline{W}$ and let $\alpha_1,...,\alpha_{d+1} \in \R$ with 
\begin{align*}
0 &= \sum_{i=1}^d \alpha_i \partial_i \gamma_{\rho \, |x} + \alpha_{d+1} \nu_{\Sigma \, | \gamma(x)} \\
&= \sum_{i=1}^d \alpha_i \mathrm{d}_{\gamma(x)}\bar{\theta}[\partial_i \gamma_{|x}] + \rho_{|\gamma(x)} \sum_{i=1}^d \alpha_i \mathrm{d}_{\gamma(x)}\nu_\Sigma[\partial_i \gamma_{|x}] + \left( \sum_{i=1}^d \alpha_i \partial_i(\rho \circ \gamma)_{|x} + \alpha_{d+1} \right) \nu_{\Sigma \, |\gamma(x)}.
\end{align*}
With the statement in \eqref{eq_theta_imm_perp},
\begin{align}\label{eq_theta_imm_1}
0 = \sum_{i=1}^d \alpha_i \mathrm{d}_{\gamma(x)}\bar{\theta}[\partial_i \gamma_{|x}] + \rho_{|\gamma(x)} \sum_{i=1}^d \alpha_i \mathrm{d}_{\gamma(x)}\nu_\Sigma[\partial_i \gamma_{|x}]
\end{align}
and 
\begin{align}\label{eq_theta_imm_2}
0 = \sum_{i=1}^d \alpha_i \partial_i(\rho \circ \gamma)_{|x} + \alpha_{d+1}
\end{align}
hold independently. For $\|\rho\|_{C^0(M)}$ sufficiently small, Equation \eqref{eq_theta_imm_1} yields $\alpha_1,...,\alpha_d=0$ and then $\alpha_{d+1}=0$ follows with Equation \eqref{eq_theta_imm_2}. So, the claimed linear independency does indeed hold. 
\end{proof}

Now, we turn to the promised regularity statements. 

\begin{lemma}\label{P,Q}
Let $s \in \R_{>0} \setminus \N$ and let $\Sigma = \bar{\theta}(M) \subset \R^{d+1}$ be an $h^{3+s}$-immersed closed hypersurface. We use Notations \ref{Not1}. For $R^\Sigma > 0$ sufficiently small, there exist functions 
\begin{align*}
P \in C^\infty\Big(U^h_{s,1},\mathcal{L}\big(Z_s, X_s\big)\Big) 
\text{ and } 
Q \in C^\infty\big(U^h_{s,1},X_s\big)
\end{align*}
such that the mean curvature $H(\rho)$ of the $h^{2+s}$-immersed closed hypersurface $\Gamma_\rho = \theta_\rho(M)$ from Lemma \ref{theta_h_wohldef} is given by
\begin{align*}
H(\rho) = P(\rho)[\rho] + Q(\rho) \text{ in } X_s
\end{align*}
for all $\rho \in U^h_{s,1} \cap Z_s$.
\end{lemma}

\begin{proof}
By \cite[Lemma 3.1]{ES}, for any sufficiently small local parameterization $(\gamma,W)$ of $M$, we have
\begin{gather*}
H(\rho) \circ \gamma = P(\rho)[\rho] \circ \gamma + Q(\rho) \circ \gamma \text{ with } \\
P(\rho)[u] \circ \gamma = \frac{1}{d} \left( \sum_{i,j=1}^d p_{ij}(\rho) \partial_i \partial_j (u \circ \gamma) + \sum_{k=1}^d p_k(\rho) \partial_k (u \circ \gamma) \right) \text{ and } 
Q(\rho) \circ \gamma = \frac{1}{d} q(\rho),
\end{gather*}
where $p_{ij}, p_i, q \in C^\infty\big(U^h_{s,1}, h^s(\overline{W}) \big)$ hold for $R^\Sigma>0$ sufficiently small. Hence,
\begin{align*}
P \in C^\infty\big( U^h_{s,1}, \mathcal{L}(Z_s,X_s) \big) 
\phantom{xx} \text{ and } \phantom{xx}
Q \in C^\infty( U^h_{s,1}, X_s )
\end{align*}
follow with the help of a partition of unity. 
Note that \cite{ES} assumes $\Sigma$ to be a sphere. In \cite[Section 2.2.5]{PS}, the same statement is shown for an arbitrary embedded closed hypersurface $\Sigma$ but as the proof therein is less clearly arranged, we chose to cite \cite{ES}. Both proofs reduce the statement to local coordinates and therefore neither the shape of a sphere nor the embeddedness property are necessary. Instead, the proofs can be transferred w.l.o.g. to our setting of an immersed closed hypersurface $\Sigma$, when choosing the local parameterization $(\gamma,W)$ so small that $\theta_\rho \big(\gamma(\overline{W})\big)$ is a subset of an embedded patch of $\Sigma$ and thus $(\gamma_\rho,W)$ is a local parameterization of an embedded patch of $\Sigma$.
\end{proof}

The fact that the mean curvature $H$ has a quasilinear structure is the key argument to ensure that the PDE for the height function \eqref{eq_lokEx_gls1} is also quasilinear. Even more, its main part $P(\rho)$ is elliptic, as we will see in the upcoming lemma. 

\begin{lemma}\label{Vorfaktoren_elliptisch_P}
Let $s \in \R_{>0} \setminus \N$ and let $\Sigma = \bar{\theta}(M) \subset \R^{d+1}$ be an $h^{3+s}$-immersed closed hypersurface. We use Notations \ref{Not1} and choose $P$ as in Lemma \ref{P,Q}. For $R^\Sigma>0$ sufficiently small and $\rho \in U^h_{s,1}$, $P(\rho) \in \mathcal{L}(Z_s,X_s)$ is a symmetric and elliptic differential operator of second order, i.e., given a sufficiently small local parameterization $(\gamma,W)$ of $M$,
\begin{align*}
P(\rho)[\cdot] \circ \gamma = \sum_{i,j} a^{ij} \partial_i \partial_j(\cdot \circ \gamma) + \text{ lower order terms }
\end{align*}
holds with a symmetric and positive definite coefficient matrix $[a^{ij}]_{i,j} \in h^s(\overline{W},\R^{d \times d})$.
\end{lemma}

\begin{proof}
Let $\rho \in U^h_{s,1}$. With our sign convention, \cite[Lemma 3.1]{ES} yields
\begin{align*}
P(\rho) \circ \gamma = \frac{1}{d} \sum_{i,j} p_{ij}(\rho) \partial_i \partial_j(\cdot \circ \gamma) + \text{ lower order terms}
\end{align*}
with
\begin{align*}
p_{ij}(\rho) = \frac{w^{ij}(\rho) \left( 1 + \sum_{k,l} w^{kl}(\rho) \partial_k (\rho \circ \gamma) \partial_l (\rho \circ \gamma) \right) - \sum_{k,l} w^{ik}(\rho) w^{jl}(\rho) \partial_k (\rho \circ \gamma) \partial_l (\rho \circ \gamma)}{\left( 1 + \sum_{k,l} w^{kl}(\rho) \partial_k (\rho \circ \gamma) \partial_l (\rho \circ \gamma) \right)^{3/2}}
\end{align*}
and $w_{kl}(\rho) = g_{kl}^{\bar{\theta}} + (\rho \circ \gamma) \big( \partial_k (\nu_\Sigma \circ \gamma) \cdot \partial_l \gamma_{\bar{\theta}} + \partial_l (\nu_\Sigma \circ \gamma) \cdot \partial_k \gamma_{\bar{\theta}} \big) + (\rho \circ \gamma)^2 \big( \partial_k (\nu_\Sigma \circ \gamma) \cdot \partial_l (\nu_\Sigma \circ \gamma) \big)$ as well as $[w^{kl}(\rho)]_{k,l} = \big( [w_{kl}(\rho)]_{k,l} \big)^{-1}$. In particular, $a^{ij}_\rho \defr \frac{1}{d} p_{ij}(\rho) \in h^s(\overline{W})$ holds for all $i,j=1,...,d$. 
On account of $\rho \in U^h_{s,1}$, we have $\|\rho\|_{C^1(\Sigma)} < 2R^\Sigma$. Thus, choosing $R^\Sigma>0$ sufficiently small,
symmetry and positive definiteness of the first fundamental form $[g_{ij}^{\bar{\theta}}]_{i,j}$ and its inverse $[g^{ij}_{\bar{\theta}}]_{i,j}$ ensures the same for $[a^{ij}_\rho]_{i,j}$.
\end{proof}

We gather some further regularity statements in the following lemma. 

\begin{lemma}\label{nabla,div,a}
Let $s \in \R_{>0} \setminus \N$ and let $\Sigma = \bar{\theta}(M) \subset \R^{d+1}$ be an $h^{3+s}$-immersed closed hypersurface with unit normal $\nu_\Sigma$. We use the notation $\nabla_\rho, \Div_\rho, \Delta_\rho$ and $\nu_\rho$ as in Remark \ref{IMM_evolvHF_rho_bem} as well as Notations \ref{Not1}. For $R^\Sigma>0$ sufficiently small,
\begin{enumerate}
\item[(i)]
$\rho \mapsto \big( \nabla_\rho: f \mapsto \nabla_\rho f \big) \in C^\infty\big(U^h_{s,1},\mathcal{L}(Y_s,X_s^{d+1})\big)$ and \\ $\rho \mapsto \big( \Div_\rho: F \mapsto \Div_\rho F \big) \in C^\infty\big(U^h_{s,1},\mathcal{L}(Y_s^{d+1},X_s)\big)$ hold,
\item[(ii)]
$\rho \mapsto a(\rho) \defr \frac{1}{\nu_\rho \cdot \nu_\Sigma} \in C^\infty(U^h_{s,1},X_s)$ holds and
\item[(iii)]
there exist functions $D \in C^\infty\big(U^h_{s,1}, \mathcal{L}(Z_s,X_s)\big)$ and $J \in C^\infty\big(U^h_{1+s,1}, \mathcal{L}(Y_s,X_s)\big)$ with $J \in C^\infty_b\big( U^h_{1+s,1} \cap \mathcal{B},\mathcal{L}(Y_s, X_s)\big)$ for any bounded subset $\mathcal{B} \subset Z_s$ such that we have $\Delta_\rho u = D(\rho)[u] + J(\rho)[u]$ for all $\rho \in U^h_{1+s,1}$ and $u \in Z_s$. \\
In particular, $\rho \mapsto \big(\Delta_\rho: f \mapsto \Delta_\rho f \big) \in C^\infty\big(U^h_{1+s,1},\mathcal{L}(Z_s,X_s)\big)$ follows. 
\end{enumerate}
\end{lemma}

\begin{proof}$\phantom{.}$
\begin{enumerate}
\item[Ad (i)]
Let $f: \R^{d+1} \times \R^{d+1} \times \R^{d+1} \times \R \times \R \rightarrow \R^{d+1}, \, f(v_1,v_2,v_3,u_1,u_2) \defr v_1 + u_2 v_2 + u_1 v_3$. As $\Sigma=\bar{\theta}(M)$ is an $h^{2+s}$-immersed hypersurface, we have $\partial_j \gamma, \nu_\Sigma \circ \gamma, \partial_j(\nu_\Sigma \circ \gamma) \in h^s(\overline{W},\R^{d+1})$ for any sufficiently small local parameterization $(\gamma,W)$ of $M$. Thus, smoothness of $f$ and Corollary \ref{compositionOP_finitedim}(ii) yield
\begin{align*}
F \in C^\infty\big( h^s(\overline{W}) \times h^s(\overline{W}), h^s(\overline{W},\R^{d+1})\big) \cap C^\infty_b\big( B,h^s(\overline{W},\R^{d+1}) \big)
\end{align*}
for $F: (u_1,u_2) \mapsto \partial_j \gamma + u_2(\nu_\Sigma \circ \gamma) + u_1 \partial_j(\nu_\Sigma \circ \gamma)$ and arbitrary bounded subsets $B \subset h^s(\overline{W}) \times h^s(\overline{W})$. Additionally, $G: u \mapsto \big( u \circ \gamma, \partial_j(u \circ \gamma) \big) \in \mathcal{L}\big( Y_s,h^s(\overline{W}) \times h^s(\overline{W}) \big)$ holds and therefore we have
\begin{align*}
\rho \mapsto \partial_j \gamma_\rho = F \circ G(\rho) \in C^\infty\big( Y_s, h^s(\overline{W},\R^{d+1})\big) \cap C^\infty_b\big( \mathcal{B},h^s(\overline{W},\R^{d+1}) \big)
\end{align*}
for bounded subsets $\mathcal{B} \subset Y_s$. In particular, $\rho \mapsto g^\rho_{ij} = \partial_i \gamma_\rho \cdot \partial_j \gamma_\rho \in C^\infty\big( Y_s, h^s(\overline{W})\big)$ and $\rho \mapsto g^\rho_{ij} \in C^\infty_b\big( \mathcal{B},h^s(\overline{W})\big)$ follow. 
According to Lemma \ref{theta_h_wohldef}, for $\rho \in U^h_{s,1}$ with $R^\Sigma>0$ sufficiently small, $[g_{ij}^\rho]_{1 \leq i,j \leq d}$ is invertible on $\overline{W}$ and thus $\min_{\overline{W}} |\det [g_{ij}^\rho]| > 0$ holds. So, with the open subset $U \defr \{ A \in \R^{d \times d} \, | \, \det A \neq 0 \}$, we have $[g_{ij}^\rho] \in h^s(\overline{W},U)$ for all $\rho \in U^h_{s,1}$. Even more, as $\rho \mapsto \min_{\overline{W}} |\det [g_{ij}^\rho]|$ is continuous as mapping on $C^1(M)$, there exists $\varepsilon>0$ with $\min_{\overline{W}} |\det [g_{ij}^\rho]| \geq \varepsilon$ for all $\rho \in U^h_{s,1} \subset \{ \rho \in C^1(M) \, | \, \|\rho\|_{C^1(M)} < 2R^\Sigma \}$ with $R^\Sigma>0$ sufficiently small. For the closed subset $\mathcal{A} \defr \{ A \in \R^{d \times d} \, | \, |\det A| \geq \varepsilon \} \subset U$, we thus have $[g_{ij}^\rho] \in h^s(\overline{W},\mathcal{A})$ for all $\rho \in U^h_{s,1}$. In particular, $\rho \mapsto [g_{ij}^\rho] \in C^\infty\big( U^h_{s,1}, h^s(\overline{W},U)\big) \cap C^\infty_b\big( U^h_{s,1} \cap \mathcal{B}, h^s(\overline{W},\mathcal{A})\big)$ follows.
By Remark \ref{h_Matrixinverse}, $(\cdot)^{-1} \in C^\infty\big( h^s(\overline{W},U), h^s(\overline{W},\R^{d \times d}) \big) \cap C^\infty_b\big( \mathscr{B},h^s(\overline{W},\R^{d \times d}) \big)$ holds for the inversion $(\cdot)^{-1}$ of matrices and any bounded subset $\mathscr{B} \subset h^s(\overline{W},\mathcal{A})$.
Hence, combination implies
\begin{align*}
\rho \mapsto g^{ij}_\rho \in C^\infty\big( U^h_{s,1}, h^s(\overline{W})\big) \cap C^\infty_b\big( U^h_{s,1} \cap \mathcal{B}, h^s(\overline{W}) \big).
\end{align*}
Due to $f \mapsto \partial_i (f \circ \gamma) \in \mathcal{L}\big( Y_s, h^s(\overline{W})\big)$ and $F \mapsto \partial_i (F \circ \gamma) \in \mathcal{L}\big( Y_s^{d+1}, h^s(\overline{W},\R^{d+1})\big)$, we finally have
\begin{align*}
(\rho,f) &\mapsto \nabla_\rho f \circ \gamma = \sum_{i,j} g_\rho^{ij} \partial_i(f  \circ \gamma) \partial_j \gamma_\rho \in C^\infty\big(U^h_{s,1}, \mathcal{L}(Y_s,h^s(\overline{W},\R^{d+1}))\big) \text{ and } \\
(\rho,F) & \mapsto \Div_\rho F \circ \gamma = \sum_{i,j} g_\rho^{ij} \partial_i(F  \circ \gamma) \cdot  \partial_j \gamma_\rho \in C^\infty\big(U^h_{s,1}, \mathcal{L}(Y_s^{d+1},h^s(\overline{W}))\big).
\end{align*} 
\item[Ad (iii)]
For any sufficiently small local parameterization $(\gamma,W)$ of $M$, we have 
\begin{align*}
\Delta_\rho f \circ \gamma = \sum_{i,j} g_\rho^{ij} \, \partial_i \partial_j (f \circ \gamma) + \sum_{i,j,k,l} g_\rho^{ij} \partial_i \big( g_\rho^{kl} \partial_l \gamma_\rho \big) \cdot \partial_j \gamma_\rho \, \partial_k (f \circ \gamma)
\end{align*}
by Remark \ref{IMM_evolvHF_rho_bem}. We choose $D$ as the principal part of $\Delta$ and define $J \defr \Delta - D$ such that   
\begin{align*}
D(\rho)[f] \circ \gamma &= \sum_{i,j} g_\rho^{ij} \, \partial_i \partial_j (f \circ \gamma) \phantom{xxx} \text{ and } \\
J(\rho)[f] \circ \gamma &= \sum_{i,j,k,l} g_\rho^{ij} \partial_i \big( g_\rho^{kl} \partial_l \gamma_\rho \big) \cdot \partial_j \gamma_\rho \, \partial_k (f \circ \gamma)
\end{align*}
hold on $\overline{W}$. With the help of a partition of unity, 
$D(\rho)[f]$ and $J(\rho)[f]$ are well-defined on the whole hypersurface $M$. As in (i), we have $\rho \mapsto g_\rho^{ij} \in C^\infty\big(U^h_{s,1}, h^s(\overline{W})\big)$ and on account of $f \mapsto \partial_i \partial_j (f \circ \gamma) \in \mathcal{L}\big(Z_s,h^s(\overline{W})\big)$ 
\begin{align*}
(\rho, f) \mapsto D(\rho)[f] \circ \gamma \in C^\infty\big( U^h_{s,1}, \mathcal{L}(Z_s,h^s(\overline{W}))\big)
\end{align*}
follows. 
But we only needed $\Sigma$ to be an $h^{2+s}$-immersed hypersurface for the proof of (i). Thus, also
\begin{align*}
\rho \mapsto \partial_j \gamma_\rho &\in C^\infty\big(Z_s, h^{1+s}(\overline{W},\R^{d+1})\big) \cap C^\infty_b\big( \mathcal{B}, h^{1+s}(\overline{W},\R^{d+1})\big) \text{ and } \\
\rho \mapsto g^{ij}_\rho &\in C^\infty\big(U^h_{1+s,1},h^{1+s}(\overline{W})\big) \cap C^\infty_b\big( U^h_{1+s,1} \cap \mathcal{B},h^{1+s}(\overline{W})\big)
\end{align*}
hold for bounded subsets $\mathcal{B} \subset Z_s$ with $h^{1+s}(\overline{W},\R^m) \hookrightarrow h^s(\overline{W},\R^m)$ for $m \in \{1,d+1\}$ due to Lemma \ref{HS_Einbettung}. Differentiating once 
yields
\begin{align*}
\rho \mapsto \partial_i \partial_l \gamma_\rho &\in C^\infty\big(Z_s, h^{s}(\overline{W},\R^{d+1})\big) \cap C^\infty_b\big( \mathcal{B},h^s(\overline{W},\R^{d+1})\big) \text{ and } \\
\rho \mapsto \partial_i g^{kl}_\rho &\in C^\infty\big(U^h_{1+s,1},h^{s}(\overline{W})\big) \cap C^\infty_b\big( U^h_{1+s,1} \cap \mathcal{B},h^s(\overline{W})\big).
\end{align*}
Due to $f \mapsto \partial_k(f \circ \gamma) \in \mathcal{L}\big(Y_s,h^s(\overline{W})\big)$, we hence have
\begin{align*}
(\rho, f) \mapsto J(\rho)[f] \circ \gamma &\in C^\infty\big( U^h_{1+s,1}, \mathcal{L}(Y_s, h^s(\overline{W}))\big) \cap C^\infty_b\big( U^h_{1+s,1} \cap \mathcal{B},\mathcal{L}(Y_s, h^s(\overline{W}))\big).
\end{align*}
\item[Ad (ii)]
Let $\mathcal{K}: (\R^{d+1})^d \rightarrow \R^{d+1}$ be a generalized cross product; 
in particular $\mathcal{K}$ is smooth. For the open subset
\begin{align*}
U \defr \big\{ (v_1,...,v_{d+1}) \in (\R^{d+1})^{d+1} \, \big| \, (v_1,...,v_{d+1}) \subset \R^{d+1} \text{ linearly independent} \big\},
\end{align*}
the map $f: U \rightarrow \R$ with
\begin{align*}
f(v_1,...,v_{d+1}) \defr \frac{|\mathcal{K}(v_1,...,v_d)|}{\mathcal{K}(v_1,...,v_d) \cdot v_{d+1}}
\end{align*}
is well-defined and also smooth. So, by Corollary \ref{compositionOP_finitedim}(ii), $F \in C^\infty\big( h^s(\overline{W},U),h^s(\overline{W})\big)$ holds with $\big(F(u)\big)(x) \defr f\big(u(x)\big)$ for $u: \overline{W} \rightarrow U$. 
As in the proof of (i), we have $\rho \mapsto \partial_j \gamma_\rho \in C^\infty\big(Y_s,h^s(\overline{W},\R^{d+1})\big)$ for any sufficiently small local parameterization $(\gamma,W)$ of $M$ and thus $G: \rho \mapsto (\partial_1 \gamma_\rho,...,\partial_d \gamma_\rho,\nu_\Sigma \circ \gamma) \in C^\infty\big(Y_s,h^s(\overline{W},(\R^{d+1})^{d+1}) \big)$. Due to Lemma \ref{theta_h_wohldef}, $\big( \partial_1 \gamma_{\rho \, |x}, ..., \partial_d \gamma_{\rho \, |x}, \nu_\Sigma \circ \gamma_{|x} \big) \subset \R^{d+1}$ are linearly independent for every $x \in \overline{W}$ if $\rho \in U^h_{s,1}$ with $R^\Sigma>0$ sufficiently small. Therefore, $G \in C^\infty\big(U^h_{s,1},h^s(\overline{W},U) \big)$ follows. Composition yields 
\begin{align*}
\big( F \circ G \big)(\rho) = \frac{|K(\partial_1 \gamma_\rho, ..., \partial_d \gamma_\rho)|}{K(\partial_1 \gamma_\rho, ..., \partial_d \gamma_\rho) \cdot (\nu_\Sigma \circ \gamma)} = \frac{1}{(\nu_\rho \circ \gamma) \cdot (\nu_\Sigma \circ \gamma)} = a(\rho) \circ \gamma
\end{align*}
and hence $\rho \mapsto a(\rho) \circ \gamma \in C^\infty\big( U^h_{s,1}, h^s(\overline{W}) \big)$. 
\qedhere
\end{enumerate}
\end{proof}

\begin{remark}\label{a_pos}
Let $s \in \R_{>0} \setminus \N$, let $\Sigma = \bar{\theta}(M)$ be an $h^{3+s}$-immersed closed hypersurface and let $\alpha \in (0,1)$ with $\alpha \leq s$. We use Notations \ref{Not1}. Due to the smoothness of $a: U^h_{\alpha,1} \rightarrow X_{\alpha}$, $a(\rho) \defr \frac{1}{\nu_\rho \cdot \nu_\Sigma}$ by Lemma \ref{nabla,div,a}(ii) and $a(0) = \frac{1}{|\nu_\Sigma|^2}=1$ for $0 \in U^h_{\alpha,1}$, we can choose $R^\Sigma>0$ sufficiently small such that $a \geq \frac{1}{2}$ holds on $\big\{ \rho \in Y_{\alpha} \, \big| \, \|\rho\|_{Y_{\alpha}} < 2R^\Sigma \big\}$. In particular, we thus have $a \geq \frac{1}{2}$ on the set $U^h_s$ defined in Notations \ref{Not2}. Analogously, there exists a constant $C>0$ such that $\|a(\rho)\|_{X_\alpha} \leq C$ holds for all $\rho \in U^h_s$. 
\end{remark}

The pullback $\Delta_\rho$ of the Laplace-Betrami operator obviously is a linear operator, so that the PDE for the concentration \eqref{eq_lokEx_gls2} is quasilinear. Its parabolicity relies mainly on the fact that $\Delta_\rho$ is an elliptic operator, as we state in the next remark.  

\begin{remark}\label{Vorfaktoren_elliptisch_Laplace}
Let $s \in \R_{>0} \setminus \N$ and let $\Sigma = \bar{\theta}(M) \subset \R^{d+1}$ be an $h^{3+s}$-immersed closed hypersurface. We use the notation $\Delta_\rho$ as in Remark \ref{IMM_evolvHF_rho_bem} as well as Notations \ref{Not1}. For $R^\Sigma>0$ sufficiently small and $\rho \in U^h_{1+s,1}$, the pullback $\Delta_\rho \in \mathcal{L}(Z_s,X_s)$ of the Laplace-Beltrami operator is a symmetric and elliptic differential operator of second order, as for a sufficiently small local parameterization $(\gamma,W)$ of $M$, 
\begin{align*}
\Delta_\rho[\cdot] \circ \gamma = \sum_{i,j} g_\rho^{ij} \partial_i \partial_j(\cdot \circ \gamma) + \text{ lower order terms }
\end{align*}
holds with the symmetric and positive definite matrix $[g_\rho^{ij}]_{i,j} \in h^s(\overline{W},\R^{d \times d})$.
\end{remark}

We end the collection of regularity statements by a simple consequence of Section \ref{Sec_CompOp} on the regularity of composition operators that will be applied to the functions $G$ and $g$ later on. 

\begin{lemma}\label{g}
Let $s \in \R_{>0} \setminus \N$ and let $\Sigma = \bar{\theta}(M) \subset \R^{d+1}$ be an $h^{1+s}$-immersed closed hypersurface. We use Notations \ref{Not1}. If $F \in C^{k+\lfloor s \rfloor+2}(\R)$, we have
\begin{align*}
u \mapsto F(u) \, \in \, C^k(X_s,X_s)
\end{align*}
 and in particular, $u \mapsto F(u) \in C^k(U^c_s,X_s)$. 
\end{lemma}

\begin{proof}
Let $(\gamma,W)$ be any sufficiently small local parameterization of $M$ and let $R>0$. Due to $F \in C^{k+\lfloor s \rfloor+2}_b\big((-R,R)\big)$, Proposition \ref{compositionOP_C0}(iii) yields $F \in C^k\big( h^s(\overline{W},(-R,R)), h^s(\overline{W})\big)$. As $R>0$ was arbitrary, $F \in C^k\big( h^s(\overline{W}),h^s(\overline{W}) \big)$ holds. With $u \mapsto u \circ \gamma \in \mathcal{L}\big(X_s,h^s(\overline{W}) \big)$ 
the claim follows. 
\end{proof}

Having gathered these general regularity statements, we proceed to the more specific setting in which we will prove the existence of short-time solutions. First, we list the assumptions needed for our proof. 

\begin{ass}\label{Vor}$\phantom{.}$
\begin{enumerate}
\item[(i)]
Let $\alpha \in (0,1)$ and $\beta \in (0,\frac{1}{2})$ with $2\beta + \alpha \notin \N$. 
Furthermore, let $G \in C^7(\R)$ with $G''>0$ and $g \defr G - G' \cdot \Id > 0$. 
\item[(ii)]
Let $\Sigma = \bar{\theta}(M) \subset \R^{d+1}$ be an $h^{4+\alpha}$-immersed closed hypersurface with unit normal $\nu_\Sigma$ and let $R^\Sigma>0$ be sufficiently small. 
\item[(iii)]
Let $u_0 \in h^{2+2\beta+\alpha}(M)$ and let $\delta_1>R^\Sigma$ be arbitrary. 
\item[(iv)]
Let $R^c,R^h>0$ be sufficiently large such that $2\|u_0\|_{h^{2+\alpha}(M)} \leq R^c$ and $2\delta_1 \leq R^h$ holds. Let $\delta_0 \in (0,R^\Sigma)$. Then, let $T \in (0,1)$ be sufficiently small such that 
\begin{align}\label{Absch_T}
R^h T^\beta + \delta_0 < R^\Sigma
\end{align}
is valid. 
\end{enumerate}
\end{ass}

We give a few comments on these assumptions and explain why they are postulated by refering to later statements. So, these comments will not be understandable in detail for the reader yet, but serve as a later look-up. 
Choosing $\beta < \frac{1}{2}$ ensures that the embedding $Z_\alpha \hookrightarrow Y_{2\beta+\alpha}$ is compact and thus $K^c_{2\beta+\alpha}$, $K^h_{2\beta+\alpha}$ as in Definition \ref{Not2}(i) are compact sets in $Y_{2\beta+\alpha}$. 
Assuming the immersed hypersurface $\Sigma$ to be of $h^{4+\alpha}$-regularity guarantees that we can apply Lemmas \ref{P,Q} and \ref{nabla,div,a} for $s \defr 2\beta+\alpha$. Together with the $C^7$-regularity of $G$, this is used in Corollary \ref{Vorfaktoren_C2} to gain regularity properties for our operators. The conditions $G''>0$ and $g>0$ ensure that our PDEs are parabolic. 
The $h^{2+2\beta+\alpha}$-regularity, which we assume for the initial value $u_0$ of the concentration, as well as for the initial value $\rho_0$ of the height function later on, makes sure that by applying our second order operators, we still end up with an $h^{2\beta+\alpha}$-regularity. This turns out to be the necessary compability condition and is used in Lemmas \ref{G_wohldef_h} and \ref{G_wohldef_c}. \\
We will obtain a short-time existence result for any initial height function $\rho_0 \in h^{2+2\beta+\alpha}(M)$ with $\|\rho_0\|_{h^{2+2\beta+\alpha}(M)} < \delta_1$ and $\|\rho_0\|_{h^{1+\alpha}(M)} < \delta_0$. Particularly, $2\|\rho_0\|_{h^{2+\alpha}(M)} < 2\delta_1 \leq R^h$ follows. As $\delta_1>0$ can be chosen arbitrarily large, $\|\rho_0\|_{h^{2+2\beta+\alpha}(M)} < \delta_1$ is not an actual restriction on $\rho_0$. To yield a suitable height function as in Lemma \ref{theta_h_wohldef}, the initial value $\rho_0$ only needs to be small in the $C^1$-norm. But to achieve $a(\rho_0)>0$ with Remark \ref{a_pos}, and also later on in the proofs of Theorem \ref{lokEx_h} and Proposition \ref{Kontraktion_c}, smallness of $\rho_0$ in the $h^{1+\alpha}$-norm is necessary. This is why we set the condition $\|\rho_0\|_{h^{1+\alpha}(M)} < \delta_0$. \\
Assuming $R^\Sigma>0$ sufficiently small means that Lemmas \ref{theta_h_wohldef}, \ref{P,Q}, \ref{Vorfaktoren_elliptisch_P} and \ref{nabla,div,a} as well as Remarks \ref{a_pos} and \ref{Vorfaktoren_elliptisch_Laplace} hold. In particular, this implies that any function $\|\rho_t\| < R^\Sigma$ is a well-defined height function as in Lemma \ref{theta_h_wohldef} and the regularity statements in terms of $\rho_t$ hold for all the geometric quantities from Lemmas \ref{P,Q} and \ref{nabla,div,a}. \\
In the following, we will choose $R^c$ and $R^h$ even larger and $\delta_0>0$ and $T>0$ even smaller, where $T$ always has to be so small that estimate \eqref{Absch_T} holds. Enlarging $R^c$ and $R^h$ increases the set of possible solutions to our system of PDEs, which we seek in balls with radii $R^c$ and $R^h$. Then, estimate \eqref{Absch_T} together with the Hölder-regularity of the solution guarantees that every $\|\rho\| \leq R^h$ with initial value $\|\rho(0)\| < \delta_0$ fulfills $\|\rho(t)\| < R^\Sigma$. Particularly, $\rho(t)$ remains small in the $h^{1+\alpha}$-norm for all times $t \in [0,T]$ such that all the properties mentioned above hold; most importantly, $\rho(t)$ is a well-defined height function as in Lemma \ref{theta_h_wohldef}  for every $t \in [0,T]$.\\

Now, we give a summary of the notation used in the following sections. It relies on Notations \ref{Not1}, but is reduced to our more specific setting.

\begin{notation}\label{Not2}$\phantom{.}$
Suppose Assumptions \ref{Vor} are valid and let $s \in \{ \alpha, 2\beta+\alpha \}$. This guarantees that $\Sigma = \bar{\theta}(M)$ is an $h^{3+s}$-immersed closed hypersurface and thus permits to use Notations \ref{Not1}: $X_s \defr h^s(M)$, $Y_s \defr h^{1+s}(M)$ and $Z_s \defr h^{2+s}(M)$. We also recall
\begin{align*}
\begin{alignedat}{3}
U^h_{s,1} &\defr \big\{ \rho \in Y_s \, \big| \, \|\rho\|_{C^1(M)} < 2R^\Sigma \big\}, \phantom{x}
&&\phantom{xx}U^c_s &&\defr \big\{ u \in Y_s \, \big| \, \|u\|_{Y_s} < 2R^c \big\}, \\
U^h_{1+s,1} &\defr \big\{ \rho \in Z_s \, \big| \, \|\rho\|_{C^1(M)} < 2R^\Sigma \big\} && &&
\end{alignedat}
\end{align*}
and define
\begin{align*}
U^h_s \defr \big\{ \rho \in Y_s \, \big| \, \|\rho\|_{Y_s} < 2R^h, \, \|\rho\|_{Y_\alpha} < 2R^\Sigma \big\}.
\end{align*}
\begin{enumerate}
\item[(i)]
Furthermore, we define 
\begin{align*}
K^h_s \defr \overline{ \big\{ \rho \in Z_\alpha \, \big| \, \|\rho\|_{Z_\alpha} \leq R^h, \, \|\rho\|_{Y_\alpha} \leq R^\Sigma \big\} }^{\|\cdot\|_{Y_s}}, \phantom{xxx}
K^c_s \defr \overline{ \big\{ u \in Z_\alpha \, \big| \, \|u\|_{Z_\alpha} \leq R^c \big\}}^{\|\cdot\|_{Y_s}}.
\end{align*}
\item[(ii)]
We use the following notation for spaces and sets with time-dependence
\begin{align*}
\E_{0,T} &\defr h^\beta([0,T],X_\alpha), \\
\E_{1,T} &\defr h^{1,\beta}([0,T],X_\alpha) \cap h^\beta([0,T],Z_\alpha), \\
M_T^c &\defr \{ u \in \E_{1,T} \, | \, \|u\|_{\E_{1,T}} \leq R^c \text{ and } u(0)=u_0 \text{ in } Z_\alpha\}, \\
M_T^h &\defr \big\{ \rho \in \E_{1,T} \, \big| \, \|\rho\|_{\E_{1,T}} \leq R^h \text{ and } \|\rho(t)\|_{Y_\alpha} \leq R^\Sigma \text{ for all } t \in [0,T] \big\} \text{ and } \\
M_{T,\rho_0}^h &\defr \{ \rho \in \E_{1,T} \, | \, \|\rho\|_{\E_{1,T}} \leq R^h \text{ and } \rho(0)=\rho_0 \text{ in } Z_\alpha\} \\
&\phantom{xxx} \text{ for any } \rho_0 \in Z_{2\beta+\alpha} \text{ with } \|\rho_0\|_{Z_\alpha} \leq R^h \text{ and } \|\rho_0\|_{Y_\alpha} < \delta_0. 
\end{align*}
\item[(iii)]
For the sake of completeness, we also define the operators used in the following sections. For $u,\rho \in \E_{1,T}$ and $u_1, \rho_1 \in Z_\alpha$, we set
\begin{align*}
\begin{alignedat}{2}
A^h_{u_1,\rho_1}[\rho] &\defr g(u_1)a(\rho_1)P(\rho_1)[\rho], \\
A^h[\rho] &\defr A_{u_0,0}^h[\rho] = g(u_0) a(0) P(0)[\rho], \\
G_u^h(\rho) &\defr g(u) a(\rho) H(\rho) - A^h[\rho], \\
L^h[\rho] &\defr \binom{\partial_t \rho - A^h[\rho]}{\rho(0)}, \\
A^c_{u_1,\rho_1}[u] &\defr G''(u_1) \Delta_{\rho_1} u + g(u_1) a(\rho_1) H(\rho_1) \nu_\Sigma {\cdot} \nabla_{\rho_1} u + g(u_1) H(\rho_1)^2 u, \\
A^c[u] &\defr A_{u_0,0}^c[u] = G''(u_0) \Delta_\Sigma u + g(u_0)H_\Sigma^2u, \\
G^c_{\rho_0}(u) &\defr \Delta_{\rho_{u,\rho_0}} G'(u) + g(u) a(\rho_{u,\rho_0}) H(\rho_{u,\rho_0}) \nu_\Sigma {\cdot} \nabla_{\rho_{u,\rho_0}} u + g(u) H(\rho_{u,\rho_0})^2 u - A^c[u], \\
L^c[u] &\defr \binom{\partial_t u - A^c[u]}{u(0)}.
\end{alignedat} 
\end{align*}
Here, $H,P,Q$ are the functionals from Lemma \ref{P,Q}. Moreover, we have $a(\rho) \defr \frac{1}{\nu_\rho \cdot \nu_\Sigma}$ as in Lemma \ref{nabla,div,a}, where $\nu_\rho$ as well as the differential operators  $\nabla_\rho$, $\Delta_\rho$ were introduced in Remark \ref{IMM_evolvHF_rho_bem}; in particular, $\nu_0=\nu_\Sigma$, $\nabla_0=\nabla_\Sigma$, $\Delta_0=\Delta_\Sigma$ and $H(0)=H_\Sigma$ hold in the case of $\rho=0$. Finally, $\rho_{u,\rho_0} \in M_{T,\rho_0}^h$ is the solution from Theorem \ref{lokEx_h} associated with the concentration $u \in M_T^c$ and the initial value $\rho_0 \in Z_{2\beta+\alpha}$ with $\|\rho_0\|_{Z_{2\beta+\alpha}} < \delta_1$ and $\|\rho_0\|_{Y_\alpha} < \delta_0$.
\end{enumerate}
\end{notation}

Both our PDEs \eqref{eq_lokEx_gls} are parabolic, quasilinear equations of second order (see Lemmas \ref{P,Q}, \ref{Vorfaktoren_elliptisch_P} and \ref{nabla,div,a} as well as Remarks \ref{a_pos} and \ref{Vorfaktoren_elliptisch_Laplace}) and will be solved by similar approaches. To underline this parallel structure, we use the same notation for all corresponding sets and operators and mark the association to the respective equation with a superscript, using the letter $h$ for the first equation \eqref{eq_lokEx_gls1} concerning the height function and the letter $c$ for the second equation  \eqref{eq_lokEx_gls2} concerning the concentration function. Dependences of sets or operators will never be denoted by superscripts, but only by indices. To clarify this even more, we use the letters $h$ and $c$ only to denote the association to the equation; while height functions and concentrations will always be called $\rho$ and $u$, respectively. \\
Whereas the initial value $u_0$ for the concentration is chosen fixed in Assumptions \ref{Vor}, our short-time existence result allows for small variations in the initial value $\rho_0$ of the height function. More precisely, for any initial value $\rho_0 \in h^{2+2\beta+\alpha}(M)$ with $\|\rho_0\|_{h^{2+2\beta+\alpha}(M)} < \delta_1$ and $\|\rho_0\|_{h^{1+\alpha}(M)} < \delta_0$, we will obtain a solution to \eqref{eq_lokEx_gls} on a time interval $[0,T]$ with $T$ independent of $\rho_0$. This is crucial to prove the formation of self-intersections, which will be done in an upcoming publication. 
However, we thus can not linearize the system \eqref{eq_lokEx_gls} around the initial value for the height function, as we do for the concentration. Instead, we linearize around the fixed value $0$. This is possible, as due to $\|\rho_0\|_{h^{1+\alpha}(M)} < \delta_0$ all eligible initial values $\rho_0$ are close to the zero-function in a suitable sense. \\
We will solve our PDEs in the space $\E_{0,T}$ and therefore the solution functions lie in $\E_{1,T}$. To be precise, we seek the solution functions in $M_T^c$ and $M_{T}^h$, which are the balls with radii $R^c$ and $R^h$ mentioned earlier. As forecasted, estimate \eqref{Absch_T} guarantees that any $\rho \in M_{T,\rho_0}^h$ fulfills
\begin{align*}
\|\rho(t)\|_{Y_\alpha}
&\leq \|\rho(t)-\rho(0)\|_{Y_\alpha} + \|\rho(0)\|_{Y_\alpha} \\
&\leq \|\rho\|_{h^\beta([0,T],Y_\alpha)}T^\beta + \|\rho_0\|_{Y_\alpha} 
< R^hT^\beta + \delta_0 
< R^\Sigma,
\end{align*}
i.e., $M_{T,\rho_0}^h \subset M_T^h$ holds. In particular, $\rho(t)$ is a well-defined height function as in Lemma \ref{theta_h_wohldef} for every $\rho \in M_T^h$ and $t \in [0,T]$. \\
Now, we give a few crucial comments on embeddings of the sets defined above.  Here, the superscripts $c$ and $h$ are omitted whenever the corresponding statement holds for both of them. As always, we set $s \in \{\alpha, 2\beta+\alpha\}$. 
\begin{enumerate}
\item[(i)]
We have $K_s \subset U_s$ due to $Z_\alpha \hookrightarrow Y_s$. Moreover, $K_s \subset Y_s$ is compact and convex, because $Z_\alpha \hookrightarrow Y_s$ is a compact embedding due to Lemma \ref{kompHölder} with $\beta < \frac{1}{2}$. Obviously, $U_s \subset Y_s$ is open. 
\item[(ii)]
As $u \in M_T^c$, $\rho \in M_T^h$ fulfill $\|u(t)\|_{Z_\alpha} \leq R^c$, $\|\rho(t)\|_{Z_\alpha} \leq R^h$ and $\|\rho(t)\|_{Y_\alpha} \leq R^\Sigma$ for every $t \in [0,T]$, the inclusions
\begin{align*}
M_T^h &\subset h^\beta\big([0,T],K^h_s\big) \cap h^\beta\big([0,T],U^h_{1+\alpha,1} \big) \subset h^\beta\big([0,T],U^h_s\big), \\
M_T^c &\subset h^\beta\big([0,T],K^c_s\big) \subset h^\beta\big([0,T],U^c_s\big)
\end{align*}
follow. Furthermore, for any $u \in M_T^c$ and any $\rho \in M_T^h$,
\begin{align*}
\|\rho\|_{h^\beta([0,T],Y_s)} &\leq \|\rho\|_{h^\beta([0,T],Z_\alpha)} \leq \|\rho\|_{\E_{1,T}} \leq R^h, \\
\|u\|_{h^\beta([0,T],Y_s)} &\leq \|u\|_{h^\beta([0,T],Z_\alpha)} \leq \|u\|_{\E_{1,T}} \leq R^c
\end{align*}
hold on account of $Z_\alpha \hookrightarrow Y_s$.
\end{enumerate}

In the next corollary, we state some regularity properties for the components of the operators from Notations \ref{Not2}(iii). This corollary, and even more so the subsequent remark, are crucial for the proof of short-time existence as many of the following statements are based on this regularity.

\begin{corollary}\label{Vorfaktoren_C2}
We suppose Assumptions \ref{Vor} are valid and use Notations \ref{Not2}. For $s \in \{\alpha,2\beta+\alpha\}$, 
\begin{align*}
g a P &\in C^2\big(U^c_s \times U^h_s,\mathcal{L}(Z_s,X_s)\big), \\
g a Q &\in C^2(U^c_s \times U^h_s,X_s), \\
G'' D &\in C^2\big( U^c_s \times U^h_s, \mathcal{L}(Z_s,X_s) \big), \\
G''J &\in C^2\big( U^c_s \times U^h_{1+s,1}, \mathcal{L}(Y_s,X_s) \big), \\
g a P \nu_\Sigma \cdot \nabla &\in C^2\big( U^c_s \times U^h_s, \mathcal{L}(Z_s,\mathcal{L}(Y_s,X_s))\big), \\
g a Q \nu_\Sigma \cdot \nabla &\in C^2\big( U^c_s \times U^h_s, \mathcal{L}(Y_s,X_s)\big), \\
g P^2 \Id^{\,c} &\in C^2\big( U^c_s \times U^h_s, \mathcal{L}(Z_s, \mathcal{L}(Z_s, \mathcal{L}(X_s,X_s)))\big), \\
g P Q \Id^{\,c} &\in C^2\big( U^c_s \times U^h_s, \mathcal{L}\big(Z_s, \mathcal{L}(X_s,X_s))\big) \text{ and} \\
g Q^2 \Id^{\,c} &\in C^2\big( U^c_s \times U^h_s, \mathcal{L}(X_s,X_s)\big),
\end{align*}
hold. Furthermore, for $G''' \nabla_{(\cdot)} \cdot \nabla_{(\cdot)}: (u,\rho) \mapsto G'''(u) \nabla_\rho \cdot \nabla_\rho$, we have 
\begin{align*}
G''' \, \nabla_{(\cdot)} \cdot \nabla_{(\cdot)} &\in C^2 \big( U^c_\alpha \times U^h_\alpha, \mathcal{L}(Y_\alpha,\mathcal{L}(Y_\alpha,X_\alpha))\big),\\
G''' \, \nabla_{(\cdot)} \cdot \nabla_{(\cdot)} &\in C^1 \big( U^c_{2\beta+\alpha} \times U^h_{2\beta+\alpha}, \mathcal{L}(Y_{2\beta+\alpha},\mathcal{L}(Y_{2\beta+\alpha},X_{2\beta+\alpha}))\big).
\end{align*}
\end{corollary}

\begin{proof}
We have $2\beta + \alpha \in (0,2) \setminus \{1\}$ with $2\beta < 1$. Because $\Sigma=\bar{\theta}(M)$ is an $h^{4+\alpha}$-immersed closed hypersurface and $G \in C^7(\R)$, we can choose $s \in \{\alpha, 2\beta + \alpha\}$ and mostly $k \defr 2$ in Lemmas \ref{P,Q}, \ref{nabla,div,a} and \ref{g}. (If $s>1$, we have to restrict to $k \defr 1$ in Lemma \ref{g} for $G'''$ and thus can only conclude $C^1$-differentiability for $G'''$. The $C^7$-regularity of $G$ is used to obtain $G''' \in C^2(U^c_\alpha, X_\alpha)$ with Lemma \ref{g}.) Moreover, the inclusion $U^h_s \subset U^h_{s,1}$ holds. By considering functions independent of $c$ or $h$ as constant in these variables, several multiplications 
and $Z_s \hookrightarrow Y_s \hookrightarrow X_s$ prove the claims. 
\end{proof}

\begin{remark}\label{compositionOP_Zus}
As $K^c_s \times K^h_s \subset U^c_s \times U^h_s$ is compact and convex, we can apply Corollary \ref{compositionOP_compact} in the following way: For any Banach space $W_s$ and any functional $F \in C^2\big( U^c_s \times U^h_s, W_s \big)$, there exists a constant $C=C(R^\Sigma, R^c, R^h)$ such that $F(u_1,\rho_1) \in h^\beta\big([0,T],W_s\big)$ holds with 
\begin{align*}
\big\| F(u_1,\rho_1) \big\|_{h^\beta([0,T],W_s)} &\leq C \phantom{xxx} \text{ and} \\
\big\| F(u_1,\rho_2)-F(u_2,\rho_2) \big\|_{h^\beta([0,T],W_s)} &\leq C \big( \|u_1-u_2\|_{h^\beta([0,T],Y_s)} + \|\rho_1-\rho_2\|_{h^\beta([0,T],Y_s)} \big)
\end{align*}
for $u_i \in h^\beta\big([0,T],K^c_s\big)$, $\rho_i \in h^\beta\big([0,T],K^h_s\big)$ with $\|u_i\|_{h^\beta([0,T],Y_s)} \leq R^c$, $\|\rho_i\|_{h^\beta([0,T],Y_s)} \leq R^h$. In particular, these conditions are fulfilled for $u_i \in M_T^c$ and $\rho_i \in M_T^h$. \\
Except for $G''J$, all of the functionals listed in Corollary \ref{Vorfaktoren_C2} can be estimated in this way. Because $G''' \nabla_{(\cdot)} \cdot \nabla_{(\cdot)}$ is a $C^1$-function only for $s=2\beta+\alpha$, Corollary \ref{compositionOP_compact} only yields the first of the two estimates stated above in that case. But if we restrict to $u_1 = u_2$, the second estimate also holds: As in Corollary \ref{Vorfaktoren_C2}, we have $G''' \in C^1(U^c_s,X_s)$ and $\nabla_{(\cdot)} \cdot \nabla_{(\cdot)}: \rho \mapsto \nabla_\rho \cdot \nabla_\rho \in C^2(U^h_s,W_s)$ with $W_s \defr \mathcal{L}(Y_s, \mathcal{L}(Y_s,X_s))$. Thus, Corollary \ref{compositionOP_compact} yields the existence of a constant $C=C(R^\Sigma,R^c,R^h)$ with
\begin{align*}
&\big\| G'''(u) \nabla_{\rho_1} \cdot \nabla_{\rho_1} - G'''(u) \nabla_{\rho_2} \cdot \nabla_{\rho_2} \big\|_{h^\beta([0,T],W_s)} \\
&\leq \| G'''(u) \|_{h^\beta([0,T],X_s)} \big\| \nabla_{\rho_1} \cdot \nabla_{\rho_1} - \nabla_{\rho_2} \cdot \nabla_{\rho_2} \big\|_{h^\beta([0,T],W_s)} \\
&\leq C \|\rho_1-\rho_2\|_{h^\beta([0,T],Y_s)} 
\end{align*}
for all $u \in M_T^c$ and $\rho_i \in M_T^h$. \\
Due to $U^h_{1+s,1} \subset Z_s$ and $M_T^h \subset \E_{1,T} = h^{1+\beta}([0,T],X_\alpha) \cap h^\beta([0,T],Z_\alpha)$, we can not find a compact set $K \subset U^h_{1+s,1}$ with $M_T^h \subset h^\beta\big([0,T],K\big)$. Therefore, the functional $G''J$, which is defined on $U^c_s \times U^h_{1+s,1}$, has to be handled differently. But as $J: \rho \mapsto J(\rho)$ is bounded on bounded sets by Lemma \ref{nabla,div,a}(iii), we have $J: \rho \mapsto J(\rho) \in C^2_b\big( U^h_{1+s,1} \cap \overline{B_R^{Z_s}(0)}, \mathcal{L}(Y_s,X_s)\big)$ for any $R>0$. As $U^h_{1+s,1} \cap \overline{B_R^{Z_s}(0)} \subset Z_s$ is convex, we can apply Proposition \ref{compositionOP_C0} instead of Corollary \ref{compositionOP_compact} to $G''J$. 
With $W_s \defr \mathcal{L}(Y_s,X_s)$, this means that there exists a constant $C=C(R^\Sigma,R^c,R)$ such that $G''J(u_1,\rho_1) \in h^\beta\big([0,T],W_s\big)$ holds with
\begin{align*}
\big\| G''J(u_1,\rho_1) \big\|_{h^\beta([0,T],W_s)} &\leq C \phantom{xxx} \text{ and} \\
\big\| G''J(u_1,\rho_2)-G''J(u_2,\rho_2) \big\|_{h^\beta([0,T],W_s)} &\leq C \big( \|u_1-u_2\|_{h^\beta([0,T],Y_s)} + \|\rho_1-\rho_2\|_{h^\beta([0,T],Z_s)} \big)
\end{align*}
for $u_i \in h^\beta\big([0,T],K^c_s\big)$, $\rho_i \in h^\beta\big([0,T],U^h_{1+s,1}\big)$ with $\|u_i\|_{h^\beta([0,T],Y_s)} \leq R^c$, $\|\rho_i\|_{h^\beta([0,T],Z_s)} \leq R$. For $s=\alpha$, this is again fulfilled for $u_i \in M_T^c$ and $\rho_i \in M_T^h$ with $R=R^h$.
\end{remark}

As preparation for the following two sections, we deduce a technical auxiliary corollary from Remark \ref{compositionOP_Zus}.

\begin{corollary}\label{Produkt_Wohldef_h}
We suppose Assumptions \ref{Vor} are valid and use Notations \ref{Not2}. For $u \in M_T^c$ and $\rho \in M_{T}^h$, we have $A^h[\rho] \in \E_{0,T}$ and $G^h_u(\rho) \in \E_{0,T}$ and  
\begin{align*}
\big(g(u)a(\rho)Q(\rho)\big)(t) \in X_{2\beta+\alpha}
\phantom{xxx} \text{as well as} \phantom{xxx} 
\left( G'''(u) \big|\nabla_\rho u\big|^2\right)(t) \in X_{2\beta+\alpha}
\end{align*}
holds for all $t \in [0,T]$. 
\end{corollary}

\begin{proof}
We have $u,u_0 \in M_T^c$ and $\rho,0 \in M_T^h \subset h^\beta([0,T],Z_\alpha)$. Thus, Remark \ref{compositionOP_Zus} together with Lemma \ref{compositionOP_lin} yields $A^h[\rho],G^h_u(\rho) \in \E_{0,T}$. Furthermore, for every $t \in [0,T]$, we have $u(t) \in U^c_{2\beta+\alpha}$ and $\rho(t) \in U^h_{2\beta+\alpha} \subset Y_{2\beta+\alpha}$. Thus, Corollary \ref{Vorfaktoren_C2} with $s \defr 2\beta+\alpha$ yields the remaining claims.
\end{proof}

\subsection{Short-Time Existence for $\rho$}\label{ChaplokExh}

This section deals with the first equation \eqref{eq_lokEx_gls1}
\begin{align*}
\partial_t \rho = g(u)a(\rho)H(\rho)
\end{align*}
for height functions $\rho$ with initial value $\rho(0)=\rho_0$. We use the standard approach for parabolic, quasilinear partial differential equations of second order relying on linearization and a contraction argument, as explicated e.g. in \cite[Chapter 7]{LunardiASG}. For this, we first show that the linearization of the (elliptic) operator on the right hand side of the equation generates an analytic $C^0$-semigroup (see Proposition \ref{P_erzeugtHG_h}). In particular, the linearization of the initial value problem then yields an invertible operator (see Proposition \ref{L_bij_h}).

\begin{proposition}\label{P_erzeugtHG_h}
We suppose Assumptions \ref{Vor} are valid and use Notations \ref{Not2}. Then, 
\begin{align*}
A^h = A^h_{u_0,0}: Z_\alpha \rightarrow X_\alpha
\end{align*}
generates an analytic $C^0$-semigroup with $\mathcal{D}_{A^h}(\beta) = X_{2\beta+\alpha}$. If $u \in M_T^c$ and $\rho \in M_{T}^h$, also
\begin{align*}
A^h_{u(t),\rho(t)}: Z_s \rightarrow X_s 
\end{align*}
generates an analytic $C^0$-semigroup for $s \in \{\alpha, 2\beta+\alpha\}$ and $t \in [0,T]$. 
\end{proposition}

\begin{proof}
Let $u \in M_T^c$, $\rho \in M_T^h$ and fix $s \in \{\alpha, 2\beta+\alpha\}$, $t \in [0,T]$. 
Then, 
$g\big(u(t)\big), a\big(\rho(t)\big) \in X_s$ holds with Lemmas \ref{g} and \ref{nabla,div,a}(ii). Also, Lemma \ref{Vorfaktoren_elliptisch_P} yields that $P(\rho(t)) \in \mathcal{L}(Z_s,X_s)$ is a symmetric and elliptic differential operator of second order. Because we have 
\begin{align*}
A^h_{u(t),\rho(t)} = g\big(u(t)\big)a\big(\rho(t)\big)P\big(\rho(t)\big)
\end{align*}
with $g>0$ and $a>0$ by Assumption \ref{Vor}(i) and Remark \ref{a_pos}, $A^h_{u(t),\rho(t)} \in \mathcal{L}(Z_s,X_s)$ is a symmetric and elliptic differential operator of second order, too.
Due to Proposition \ref{HG_elliptic_semigroup_M},  $A^h_{u(t),\rho(t)}: \mathcal{D}\big(A^h_{u(t),\rho(t)}\big) \subset X_s \rightarrow X_s$ therefore generates an analytic $C^0$-semigroup with $\mathcal{D}\big(A^h_{u(t),\rho(t)}\big) = Z_s$.
Lemma \ref{Interpol_ComO} and the reiteration theorem finally imply
\begin{align*}
\mathcal{D}_{A^h_{u_0,0}}(\beta)
= \big( X_\alpha, \mathcal{D}\big(A^h_{u_0,0}\big) \big)_{\beta}
= \big( h^\alpha(\Sigma), h^{2+\alpha}(\Sigma) \big)_{\beta}
= h^{2\beta + \alpha}(\Sigma)
= X_{2\beta+\alpha}. &\qedhere
\end{align*}
\end{proof}

\begin{proposition}\label{L_bij_h}
We suppose Assumptions \ref{Vor} are valid and use Notations \ref{Not2}. Then, 
\begin{align*}
L^h: \E_{1,T} \rightarrow (\E_{0,T} \times Z_\alpha)^h_+
\end{align*}
is bijective with 
\begin{align*}
\Lambda^h \defr \sup_{0<T\leq 1} \| \big(L^h\big)^{-1} \|_{\mathcal{L}((\E_{0,T} \times Z_\alpha)^h_+,\E_{1,T})} < \infty,
\end{align*}
where
\begin{align*}
(\E_{0,T} \times Z_\alpha)^h_+ &\defr \big\{ (f,f_0) \in (\E_{0,T} \times Z_\alpha) \, \big| \, f(0)+A^h[f_0] \in \mathcal{D}_{A^h}(\beta)=X_{2\beta+\alpha} \big\} \text{ with } \\
\|(f,f_0)\|_{(\E_{0,T} \times Z_\alpha)^h_+} &\defr \|f\|_{\E_{0,T}} + \|f_0\|_{Z_\alpha} + \|f(0)+A^h[f_0]\|_{X_{2\beta+\alpha}} \text{ for } (f,f_0) \in (\E_{0,T} \times Z_\alpha)^h_+.
\end{align*}
In particular, $\Lambda^h = \Lambda^h(u_0)$ only depends on the initial value $u_0$.
\end{proposition}

\begin{proof}
By Proposition \ref{P_erzeugtHG_h}, $A^h$ satisfies the conditions of Proposition \ref{L_bij_allg}, which yields the claim. 
\end{proof}

As a next step, we prove a technical auxiliary lemma.

\begin{lemma}\label{G_wohldef_h}
We suppose Assumptions \ref{Vor} are valid and use Notations \ref{Not2}.
\begin{enumerate}
\item[(i)]
If $u \in M_T^c$ and $\rho \in M_{T}^h$ with $\rho(0) \in Z_{2\beta+\alpha}$, then $\big(G_u^h(\rho)\big)(0)=G_{u_0}^h\big(\rho(0)\big)$ holds in $X_\alpha$ and we have 
\begin{align*}
\big( G_u^h(\rho), \rho(0) \big) \in \big(\E_{0,T} \times Z_\alpha\big)^h_+.
\end{align*}
\item[(ii)]
There exists a constant $N^h = N^h \big(R^c,\delta_1\big)$ independent of $T$, $R^h$ and $u \in M_T^c$ such that
\begin{align*}
\left\| \big(G_u^h(\rho_0),\rho_0\big) \right\|_{(\E_{0,T} \times Z_\alpha)^h_+} \leq N^h
\end{align*}
holds for all $\rho_0 \in Z_{2\beta+\alpha}$ with $\|\rho_0\|_{Z_{2\beta+\alpha}} < \delta_1$, $\|\rho_0\|_{Y_\alpha} < \delta_0$.
\end{enumerate}
\end{lemma}

\begin{proof}$\phantom{.}$
\begin{enumerate}
\item[Ad (i)]
We have $G_u^h(\rho) \in \E_{0,T}$ by Corollary \ref{Produkt_Wohldef_h}, hence $\big( G_u^h(\rho), \rho(0) \big) \in \E_{0,T} \times Z_\alpha$ holds. 
Moreover, we have $u(0)=u_0 \in U^c_{2\beta+\alpha}$ and $\rho(0), 0 \in U^h_{2\beta+\alpha} \cap Z_{2\beta+\alpha}$. So, Corollary \ref{Vorfaktoren_C2} yields $\big(G_u^h(\rho)\big)(0) =  G_{u_0}^h\big(\rho(0)\big)$ in $X_{2\beta+\alpha} \hookrightarrow X_\alpha$ 
and therefore
\begin{align*}
A^h\big[\rho(0)\big] + \big(G_u^h(\rho)\big)(0) 
= g(u_0)a\big(\rho(0)\big)H\big(\rho(0)\big)  \in X_{2\beta+\alpha}
\end{align*}
follows with $X_{2\beta+\alpha} = \mathcal{D}_{A^h}(\beta)$ by Proposition \ref{P_erzeugtHG_h}. 
\item[Ad (ii)]
We have $u, u_0 \in M_T^c$ and $\rho_0, 0 \in \widetilde{M_T^h}$ with $\widetilde{M_T^h}$ defined as $M_T^h$ but with $\widetilde{R^h} \defr 2\delta_1$ instead of $R^h$. So, Remark \ref{compositionOP_Zus} together with Lemma \ref{compositionOP_lin} yields 
\begin{align*}
\big\| G_u^h(\rho_0) \big\|_{\E_{0,T}}
&= \left\| \big(g a H\big)(u,\rho_0) - \big(g a P\big)(u_0,0)[\rho_0] \right\|_{\E_{0,T}} \\
&\leq \left\| \Big(\big(g a P\big)(u,\rho_0) - \big(g a P\big)(u_0,0)\Big)[\rho_0] \right\|_{\E_{0,T}} + \big\| \big(g a Q\big)(u,\rho_0) \big\|_{\E_{0,T}} \\
&\leq C(R^c,\delta_1) \big(\|u-u_0\|_{h^\beta([0,T],Y_\alpha)} + \|\rho_0\|_{Y_\alpha} \big) \|\rho_0\|_{Z_\alpha} + C(R^c,\delta_1) \\
&\leq C(R^c,\delta_1)
\end{align*}
as well as
\begin{align*}
\left\| A^h[\rho_0] + \big(G_u^h(\rho_0)\big)(0) \right\|_{\mathcal{D}_{A^h}(\beta)} 
&\leq \left\| \big(g a P\big)(u_0,\rho_0)[\rho_0] \right\|_{X_{2\beta+\alpha}} + \left\| \big(g a Q\big)(u_0,\rho_0)\right\|_{X_{2\beta+\alpha}} \\
&\leq C(R^c,\delta_1) \big( \|\rho_0\|_{Z_{2\beta+\alpha}} + 1 \big) 
\leq C\big(R^c,\delta_1\big).
\end{align*}
(As $u_0$ and $\rho_0$ are independent of $t$, there is also no time dependence in the application of Remark \ref{compositionOP_Zus} in the estimate above.) Altogether, 
\begin{align*}
\left\| \big(G_u^h(\rho_0),\rho_0\big) \right\|_{(\E_{0,T} \times Z_\alpha)^h_+} 
\leq C\big(R^c,\delta_1\big) \defl N^h
\end{align*}
holds. \qedhere
\end{enumerate}
\end{proof}

The following proposition is the key point for the contraction argument. 

\begin{proposition}\label{Kontraktion_h}
We suppose Assumptions \ref{Vor} are valid and use Notations \ref{Not2}. There exists $\varepsilon>0$ with
\begin{align*}
\|G^h_{u_1}(\rho_1) - G^h_{u_2}(\rho_2) \|_{\E_{0,T}} 
&\leq C(R^\Sigma,R^c,R^h) T^\varepsilon \big( \| u_1-u_2\|_{\E_{1,T}} + \| \rho_1-\rho_2 \|_{\E_{1,T}} \big)\\
&\phantom{\leq x} + C(R^\Sigma,R^c,R^h) \|\rho_1(0)-\rho_2(0)\|_{Y_\alpha} \\
&\phantom{\leq x} + C(R^\Sigma, \|u_0\|_{Z_\alpha}, \|\rho_1(0)\|_{Z_\alpha}) \|\rho_1(0)\|_{Y_\alpha} \|\rho_1-\rho_2\|_{\E_{1,T}}
\end{align*}
for any $u_1, u_2 \in M_T^c$ and $\rho_1, \rho_2 \in M_{T}^h$.
\end{proposition}

\begin{proof}
Remark \ref{compositionOP_Zus} yields 
\begin{align*}
&\big\| \big( g a Q\big)(u_1,\rho_1)-\big( g a Q\big)(u_2,\rho_2) \big\|_{\E_{0,T}} \\
&\leq C(R^\Sigma,R^c,R^h) \big( \|u_1-u_2\|_{h^\beta([0,T],Y_\alpha)} + \|\rho_1-\rho_2\|_{h^\beta([0,T],Y_\alpha)} \big) \\
&\leq C(R^\Sigma,R^c,R^h) \Big( T^{\gamma-\beta} \big( \|u_1-u_2\|_{h^\gamma([0,T],Y_\alpha)} + \|\rho_1-\rho_2\|_{h^\gamma([0,T],Y_\alpha)} \big) + \|\rho_1(0)-\rho_2(0)\|_{Y_\alpha} \Big) \\
&\leq C(R^\Sigma,R^c,R^h) \Big( T^{\gamma-\beta} \big( \|u_1-u_2\|_{\E_{1,T}} + \|\rho_1-\rho_2\|_{\E_{1,T}} \big) + \|\rho_1(0)-\rho_2(0)\|_{Y_\alpha} \Big),
\end{align*}
where we used Remark \ref{Absch_Höldernorm} und Lemma \ref{EinbettungE_1} for the further estimate and $\gamma \in (0,1)$ with $\gamma>\beta$ is the exponent from Lemma \ref{EinbettungE_1}. For $w \in \E_{1,T} \subset h^\beta\big([0,T],Z_\alpha\big)$ and using Lemma \ref{compositionOP_lin}, we have analogously
\begin{align*}
&\Big\| \Big(\big( g a P \big)(u_1,\rho_1)-\big( g a P \big)(u_2,\rho_2)\Big)  [w] \Big\|_{\E_{0,T}} \\
&\leq C(R^\Sigma,R^c,R^h) \left( \|u_1-u_2\|_{h^\beta([0,T],Y_\alpha)} + \|\rho_1-\rho_2\|_{h^\beta([0,T],Y_\alpha)} \right) \|w\|_{h^\beta([0,T],Z_\alpha)} \\
&\leq C(R^\Sigma,R^c,R^h) \Big( T^{\gamma-\beta} \big( \|u_1-u_2\|_{\E_{1,T}} + \|\rho_1-\rho_2\|_{\E_{1,T}} \big) + \|\rho_1(0)-\rho_2(0)\|_{Y_\alpha} \Big) \|w\|_{\E_{1,T}}.
\end{align*}
Finally, using $\widetilde{R^c} \defr \|u_0\|_{Z_\alpha}$ and $\widetilde{R^h} \defr \|\rho_1(0)\|_{Z_\alpha}$ instead of $R^c$ and $R^h$, Remark \ref{compositionOP_Zus} with Lemma \ref{compositionOP_lin} implies
\begin{align*}
&\Big\| \Big(\big( g a P \big)(u_0,\rho_1(0))-\big( g a P \big)(u_0,0)\Big)  [w] \Big\|_{\E_{0,T}} 
\leq C(R^\Sigma, \|u_0\|_{Z_\alpha}, \|\rho_1(0)\|_{Z_\alpha}) \|\rho_1(0)\|_{Y_\alpha} \|w\|_{\E_{1,T}}
\end{align*}
for $w \in \E_{1,T}$.
Overall, 
\begin{align*}
&\| G^h_{u_1}(\rho_1)-G^h_{u_2}(\rho_2) \|_{\E_{0,T}} \\
&\leq \Big\| \Big(\big(g a P\big)(u_1,\rho_1)-\big(g a P\big)(u_0,\rho_1(0))\Big)[\rho_1-\rho_2] \Big\|_{\E_{0,T}} \\
&\phantom{\leq} + \Big\| \Big(\big(g a P\big)(u_0,\rho_1(0))-\big(g a P\big)(u_0,0)\Big)[\rho_1-\rho_2] \Big\|_{\E_{0,T}}\\
&\phantom{\leq} + \Big\| \Big(\big(g a P\big)(u_1,\rho_1)-\big(g a P\big)(u_2,\rho_2)\Big)[\rho_2] \Big\|_{\E_{0,T}} + \big\| \big(g a Q\big)(u_1,\rho_1)-\big(g a Q\big)(u_2,\rho_2) \big\|_{\E_{0,T}} \\
&\leq C(R^\Sigma,R^c,R^h) T^{\gamma-\beta} \big( \|u_1-u_0\|_{\E_{1,T}} + \|\rho_1-\rho_1(0)\|_{\E_{1,T}} \big) \|\rho_1-\rho_2\|_{\E_{1,T}}\\
&\phantom{\leq} + C(R^\Sigma, \|u_0\|_{Z_\alpha}, \|\rho_1(0)\|_{Z_\alpha}) \|\rho_1(0)\|_{Y_\alpha} \|\rho_1-\rho_2\|_{\E_{1,T}}\\
&\phantom{\leq} + C(R^\Sigma,R^c,R^h) \Big(T^{\gamma-\beta} \big( \|u_1-u_2\|_{\E_{1,T}} + \|\rho_1-\rho_2\|_{\E_{1,T}} \big) + \|\rho_1(0)-\rho_2(0)\|_{Y_\alpha} \Big) \big(\|\rho_2\|_{\E_{1,T}} {+} 1 \big) \\
&\leq C(R^\Sigma,R^c,R^h) \Big( T^{\gamma-\beta} \big( \|u_1-u_2\|_{\E_{T}} + \|\rho_1-\rho_2\|_{\E_{1,T}} \big) + \|\rho_1(0)-\rho_2(0)\|_{Y_\alpha} \Big) \\
&\phantom{\leq} + C(R^\Sigma, \|u_0\|_{Z_\alpha}, \|\rho_1(0)\|_{Z_\alpha}) \|\rho_1(0)\|_{Y_\alpha} \|\rho_1-\rho_2\|_{\E_{1,T}}
\end{align*}
follows.
\end{proof}

With this preparatory work, we can now prove short-time existence for the first equation \eqref{eq_lokEx_gls1}.

\begin{theorem}\label{lokEx_h}
We suppose Assumptions \ref{Vor} are valid and use Notations \ref{Not2}. Therein, choose $R^h = R^h\big(R^c,u_0,\delta_1\big)>0$ sufficiently large, choose $\delta_0 = \delta_0\big(R^\Sigma,u_0,\delta_1\big) \in (0,R^\Sigma)$ sufficiently small and choose $T=T\big(R^\Sigma,R^c,R^h,u_0,\delta_0\big) \in (0,1)$ sufficiently small. Then, for any initial value $\rho_0 \in Z_{2\beta+\alpha}$ with $\|\rho_0\|_{Z_{2\beta+\alpha}} < \delta_1$ and $\|\rho_0\|_{Y_\alpha} < \delta_0$ and any concentration $u \in M_T^c$, there exists a unique solution $\rho \defr \rho_{u,\rho_0} \in M_{T,\rho_0}^h$ of 
\begin{align*}
\phantom{\Leftrightarrow}\begin{cases}
\partial_t \rho &= \phantom{bl} g(u) a(\rho) H(\rho) \phantom{\rho_0} \text{in } \E_{0,T}, \\
\rho(0) &= \phantom{bl} \rho_0 \phantom{g(u) a(\rho) H(\rho)} \text{in } Z_\alpha.
\end{cases} 
\end{align*}
\end{theorem}

\begin{proof}
We show the existence of a unique solution $\rho \in M_{T,\rho_0}^h$ of
\begin{align}\label{DGL_h}
\begin{cases}
\partial_t \rho &= \phantom{bl} g(u) a(\rho) H(\rho) \phantom{\rho_0} \text{in } \E_{0,T}, \\
\rho(0) &= \phantom{bl} \rho_0 \phantom{g(u) a(\rho) H(\rho)} \text{in } Z_\alpha
\end{cases} 
\phantom{bla} \Leftrightarrow \phantom{bla} 
L^h[\rho] = \binom{G_u^h(\rho)}{\rho_0} \text{ in } \E_{0,T} \times Z_\alpha.
\end{align}
Equation \eqref{DGL_h} is well-defined because $A^h[\rho], G_u^h(\rho) \in \E_{0,T}$ holds for $\rho \in M_T^h$ and $u \in M_T^c$ by Corollary \ref{Produkt_Wohldef_h}. Due to Lemma \ref{G_wohldef_h}(i) and Proposition \ref{L_bij_h} it is equivalent to prove the existence of a unique $\rho \in M_{T,\rho_0}^h$ with 
\begin{align*}
L^h[\rho] = \binom{G_u^h(\rho)}{\rho_0} \text{ in } (\E_{0,T} \times Z_\alpha)_+^h
\phantom{bla} \Leftrightarrow \phantom{bla} 
\rho = \big(L^h\big)^{-1} \binom{G_u^h(\rho)}{\rho_0} \defl K_{u,\rho_0}^h(\rho) \text{ in } \E_{1,T}.
\end{align*}
So, we show that $K_{u,\rho_0}^h: M_{T,\rho_0}^h \subset \E_{1,T} \rightarrow \E_{1,T}$ has a unique fixed point $\rho \in M_{T,\rho_0}^h$ using the Banach fixed-point theorem. Due to Lemma \ref{G_wohldef_h}(i) and Proposition \ref{L_bij_h}, $K_{u,\rho_0}^h(\rho) \in \E_{1,T}$ is well-defined for $\rho \in M_{T,\rho_0}^h$. 
\begin{enumerate}
\item[Step 1:] 
We have to verify that $K_{u,\rho_0}^h$ is a contraction on $M_{T,\rho_0}^h$. For any $\rho_1, \rho_2 \in M_{T,\rho_0}^h$  
\begin{align*}
&\| K_{u,\rho_0}^h(\rho_1) - K_{u,\rho_0}^h(\rho_2) \|_{\E_{1,T}} 
\leq \Lambda^h \|G_u^h(\rho_1)-G_u^h(\rho_2)\|_{\E_{0,T}} \\
&\leq \Big( C(R^\Sigma,R^c,R^h,\Lambda^h) T^\varepsilon + C(R^\Sigma, \|u_0\|_{Z_\alpha}, \|\rho_0\|_{Z_\alpha}, \Lambda^h) \|\rho_0\|_{Y_\alpha} \Big) \|\rho_1-\rho_2\|_{\E_{1,T}} \\
&\leq \Big( C(R^\Sigma,R^c,R^h,\Lambda^h) T^\varepsilon + C(R^\Sigma, u_0, \delta_1, \Lambda^h) \delta_0 \Big) \|\rho_1-\rho_2\|_{\E_{1,T}}
\end{align*}
holds by Proposition \ref{L_bij_h}, Lemma \ref{G_wohldef_h}(i) as well as Proposition \ref{Kontraktion_h}. For sufficiently small $\delta_0>0$ and sufficiently small $T>0$,
\begin{align*}
\| K_{u,\rho_0}^h(\rho_1) - K_{u,\rho_0}^h(\rho_2) \|_{\E_{1,T}} 
\leq \frac{1}{4} \| \rho_1 - \rho_2 \|_{\E_{1,T}}
\end{align*}
follows. Because $\Lambda^h$ only depends on $u_0$, $\delta_0$ only depends on $R^\Sigma, u_0$ and $\delta_1$ whereas $T$ only depends on $R^\Sigma, R^c, R^h$ and $u_0$.
\item[Step 2:]
We have to show that $K_{u,\rho_0}^h: M_{T,\rho_0}^h \rightarrow M_{T,\rho_0}^h$ is a self-mapping.  Any $\rho \in M_{T,\rho_0}^h$ fulfills $\big(K_{u,\rho_0}^h(\rho)\big)(0) = \rho_0$ in $Z_\alpha$ because $w \defr K_{u,\rho_0}^h(\rho)$ is a solution to 
\begin{align*}
L^h w = \binom{[L^hw]_1}{w(0)} = \binom{G_u^h(\rho)}{\rho_0} \text{ in } \E_{0,T} \times Z_\alpha.
\end{align*}
Furthermore, we have
\begin{align*}
\| K_{u,\rho_0}^h(\rho) \|_{\E_{1,T}}
&\leq \|K_{u,\rho_0}^h(\rho_0)\|_{\E_{1,T}} + \| K_{u,\rho_0}^h(\rho) - K_{u,\rho_0}^h(\rho_0) \|_{\E_{1,T}} \\
&\leq \Lambda^h \big\| \big( G_u^h(\rho_0),\rho_0 \big) \big\|_{(\E_{0,T} \times Z_\alpha)^h_+} + \frac{1}{4} \| \rho -\rho_0\|_{\E_{1,T}} \\
&\leq \Lambda^h N^h + \frac{1}{4} \left( \|\rho\|_{\E_{1,T}} + 2\|\rho_0\|_{Z_\alpha} \right) 
\leq \frac{R^h}{2} + \frac{R^h}{2} = R^h,
\end{align*}
where the first summand is bounded by Proposition \ref{L_bij_h} and Lemma \ref{G_wohldef_h}(ii) and the second summand by the contraction-property (see step 2). The constant $R^h$ being sufficiently large thus means $R^h \geq 2\Lambda^hN^h$ and because $\Lambda^h$ only depends on $u_0$ and $N^h$ only depends on $R^c$ and $\delta_1$, we have $R^h=R^h\big(R^c,u_0,\delta_1\big)$. The two properties just deduced imply $K_{u,\rho_0}^h(\rho) \in M_{T,\rho_0}^h$ for all $\rho \in M_{T,\rho_0}^h$. \qedhere 
\end{enumerate}
\end{proof}

Now that we know that there exists a solution $\rho_{u,\rho_0}$ to the first equation \eqref{eq_lokEx_gls1}, we analyze some of its properties. First, we discuss its dependence on the concentration $u$ and the initial value $\rho_0$. The result in Proposition \ref{Absch_h_c} will be necessary for the contraction argument for the second equation \eqref{eq_lokEx_gls2}. Afterwards, we state an improved regularity in space for the solution in Proposition \ref{Neuanfang_h}.

\begin{proposition}\label{Absch_h_c}
We suppose that Assumptions \ref{Vor} are valid and use Notations \ref{Not2}. Therein, choose $R^h>0$ as large and choose $\delta_0>0$, $T>0$ as small as in Theorem \ref{lokEx_h}. There exists $\varepsilon>0$ with 
\begin{align*}
\| \rho_1 - \rho_2 \|_{\E_{1,T}} \leq C(R^\Sigma,R^c,R^h,\Lambda^h,\delta_1) \big(T^\varepsilon \|u_1-u_2\|_{\E_{1,T}} + \|\rho_{0,1}-\rho_{0,2}\|_{Z_{2\beta+\alpha}} \big) 
\end{align*}
for any $u_1,u_2 \in M_T^c$ and $\rho_{0,1}, \rho_{0,2} \in Z_{2\beta+\alpha}$ with $\|\rho_{0,i}\|_{Z_{2\beta+\alpha}} < \delta_1$ and $\|\rho_{0,i}\|_{Y_\alpha} < \delta_0$, where $\rho_i \defr \rho_{u_i,\rho_{0,i}} \in M_{T}^h$ is the solution from Theorem \ref{lokEx_h} associated with the concentration $u_i$ and the initial value $\rho_{0,i}$, respectively.
\end{proposition}

\begin{proof}
As $\rho_i \in M_{T}^h$ is the solution from Theorem \ref{lokEx_h} associated with the concentration $u_i$ and the initial value $\rho_{0,i}$, it is a fixed point of $(L^h)^{-1}\big( G_{u_i}^h(\cdot),\rho_{0,i} \big)$ as in the proof of Theorem \ref{lokEx_h}. Therefore, we have
\begin{align*}
\|\rho_1-\rho_2\|_{\E_{1,T}} 
&\leq \Lambda^h \big\|\big(G_{u_1}^h(\rho_1),\rho_{0,1}\big) - \big(G_{u_2}^h(\rho_2),\rho_{0,2}\big) \big\|_{(\E_{0,T} \times Z_\alpha)^h_+} \\
&=\Lambda^h \big\| G_{u_1}^h(\rho_1) - G_{u_2}^h(\rho_2) \big\|_{\E_{0,T}} + \Lambda^h \|\rho_{0,1}-\rho_{0,2}\|_{Z_\alpha} \\
&\phantom{= x}+ \Lambda^h \big\| G_{u_0}(\rho_{0,1}) - G_{u_0}(\rho_{0,2}) + A^h[\rho_{0,1}-\rho_{0,2}] \big\|_{X_{2\beta+\alpha}}
\end{align*}
by Proposition \ref{L_bij_h} and Lemma \ref{G_wohldef_h}(i). With $\delta_0>0$ and $T>0$ as small as in Theorem \ref{lokEx_h}, Proposition \ref{Kontraktion_h} yields
\begin{align*}
&\Lambda^h \big\| G_{u_1}^h(\rho_1) - G_{u_2}^h(\rho_2) \big\|_{\E_{0,T}} \\
&\leq \frac{1}{4} \|\rho_1-\rho_2\|_{\E_{1,T}} + C(R^\Sigma, R^c, R^h, \Lambda^h) \big( T^\varepsilon \|u_1-u_2\|_{\E_{1,T}} + \|\rho_{0,1}-\rho_{0,2}\|_{Y_\alpha} \big).
\end{align*}
Due to $u_0 \in M_T^c$ and $\rho_{0,i} \in M_T^h$, Remark \ref{compositionOP_Zus} together with Lemma \ref{compositionOP_lin} implies 
\begin{align*}
&\big\| G_{u_0}(\rho_{0,1}) - G_{u_0}(\rho_{0,2}) + A^h[\rho_{0,1}-\rho_{0,2}] \big\|_{X_{2\beta+\alpha}} \\
&\leq \Big\| \Big( \big(gaP\big)(u_0,\rho_{0,1}) - \big(gaP\big)(u_0,\rho_{0,2}) \Big) [\rho_{0,1}] \Big\|_{X_{2\beta+\alpha}} + \big\| \big(gaP\big)(u_0,\rho_{0,2})[\rho_{0,1}-\rho_{0,2}] \big\|_{X_{2\beta+\alpha}} \\
&\phantom{= x} + \big\| \big(gaQ\big)(u_0,\rho_{0,1}) - \big(gaQ\big)(u_0,\rho_{0,2}) \big\|_{X_{2\beta+\alpha}} \\
&\leq C(R^\Sigma, R^c, R^h) \big( \|\rho_{0,1}-\rho_{0,2} \|_{Y_{2\beta+\alpha}} \|\rho_{0,1}\|_{Z_{2\beta+\alpha}} + \|\rho_{0,1}-\rho_{0,2}\|_{Z_{2\beta+\alpha}} + \|\rho_{0,1}-\rho_{0,2}\|_{Y_{2\beta+\alpha}} \big) \\
&\leq C(R^\Sigma, R^c, R^h, \delta_1) \|\rho_{0,1}-\rho_{0,2}\|_{Z_{2\beta+\alpha}}.
\end{align*}
(As $u_0$ and $\rho_{0,i}$ are all independent of $t$, there is also no time dependence in the appliciation of Lemma \ref{compositionOP_lin} in the estimate above.) Altogether, we thus have
\begin{align*}
\|\rho_1-\rho_2\|_{\E_{1,T}} 
\leq C(R^\Sigma, R^c, R^h, \Lambda^h, \delta_1) \big( T^\varepsilon \|u_1-u_2\|_{\E_{1,T}} + \|\rho_{0,1}-\rho_{0,2}\|_{Z_{2\beta+\alpha}} \big). & \qedhere
\end{align*}
\end{proof}

\begin{proposition}\label{Neuanfang_h}
We suppose that Assumptions \ref{Vor} are valid and use Notations \ref{Not2}. Therein, choose $R^h>0$ as large and choose $\delta_0>0$, $T>0$ as small as in Theorem \ref{lokEx_h}. Let $u \in M_T^c$ and $\rho_0 \in Z_{2\beta+\alpha}$ with $\|\rho_0\|_{Z_{2\beta+\alpha}} < \delta_1$ and $\|\rho_0\|_{Y_\alpha} < \delta_0$ be arbitrary and let $\rho \defr \rho_{u,\rho_0} \in M_{T,\rho_0}^h$ be the associated solution from Theorem \ref{lokEx_h}. Then, $\rho(t) \in Z_{2\beta+\alpha}$ holds for all $t \in [0,T]$.
\end{proposition}

\begin{proof}
Fix $t \in [0,T]$. We have 
\begin{align*}
A^h_{u(t),\rho(t)}\big[\rho(t)\big]
= \partial_t \rho(t) - (gaQ)(u,\rho)(t) 
\in X_{2\beta+\alpha}
\end{align*}
by Proposition \ref{L_bij_allg}(i) and Corollary \ref{Produkt_Wohldef_h}. Because $A^h_{u(t),\rho(t)}: Z_{s} \rightarrow X_{s}$ generates an analytic $C^0$-semigroup for $s \in \{\alpha, 2\beta+\alpha\}$ (see Proposition \ref{P_erzeugtHG_h}), Lemma \ref{Neuanfang_Lemma} yields $\rho(t) \in Z_{2\beta+\alpha}$.
\end{proof}

\subsection{Short-Time Existence for $u$}\label{ChaplokExc}

In this section, we discuss the second equation \eqref{eq_lokEx_gls2}
\begin{align*}
\partial_t u = \Delta_\rho G'(u) + g(u) a(\rho) H(\rho) \, \nu_\Sigma \cdot \nabla_\rho u + g(u)H(\rho)^2 u
\end{align*}
for concentrations $u$ with initial value $u(0)=u_0$. As height function $\rho$, we insert the solution function $\rho_{u,\rho_0}$ from Theorem \ref{lokEx_h} with initial value $\rho_0$. Both equations \eqref{eq_lokEx_gls1} and \eqref{eq_lokEx_gls2} are parabolic, quasilinear partial differential equations of second order. Due to this parallel structure, we apply the same approach as in Section \ref{ChaplokExh} to solve this second equation, using linearization and a contraction argument. \\
First, we deduce a corollary from Remark \ref{compositionOP_Zus}, which contains the analogous statement to Corollary \ref{Produkt_Wohldef_h} but for $A^c$ and $G^c$ instead of $A^h$ and $G^h$.

\begin{corollary}\label{Produkt_Wohldef_c}
We suppose Assumptions \ref{Vor} are valid and use Notations \ref{Not2}. Therein, choose $R^h>0$ as large and choose $\delta_0>0$, $T>0$ as small as in Theorem \ref{lokEx_h}. Let $\rho_0 \in Z_{2\beta+\alpha}$ with $\|\rho_0\|_{Z_{2\beta+\alpha}} < \delta_1$ and $\|\rho_0\|_{Y_\alpha} < \delta_0$. For $u \in M_T^c$, we have $A^c[u] \in \E_{0,T}$ and $G^c_{\rho_0}(u) \in \E_{0,T}$.
\end{corollary}

\begin{proof}
Let $\rho_{u,\rho_0} \in M_{T,\rho_0}^h$ be the solution from Theorem \ref{lokEx_h} associated with the concentration $u$ and the initial value $\rho_0$. Then, we have $u,u_0 \in M_T^c \subset h^\beta([0,T],Z_\alpha)$ and $\rho_{u,\rho_0},0 \in M_T^h \subset h^\beta([0,T],Z_\alpha)$. Thus, Remark \ref{compositionOP_Zus} together with Lemma \ref{compositionOP_lin} yields the statement. 
\end{proof}

As in Section \ref{ChaplokExh}, we show that the linearization of the (elliptic) operator on the right hand side of the equation generates an analytic $C^0$-semigroup, which implies that the linearization of the initial value problem defines an invertible operator. 

\begin{proposition}\label{P_erzeugtHG_c}
We suppose Assumptions \ref{Vor} are valid and use Notations \ref{Not2}. Then, 
\begin{align*}
&A^c = A^c_{u_0,0}: Z_\alpha \rightarrow X_\alpha
\end{align*}
generates an analytic $C^0$-semigroup with $\mathcal{D}_{A^c}(\beta) = X_{2\beta + \alpha}$. \\
Let $R^h$ be as large and let $\delta_0>0$, $T>0$ be as small as in Theorem \ref{lokEx_h}. If $\rho \defr \rho_{u,\rho_0} \in M_{T,\rho_0}^h$ is the solution from Theorem \ref{lokEx_h} associated with the concentration $u \in M_T^c$ and the initial value $\rho_0 \in Z_{2\beta+\alpha}$ with $\|\rho_0\|_{Z_{2\beta+\alpha}} < \delta_1$ and $\|\rho_0\|_{Y_\alpha} < \delta_0$, also
\begin{align*}
A^c_{u(t),\rho(t)}: Z_s \rightarrow X_s
\end{align*} 
generates an analytic $C^0$-semigroup for $s \in \{ \alpha, 2\beta+\alpha \}$ and $t \in [0,T]$. 
\end{proposition}

\begin{proof}
Fix $s \in \{\alpha, 2\beta+\alpha\}$ and $t \in [0,T]$. Any solution $\rho \defr \rho_{u,\rho_0} \in M_{T,\rho_0}^h$ from Theorem \ref{lokEx_h} fulfills $\rho(t) \in Z_s$ with Proposition \ref{Neuanfang_h} and due to $\|\rho(t)\|_{C^1(M)} \leq \|\rho(t)\|_{Y_\alpha} < R^\Sigma$ (see remark after Notations \ref{Not2}), $\rho(t) \in U^h_{1+s,1}$ follows. 
By Remark \ref{Vorfaktoren_elliptisch_Laplace}, $\Delta_{\rho(t)} \in \mathcal{L}(Z_s,X_s)$ is a symmetric and elliptic differential operator of second order. Because we have 
\begin{align*}
A^c_{u(t),\rho(t)} = G''\big(u(t)\big) \Delta_{\rho(t)} + \text{ lower order terms}
\end{align*}
with $G''>0$ by Assumption \ref{Vor}(i), $A^c_{u(t),\rho(t)} \in \mathcal{L}(Z_s,X_s)$ is a symmetric and elliptic differential operator of second order, too. The operator $A^c_{u(t),\rho(t)}: \mathcal{D}\big(A^c_{u(t),\rho(t)}\big) \subset X_s \rightarrow X_s$ therefore generates an analytic $C^0$-semigroup with $\mathcal{D}\big(A^c_{u(t),\rho(t)}\big) = Z_s$ on account of Proposition \ref{HG_elliptic_semigroup_M}. 
Lemma \ref{Interpol_ComO} together with the reiteration theorem finally implies
\begin{align*}
\mathcal{D}_{A^c_{u_0,0}}(\beta)
= \big( X_\alpha, \mathcal{D}\big(A^c_{u_0,0}\big) \Big)_{\beta}
= \big( h^\alpha(\Sigma), h^{2+\alpha}(\Sigma) \big)_{\beta}
= h^{2\beta + \alpha}(\Sigma)
= X_{2\beta+\alpha}, 
\end{align*}
where $(\cdot,\cdot)_\beta$ denotes the continuous interpolation functor.
\end{proof}

\begin{proposition}\label{L_bij_c}
We suppose Assumptions \ref{Vor} are valid and use Notations \ref{Not2}. Then, 
\begin{align*}
L^c: \E_{1,T} \rightarrow (\E_{0,T} \times Z_\alpha)^c_+ 
\end{align*}
is bijective with 
\begin{align*}
\Lambda^c \defr \sup_{0<T\leq 1} \| \big(L^c\big)^{-1} \|_{\mathcal{L}((\E_{0,T} \times Z_\alpha)^c_+,\E_{1,T})} < \infty, 
\end{align*}
where
\begin{align*}
(\E_{0,T} \times Z_\alpha)^c_+ &\defr \big\{ (f,f_0) \in (\E_{0,T} \times Z_\alpha) \, \big| \, f(0)+A^c[f_0] \in \mathcal{D}_{A^c}(\beta)=X_{2\beta+\alpha} \big\} \text{ with } \\
\|(f,f_0)\|_{(\E_{0,T} \times Z_\alpha)^c_+} &\defr \|f\|_{\E_{0,T}} + \|f_0\|_{Z_\alpha} + \|f(0)+A^c[f_0]\|_{X_{2\beta+\alpha}} \text{ for } (f,f_0) \in (\E_{0,T} \times Z_\alpha)^c_+.
\end{align*}
In particular, $\Lambda^c = \Lambda^c(u_0)$ only depends on the initial value $u_0$.
\end{proposition}

\begin{proof}
By Proposition \ref{P_erzeugtHG_c}, $A^c$ satisfies the conditions of Proposition \ref{L_bij_allg}, which yields the claim. 
\end{proof}

We show a technical auxiliary lemma analogous to Lemma \ref{G_wohldef_h}.

\begin{lemma}\label{G_wohldef_c}
We suppose Assumptions \ref{Vor} are valid and use Notations \ref{Not2}. Therein, choose $R^h>0$ as large and choose $\delta_0>0$, $T>0$ as small as in Theorem \ref{lokEx_h}. 
\begin{enumerate}
\item[(i)]
Let $\rho_0 \in Z_{2\beta+\alpha}$ with $\|\rho_0\|_{Z_{2\beta+\alpha}} < \delta_1$ and $\|\rho_0\|_{Y_\alpha} < \delta_0$ and let $u \in M_T^c$. Then 
\begin{align*}
\big(G^c_{\rho_0}(u)\big)(0) 
= \Delta_{\rho_0} G'(u_0) + g(u_0)a(\rho_0)H(\rho_0)\nu_\Sigma \cdot \nabla_{\rho_0} u_0 + g(u_0)H(\rho_0)^2u_0 - A^c[u_0]
\end{align*}
holds in $X_\alpha$. In particular, $\big(G^c_{\rho_0}(u)\big)(0)$ is independent of $u$. Furthermore, we have 
\begin{align*}
\big( G^c_{\rho_0}(u), u_0\big) \in \big(\E_{0,T} \times Z_\alpha\big)^c_+. 
\end{align*}
\item[(ii)]
There exists a constant $N^c = N^c \big(R^\Sigma, u_0, \delta_1\big)$ independent of $T$, $R^c$ and $R^h$ such that 
\begin{align*}
\big\|\big(G^c_{\rho_0}(u_0),u_0\big)\big\|_{(\E_{0,T} \times Z_\alpha)^c_+} \leq N^c
\end{align*}
holds for all $\rho_0 \in Z_{2\beta+\alpha}$ with $\|\rho_0\|_{Z_{2\beta+\alpha}} < \delta_1$ and $\|\rho_0\|_{Y_\alpha} < \delta_0$.
\end{enumerate} 
\end{lemma}

\begin{proof}$\phantom{.}$
\begin{enumerate}
\item[Ad (i)]
We have $G^c_{\rho_0}(u) \in \E_{0,T}$ by Corollary \ref{Produkt_Wohldef_c}, hence $\big( G^c(u), u_0\big) \in \E_{0,T} \times Z_\alpha$ holds. \\
Let $\rho \defr \rho_{u,\rho_0} \in M_{T,\rho_0}^h$ be the solution from Theorem \ref{lokEx_h} associated with the concentration $u \in M_T^c$ and the initial value $\rho_0$. Then, $u(0)=u_0 \in U^c_{2\beta+\alpha} \cap Z_{2\beta+\alpha}$ and $\rho(0)=\rho_0,0 \in U^h_{2\beta+\alpha} \cap U^h_{1+2\beta+\alpha,1} \subset Z_{2\beta+\alpha}$ hold. So, Corollary \ref{Vorfaktoren_C2} yields
\begin{align*}
\big(G^c_{\rho_0}(u)\big)(0) 
&= \Delta_{\rho_0} G'(u_0) + g(u_0) a(\rho_0) H(\rho_0) \nu_\Sigma \cdot \nabla_{\rho_0} u_0 + g(u_0) H(\rho_0)^2 u_0 \\
&\phantom{bla} - G''(u_0) \Delta_\Sigma u_0 - g(u_0) H_\Sigma^2 u_0 \quad \text{ in } X_{2\beta+\alpha} \hookrightarrow X_\alpha
\end{align*}
and therefore
\begin{align*}
A^c[u_0] {+} \big(G^c_{\rho_0}(u)\big)(0)
&= \Delta_{\rho_0} G'(u_0) {+} g(u_0) a(\rho_0) H(\rho_0) \nu_\Sigma {\cdot} \nabla_{\rho_0} u_0 {+} g(u_0) H(\rho_0)^2 u_0 \in X_{2\beta+\alpha}
\end{align*}
follows with $X_{2\beta+\alpha} = \mathcal{D}_{A^c}(\beta)$ by Proposition \ref{P_erzeugtHG_c}. 
\item[Ad (ii)]
We have 
\begin{align*}
\big\|\big(G^c_{\rho_0}(u_0),u_0\big)\big\|_{(\E_{0,T} \times Z_\alpha)^c_+} 
= \big\| G^c_{\rho_0}(u_0) \big\|_{\E_{0,T}} + \|u_0\|_{Z_\alpha} + \left\| A^c[u_0] {+} \big(G^c_{\rho_0}(u_0)\big)(0) \right\|_{\mathcal{D}_{A^c}(\beta)}.
\end{align*}
Let $\rho \defr \rho_{u_0,\rho_0} \in M^h_{T,\rho_0} \subset h^\beta\big([0,T],Z_\alpha\big)$ be the solution from Theorem \ref{lokEx_h} associated with the concentration $u_0$ and the initial value $\rho_0$. We have $u_0 \in \widetilde{M_T^c}$ and $\rho, \rho_0, 0 \in \widetilde{M_T^h}$ with $\widetilde{M_T^c}$, $\widetilde{M_T^h}$ defined as $M_T^c$, $M_T^h$ but with $\widetilde{R^c} \defr 2\|u_0\|_{Z_\alpha}$, $\widetilde{R^h} \defr \|\rho\|_{\E_{1,T}} \geq 2\|\rho_0\|_{Z_\alpha}$ instead of $R^c$, $R^h$. We thus can use Remark \ref{compositionOP_Zus} and Lemma \ref{compositionOP_lin} to bound
\begin{align*}
\big\| G^c_{\rho_0}(u_0) \big\|_{\E_{0,T}} 
&= \big\| \Delta_\rho G'(u_0) + g(u_0)a(\rho)H(\rho)\nu_\Sigma \cdot \nabla_\rho u_0 + g(u_0)H(\rho)^2u_0 \\
&\phantom{xxx} - G''(u_0)\Delta_\Sigma u_0 - g(u_0)H_\Sigma^2u_0 \big\|_{\E_{0,T}} \\
&\leq C\big(R^\Sigma, \|u_0\|_{Z_\alpha}, \|\rho\|_{\E_{1,T}}\big).
\end{align*}
as well as 
\begin{align*}
&\left\| A^c[u_0] + \big(G^c_{\rho_0}(u_0)\big)(0) - G''(u_0) J(\rho_0)[u_0] \right\|_{\mathcal{D}_{A^c}(\beta)} \\
&= \big\| G''(u_0) D(\rho_0)[u_0] + G'''(u_0) \big|\nabla_{\rho_0} u_0 \big|^2 \\
&\phantom{= \big\| } + g(u_0) a(\rho_0) H(\rho_0) \nu_\Sigma \cdot \nabla_{\rho_0} u_0 + g(u_0) H(\rho_0)^2 u_0 \big\|_{X_{2\beta+\alpha}} \\
&\leq C\big(R^\Sigma, \|u_0\|_{Z_\alpha}, \|\rho\|_{\E_{1,T}} \big) \, C\big( \|u_0\|_{Z_{2\beta+\alpha}}, \|\rho_0\|_{Z_{2\beta+\alpha}} \big) \\
&\leq C\big( R^\Sigma, \|u_0\|_{Z_{2\beta+\alpha}}, \delta_1, \|\rho\|_{\E_{1,T}} \big).
\end{align*}
(As $u_0$ and $\rho_0$ are independent of $t$, there is also no time dependence in the application of Remark \ref{compositionOP_Zus} in the estimate above.)
Moreover, we have $\rho_0 \in U^h_{1+2\beta+\alpha,1} \cap B_{\delta_1}^{Z_{2\beta+\alpha}}(0)$ and therefore a last application of Remark \ref{compositionOP_Zus} and Lemma \ref{compositionOP_lin} yields
\begin{align*}
\big\| G''(u_0) J(\rho_0)[u_0] \big\|_{\mathcal{D}_{A^c}(\beta)} 
\leq C\big( R^\Sigma, \|u_0\|_{Z_\alpha}, \delta_1 \big) \|u_0\|_{Y_{2\beta+\alpha}} 
\leq C\big( R^\Sigma, \|u_0\|_{Z_{2\beta+\alpha}}, \delta_1 \big).
\end{align*}
So, 
\begin{align*}
\left\| A^c[u_0] + \big(G^c_{\rho_0}(u_0)\big)(0) \right\|_{\mathcal{D}_{A^c}(\beta)} 
&\leq C\big( R^\Sigma, \|u_0\|_{Z_{2\beta+\alpha}}, \delta_1, \|\rho\|_{\E_{1,T}} \big)
\end{align*}
follows. Altogether, we thus have
\begin{align*}
\big\|\big(G^c_{\rho_0}(u_0),u_0\big)\big\|_{(\E_{0,T} \times Z_\alpha)^c_+} 
\leq C\big( R^\Sigma, \|u_0\|_{Z_{2\beta+\alpha}}, \delta_1, \|\rho\|_{\E_{1,T}} \big).
\end{align*}
Now, we have to explain why $\|\rho\|_{\E_{1,T}}$ can be bounded by a constant depending only on $R^\Sigma$, $u_0$ and $\delta_1$. As $\rho$ is the solution from Theorem \ref{lokEx_h}, $\|\rho\|_{\E_{1,T}} \leq R^h$ holds with $R^h = R^h\big(R^c, u_0, \delta_1\big)$. Because $\rho$ is associated to the concentration $u_0$, it suffices to use $R^h = R^h\big( 2\|u_0\|_{Z_\alpha}, u_0, \delta_1 \big)$ for the statement of Theorem \ref{lokEx_h}. Thus, we have $\|\rho\|_{\E_{1,T}} \leq R^h = C\big( u_0, \delta_1 \big)$ and therefore finally 
\begin{align*}
\big\|\big(G^c_{\rho_0}(u_0),u_0\big)\big\|_{(\E_{0,T} \times Z_\alpha)^c_+} 
\leq C\big( R^\Sigma, u_0, \delta_1 \big) \defl N^c
\end{align*}
follows. \qedhere
\end{enumerate}
\end{proof}

With the help of Proposition \ref{Absch_h_c}, an analogous statement to Proposition \ref{Kontraktion_h} holds which again will be the key point to the contraction argument.

\begin{proposition}\label{Kontraktion_c}
We suppose that Assumptions \ref{Vor} are valid and use Notations \ref{Not2}. Therein, choose $R^h>0$ as large and choose $\delta_0>0$, $T>0$ as small as in Theorem \ref{lokEx_h}. There exists $\varepsilon>0$ with
\begin{align*}
\|G^c_{\rho_{0,1}}(u_1) - G^c_{\rho_{0,2}}(u_2) \|_{\E_{0,T}} 
&\leq C(R^\Sigma, R^c, R^h, \Lambda^h, \delta_1) \big(T^\varepsilon \|u_1-u_2\|_{\E_{1,T}} + \|\rho_{0,1}-\rho_{0,2}\|_{Z_{2\beta+\alpha}} \big) \\
&\phantom{xxx} + C(R^\Sigma, \|u_0\|_{Z_\alpha}, \delta_1) \delta_0 \|u_1-u_2\|_{\E_{1,T}}
\end{align*}
for $u_1,u_2 \in M_T^c$ and initial values $\rho_{0,1}, \rho_{0,2} \in Z_{2\beta+\alpha}$ with $\|\rho_{0,i}\|_{Z_{2\beta+\alpha}} < \delta_1$ and $\|\rho_{0,i}\|_{Y_\alpha} < \delta_0$. 
\end{proposition}

\begin{proof}
Let $\rho_i \defr \rho_{u_i,\rho_{0,i}} \in M_T^h$ be the solution from Theorem \ref{lokEx_h} associated with the concentration $u_i$ and the initial value $\rho_{0,i}$. Using appropriate  triangle inequalities (as in the proof of Proposition \ref{Kontraktion_h}), Remark \ref{compositionOP_Zus} together with Lemma \ref{compositionOP_lin} yields
\begin{align*}
&\Big\| \Big( G^c_{\rho_{0,1}}(u_1)-G^c_{\rho_{0,2}}(u_2) \Big) - \Big( G''(u_1)D(\rho_1)[u_1-u_2]- G''(u_0)D(0)[u_1-u_2] \Big) \Big\|_{\E_{0,T}} \\
&= \Big\| G''(u_1)D(\rho_1)[u_2] - G''(u_2)D(\rho_2)[u_2] \\
&\phantom{xxx} + G''(u_1)J(\rho_1)[u_1] - G''(u_2)J(\rho_2)[u_2] - G''(u_0)J(0)[u_1-u_2] \\ 
&\phantom{xxx}+ G'''(u_1)\big|\nabla_{\rho_1}u_1 \big|^2 - G'''(u_2)\big|\nabla_{\rho_2}u_2 \big|^2 \\
&\phantom{xxx}+ g(u_1) a(\rho_1) H(\rho_1) \nu_\Sigma \cdot \nabla_{\rho_1} u_1 - g(u_2) a(\rho_2) H(\rho_2) \nu_\Sigma \cdot \nabla_{\rho_2} u_2 \\
&\phantom{xxx}+ g(u_1) H(\rho_1)^2 u_1 - g(u_2) H(\rho_2)^2 u_2 - g(u_0) H_\Sigma^2 [u_1-u_2] \Big\|_{\E_{0,T}} \\
&\leq C( R^\Sigma, R^c, R^h) \big( \|u_1-u_2\|_{h^\beta([0,T],Y_\alpha)} + \|\rho_1-\rho_2\|_{h^\beta([0,T],Y_\alpha)} \big) \\
&\phantom{xxx} + C( R^\Sigma, R^c, R^h) \big( \|u_1-u_0\|_{h^\beta([0,T],Y_\alpha)} + \|\rho_1\|_{h^\beta([0,T],Z_\alpha)} \big) \|u_1-u_2\|_{h^\beta([0,T],Y_\alpha)} \\
&\phantom{xxx} + C(R^\Sigma, R^c, R^h) \big( \|u_1-u_2\|_{h^\beta([0,T],Y_\alpha)} + \|\rho_1-\rho_2\|_{h^\beta([0,T],Z_\alpha)} \big) \\
&\leq C(R^\Sigma, R^c, R^h) \big( \|u_1-u_2\|_{h^\beta([0,T],Y_\alpha)} + \|\rho_1-\rho_2\|_{h^\beta([0,T],Z_\alpha)} \big).
\end{align*}
Analogously, using $\widetilde{R^c} \defr \|u_0\|_{Z_\alpha}$ and $\widetilde{R^h} \defr \|\rho_{1,0}\|_{Z_\alpha}$ instead of $R^c$ and $R^h$ for the second summand, Remark \ref{compositionOP_Zus} with Lemma \ref{compositionOP_lin} implies
\begin{align*}
&\big\| G''(u_1)D(\rho_1)[u_1-u_2]- G''(u_0)D(0)[u_1-u_2] \big\|_{\E_{0,T}} \\
&\leq \big\| \big(G''(u_1)D(\rho_1) - G''(u_0)D(\rho_{0,1})\big)[u_1-u_2] \big\|_{\E_{0,T}} \\
&\phantom{xxx} + \big\| \big(G''(u_0)D(\rho_{0,1}) - G''(u_0)D(0)\big)[u_1-u_2] \big\|_{\E_{0,T}} \\
&\leq C(R^\Sigma, R^c, R^h) \big( \|u_1-u_0\|_{h^\beta([0,T],Y_\alpha)} + \|\rho_1-\rho_{0,1}\|_{h^\beta([0,T],Y_\alpha)} \big) \|u_1-u_2\|_{h^\beta([0,T],Z_\alpha)} \\
&\phantom{xxx} + C(R^\Sigma, \|u_0\|_{Z_\alpha}, \|\rho_{0,1}\|_{Z_\alpha}) \|\rho_{0,1}\|_{Y_\alpha} \|u_1-u_2\|_{h^\beta([0,T],Z_\alpha)}.
\end{align*}
So, altogether, we have
\begin{align*}
&\big\| G^c_{\rho_{0,1}}(u_1)-G^c_{\rho_{0,2}}(u_2) \big\|_{\E_{0,T}} \\
&\leq C(R^\Sigma, R^c, R^h) \big( \|u_1-u_2\|_{h^\beta([0,T],Y_\alpha)} + \|\rho_1-\rho_2\|_{h^\beta([0,T],Z_\alpha)} \big) \\
&\phantom{xxx} + C(R^\Sigma, R^c, R^h) \big( \|u_1-u_0\|_{h^\beta([0,T],Y_\alpha)} + \|\rho_1-\rho_{0,1}\|_{h^\beta([0,T],Y_\alpha)} \big) \|u_1-u_2\|_{h^\beta([0,T],Z_\alpha)} \\
&\phantom{xxx} + C(R^\Sigma, \|u_0\|_{Z_\alpha}, \delta_1) \delta_0 \|u_1-u_2\|_{h^\beta([0,T],Z_\alpha)}.
\end{align*}
For the further estimate, we use Remark \ref{Absch_Höldernorm} und Lemma \ref{EinbettungE_1} and choose $\gamma \in (0,1)$ with $\gamma>\beta$ as the exponent from Lemma \ref{EinbettungE_1}. We obtain
\begin{align*}
&\big\| G^c_{\rho_{0,1}}(u_1)-G^c_{\rho_{0,2}}(u_2) \big\|_{\E_{0,T}} \\
&\leq C(R^\Sigma,R^c,R^h) \left( T^{\gamma-\beta} \|u_1-u_2\|_{\E_{1,T}} + \|\rho_1-\rho_2\|_{h^\beta([0,T],Z_\alpha)} \right) \\
&\phantom{xxx} + C(R^\Sigma,R^c,R^h) T^{\gamma-\beta} \big(\|u_1-u_0\|_{\E_{1,T}} + \|\rho_1-\rho_{0,1}\|_{\E_{1,T}} \big) \|u_1-u_2\|_{h^\beta([0,T],Z_\alpha)} \\
&\phantom{xxx} + C(R^\Sigma, \|u_0\|_{Z_\alpha}, \delta_1) \delta_0 \|u_1-u_2\|_{h^\beta([0,T],Z_\alpha)} \\
&\leq C(R^\Sigma,R^c,R^h) \big( T^{\gamma-\beta} \|u_1-u_2\|_{\E_{1,T}} + \|\rho_1-\rho_2\|_{\E_{1,T}} \big) \\
&\phantom{xxx} + C(R^\Sigma, \|u_0\|_{Z_\alpha}, \delta_1) \delta_0 \|u_1-u_2\|_{\E_{1,T}}.
\end{align*}
Finally, due to Proposition \ref{Absch_h_c},
\begin{align*}
\big\| G^c_{\rho_{0,1}}(u_1)-G^c_{\rho_{0,2}}(u_2) \big\|_{\E_{0,T}} 
&\leq C(R^\Sigma, R^c, R^h, \Lambda^h, \delta_1) \big(T^\varepsilon \|u_1-u_2\|_{\E_{1,T}} + \|\rho_{0,1}-\rho_{0,2}\|_{Z_{2\beta+\alpha}} \big) \\
&\phantom{xxx} + C(R^\Sigma, \|u_0\|_{Z_\alpha}, \delta_1) \delta_0 \|u_1-u_2\|_{\E_{1,T}}
\end{align*}
holds for some $\varepsilon > 0$.
\end{proof}

The preparatory work above enables us to prove the short-time existence result for the second equation \eqref{eq_lokEx_gls2}.

\begin{theorem}\label{lokEx_c}
We suppose that Assumptions \ref{Vor} are valid and use Notations \ref{Not2}. Therein, choose $R^c=R^c\big(R^\Sigma,u_0,\delta_1\big)>0$ sufficiently large and then, depending on this $R^c$, choose $R^h=R^h\big(R^c,u_0,\delta_1\big)>0$ as large as in Theorem \ref{lokEx_h}. Also, choose $\delta_0 = \delta_0\big(R^\Sigma, u_0, \delta_1 \big)>0$ and $T=T\big(R^\Sigma, R^c, R^h, u_0, \delta_0, \delta_1 \big)>0$ sufficiently small, but at least as small as in Theorem \ref{lokEx_h}. Then, for any initial value $\rho_0 \in Z_{2\beta+\alpha}$ with $\|\rho_0\|_{Z_{2\beta+\alpha}} < \delta_1$ and $\|\rho_0\|_{Y_\alpha} < \delta_0$, there exists a unique solution $u \defr u_{\rho_0} \in M_T^c$ of 
\begin{align*}
\phantom{\Leftrightarrow}\begin{cases}
\partial_t u &= \phantom{bl} \Delta_{\rho_u} G'(u) + g(u) a(\rho_u) H(\rho_u) \nu_\Sigma \cdot \nabla_{\rho_u}u + g(u) H(\rho_u)^2 u \phantom{u_0} \text{in } \E_{0,T}, \\
u(0) &= \phantom{bl} u_0 \phantom{\Delta_{\rho_u} G'(u) + g(u) a(\rho_u) H(\rho_u) \nu_\Sigma \cdot \nabla_{\rho_u}u + g(u) H(\rho_u)^2 u} \text{in } Z_\alpha,
\end{cases} 
\end{align*}
where $\rho_u \defr \rho_{u,\rho_0} \in M_{T,\rho_0}^h$ is the solution from Theorem \ref{lokEx_h} associated with the concentration $u$ and the initial value $\rho_0$.
\end{theorem}

\begin{proof}
We show the existence of a unique solution $u \in M_T^c$ of 
\begin{align}\label{DGL_c}
\begin{cases}
\partial_t u &= \phantom{bl} \Delta_{\rho_u} G'(u) + g(u) a(\rho_u) H(\rho_u) \nu_\Sigma \cdot \nabla_{\rho_u}u + g(u) H(\rho_u)^2 u \phantom{_0} \text{in } \E_{0,T}, \\
u(0) &= \phantom{bl} u_0 \phantom{\Delta_{\rho_u} G'(u) + g(u) a(\rho_u) H(\rho_u) \cdot \nabla_{\rho_u}u + g(u) H(\rho_u)^2ii_{ii}} \text{in } Z_\alpha
\end{cases} \notag \\
\Leftrightarrow
L^c[u] = \binom{G^c_{\rho_0}(u)}{u_0} \text{ in } \E_{0,T} \times Z_\alpha. \phantom{bla} 
\end{align}
Equation (\ref{DGL_c}) is well-defined because $A^c[u], G^c_{\rho_0}(u) \in \E_{0,T}$ holds for $u \in M_T^c$ by Corollary \ref{Produkt_Wohldef_c}. Due to Lemma \ref{G_wohldef_c}(i) and Proposition \ref{L_bij_c} it is equivalent to prove the existence of a unique $u \in M_T^c$ with 
\begin{align*}
L^c[u] = \binom{G^c_{\rho_0}(u)}{u_0} \text{ in } (\E_{0,T} \times Z_\alpha)^c_+ 
\phantom{bla} \Leftrightarrow \phantom{bla}
u = \big(L^c\big)^{-1} \binom{G^c_{\rho_0}(u)}{u_0} \defl K^c_{\rho_0}(u) \text{ in } \E_{1,T}.
\end{align*} 
So, we show that $K^c_{\rho_0}: M_T^c \subset \E_{1,T} \rightarrow \E_{1,T}$ has a unique fixed point $u \in M_T^c$ using the Banach fixed-point theorem. Due to Lemma \ref{G_wohldef_c}(i) and Proposition \ref{L_bij_c}, $K^c_{\rho_0}(u) \in \E_{1,T}$ is well-defined for $u \in M_T^c$. 
\begin{enumerate}
\item[Step 1:] 
We have to verify that $K^c_{\rho_0}$ is a contraction on $M_T^c$. For any $u_1, u_2 \in M_T^c$  
\begin{align*}
&\| K^c_{\rho_0}(u_1) - K^c_{\rho_0}(u_2) \|_{\E_{1,T}} 
\leq \Lambda^c \|G^c_{\rho_0}(u_1)-G^c_{\rho_0}(u_2)\|_{\E_{0,T}} \\
&\leq \Big( C( R^\Sigma, R^c, R^h, \Lambda^c, \Lambda^h, \delta_1 ) T^\varepsilon + C(R^\Sigma, u_0, \delta_1, \Lambda^c) \delta_0 \Big) \|u_1-u_2\|_{\E_{1,T}} 
\end{align*}
holds by Proposition \ref{L_bij_c}, Lemma \ref{G_wohldef_c}(i) as well as Proposition \ref{Kontraktion_c}. For sufficiently small $\delta_0>0$ and sufficiently small $T>0$, 
\begin{align*}
\| K^c_{\rho_0}(u_1) - K^c_{\rho_0}(u_2) \|_{\E_{1,T}} 
\leq \frac{1}{4} \| u_1 - u_2 \|_{\E_{1,T}}
\end{align*}
follows. Because $\Lambda^c$ and $\Lambda^h$ only depend on $u_0$, $\delta_0$ only depends on $R^\Sigma, u_0$ and $\delta_1$ whereas $T$ only depends on $R^\Sigma, R^c, R^h, u_0$ and $\delta_1$.
\item[Step 2:]
We have to show that $K^c_{\rho_0}: M_T^c \rightarrow M_T^c$ is a self-mapping. Any $u \in M_T^c$ fulfills $\big(K^c_{\rho_0}(u)\big)(0) = u_0$ in $Z_\alpha$ because $w \defr K^c_{\rho_0}(u)$ is a solution to
\begin{align*}
L^c w = \binom{[L^c w]_1}{w(0)} = \binom{G^c_{\rho_0}(u)}{u_0} \text{ in } \E_{0,T} \times Z_\alpha.
\end{align*}
Furthermore, we have
\begin{align*}
\| K^c_{\rho_0}(u) \|_{\E_{1,T}}
&\leq \|K^c_{\rho_0}(u_0)\|_{\E_{1,T}} + \| K^c_{\rho_0}(u) - K^c_{\rho_0}(u_0) \|_{\E_{1,T}} \\
&\leq \Lambda^c \big\| \big( G^c_{\rho_0}(u_0),u_0 \big) \big\|_{(\E_{0,T} \times Z_\alpha)^c_+} + \frac{1}{4} \| u-u_0\|_{\E_{1,T}} \\
&\leq \Lambda^c N^c + \frac{1}{4} \left( \|u\|_{\E_{1,T}} + 2\|u_0\|_{Z_\alpha} \right) 
\leq \frac{R^c}{2} + \frac{R^c}{2} = R^c,
\end{align*}
where the first summand is bounded by Proposition \ref{L_bij_c} and Lemma  \ref{G_wohldef_c}(ii) and the second summand by the contraction-property (see step 2). The constant $R^c$ being sufficiently large thus means $R^c \geq 2\Lambda^cN^c$ and because $\Lambda^c$ only depends on $u_0$ and $N^c$ only depends on $R^\Sigma, u_0$ and $\delta_1$, we have $R^c = R^c\big( R^\Sigma, u_0, \delta_1 \big)$. The two properties just deduced imply $K^c_{\rho_0}(u) \in M_T^c$ for all $u \in M_T^c$. \qedhere
\end{enumerate}
\end{proof}

\begin{proposition}\label{Neuanfang_c}
We suppose that Assumptions \ref{Vor} are valid and use Notations \ref{Not2}. Therein, choose $R^c>0$, $R^h>0$ as large and choose $\delta_0>0$, $T>0$ as small as in Theorem \ref{lokEx_c}. Let $\rho_0 \in Z_{2\beta+\alpha}$ with $\|\rho_0\|_{Z_{2\beta+\alpha}} < \delta_1$ and $\|\rho_0\|_{Y_\alpha} < \delta_0$ and let $u \defr u_{\rho_0} \in M_T^c$ be the solution from Theorem \ref{lokEx_c} associated with $\rho_0$. Then, $u(t) \in Z_{2\beta+\alpha}$ holds for all $t \in [0,T]$.
\end{proposition}

\begin{proof}
Let $\rho \defr \rho_{u,\rho_0} \in M_{T,\rho_0}^h$ be the solution from Theorem \ref{lokEx_h} and fix $t \in [0,T]$. We have 
\begin{align*}
A^c_{u(t),\rho(t)}\big[u(t)\big] 
= \partial_t u(t) - \big(G'''(u) \big|\nabla_\rho u\big|^2\big)(t) 
\in X_{2\beta+\alpha}
\end{align*}
by Proposition \ref{L_bij_allg}(i) and Corollary \ref{Produkt_Wohldef_h}. Because $A^c_{u(t),\rho(t)}: Z_{s} \rightarrow X_{s}$ generates an analytic $C^0$-semigroup for $s \in \{ \alpha,  2\beta+\alpha \}$ (see Proposition \ref{P_erzeugtHG_c}), Lemma \ref{Neuanfang_Lemma} yields $u(t) \in Z_{2\beta+\alpha}$.
\end{proof}

\subsection{Analytic Short-Time Existence}\label{ChaplokExAnalyt}

Combining the results from Sections \ref{ChaplokExh} and \ref{ChaplokExc} yields our full statement on short-time existence. We formulate it in a self-contained way, such that the reader does not have to look up Assumptions \ref{Vor} or Notations \ref{Not2} that were continually used above. 

\begin{theorem}\label{lokEx_param}
Let $\alpha \in (0,1)$ and $\beta \in (0,\frac{1}{2})$ with $2\beta + \alpha \notin \N$ and let $G \in C^7(\R)$ with $G''>0$ and $g \defr G-G' \cdot \Id > 0$. Moreover, let $\Sigma=\bar{\theta}(M)$ be an $h^{4+\alpha}$-immersed closed hypersurface with unit normal $\nu_\Sigma$. Let $u_0 \in h^{2+2\beta+\alpha}(M)$ and $\delta_1>0$ be arbitrary. Then, choose $\delta_0 = \delta_0(\Sigma,u_0,\delta_1)>0$ and $T = T(\Sigma,u_0,\delta_1)>0$ sufficiently small. For every function $\rho_0 \in h^{2+2\beta+\alpha}(M)$ with $\|\rho_0\|_{h^{2+2\beta+\alpha}(M)} < \delta_1$ and $\|\rho_0\|_{h^{1+\alpha}(M)} < \delta_0$, there exists a solution $(\rho,u)$ with $\rho, u \in \E_{1,T} \defr h^{1+\beta}\big([0,T],h^\alpha(M)\big) \cap h^\beta\big([0,T],h^{2+\alpha}(M)\big)$ to
\begin{align*}
\left\{
\begin{aligned}
\partial_t \rho \phantom{bl} &= \phantom{bl} g(u)a(\rho)H(\rho) & &\text{ in } h^\beta\big([0,T],h^\alpha(M)\big), \\
\partial_t u \phantom{bl} &= \phantom{bl} \Delta_\rho G'(u) + g(u) a(\rho) H(\rho) \nu_\Sigma \cdot \nabla_\rho u + g(u) H(\rho)^2 u & &\text{ in } h^\beta\big([0,T],h^\alpha(M)\big), \\
\rho(0) \phantom{bl} &= \phantom{bl} \rho_0 & &\text{ in } h^{2+\alpha}(M), \\
u(0) \phantom{bl} &= \phantom{bl} u_0 & &\text{ in } h^{2+\alpha}(M).
\end{aligned}
\right.
\end{align*}
Furthermore, $\rho(t), u(t) \in h^{2+2\beta+\alpha}(M)$ as well as $\|\rho(t)\|_{h^{1+\alpha}(M)}<R^\Sigma$ hold for all $t \in [0,T]$ and there exists a constant $R = R(\Sigma,u_0,\delta_1)>0$ independent of $\rho_0$ with $\|\rho\|_{\E_{1,T}}, \|u\|_{\E_{1,T}} \leq R$. For any two solutions, there exists $\overline{T} \in (0,T]$ such that the solutions coincide on $[0,\overline{T}]$.
\end{theorem}

\begin{proof}
For sufficiently small $R^\Sigma>0$ and sufficiently large $R^c,R^h>0$, choosing $\delta_0>0$ and $T>0$ sufficiently small, Assumptions \ref{Vor} and the conditions of Theorems \ref{lokEx_h} and \ref{lokEx_c} are satisfied. The existence of a solution $(\rho,u)$ with $\rho, u \in \E_{1,T}$ then follows directly from Theorems \ref{lokEx_h} and \ref{lokEx_c}. With $R \defr \max \{R^c,R^h\}$, we have $\|\rho\|_{\E_{1,T}}, \|u\|_{\E_{1,T}} \leq R$, where $R^c, R^h$ and thus also $R$ only depend on $\Sigma$, $u_0$ and $\delta_1$ (see Theorems \ref{lokEx_h} and \ref{lokEx_c}). The property $\rho(t), u(t) \in h^{2+2\beta+\alpha}(M)$ for all $t \in [0,T]$ is due to Propositions \ref{Neuanfang_h} and \ref{Neuanfang_c} and 
\begin{align*}
\|\rho(t)\|_{h^{1+\alpha}(M)}
\leq \frac{\|\rho(t)-\rho(0)\|_{h^{1+\alpha}(M)}}{|t-0|^\beta} T^\beta + \|\rho_0\|_{h^{1+\alpha}(M)}
\leq \|\rho\|_{\E_{1,T}} T^\beta + \delta_0
< R^\Sigma
\end{align*}
follows with Estimate \eqref{Absch_T}. \\
To prove the stateted uniqueness property of the solution, assume that there exists a second solution $(\tilde{\rho},\tilde{u})$ with $\tilde{\rho},\tilde{u} \in \E_{1,T}$. Choose $R^c$ and $R^h$ as large as in Theorem \ref{lokEx_c}, but at least as large such that $\|u\|_{\E_{1,T}}, \|\tilde{u}\|_{\E_{1,T}} \leq R^c$ and $\|\rho\|_{\E_{1,T}}, \|\tilde{\rho}\|_{\E_{1,T}} \leq R^h$ hold. Then, choose $\overline{T}>0$ as small as in Theorem \ref{lokEx_c} but at least as small such that $\overline{T} \leq T$ holds. As $\delta_0$ is independent of $R^c$ and $R^h$, the conditions of Theorems \ref{lokEx_h} and \ref{lokEx_c} are satisfied. We hence obtain a unique solution in
\begin{align*}
M_{\overline{T}} \defr \big\{ (\bar{\rho},\bar{u}) \in \E_{1,\overline{T}} \times \E_{1,\overline{T}} \, \big| \, \|\bar{\rho}\|_{\E_{1,\overline{T}}} \leq R^h \text{ and } \|\bar{u}\|_{\E_{1,\overline{T}}} \leq R^c \big\}. 
\end{align*}
As we have $(\rho,u), (\tilde{\rho},\tilde{u}) \in M_{\overline{T}}$, the two solutions coincide on $[0,\overline{T}]$.
\end{proof}

If we could apply a continuation argument to the two solutions $(\rho,u)$ and $(\tilde{\rho},\tilde{u})$ from the proof above, we could show that they coincide on the full time interval $[0,T]$ and thus obtain uniqueness of the solution. For this, we would need to ensure that for a solution $(\rho,u)$ at any time $t$, the pair $\big(\rho(t),u(t)\big)$ fulfills the conditions for the initial values in Theorem \ref{lokEx_param}. In particular, $\rho(t)$ needs to be bounded by $\delta_0\big(u(t)\big)$ in the appropriate norm. To achieve this, the dependence of $\delta_0$ on $u_0$ should be controlled in a uniform way. 

\appendix
\section{Appendix}

\subsection{Basic Properties of Hölder Spaces}

In this section, we gather some basic properties of little Hölder spaces. For their proofs, we refer to \cite[Section 2.2]{Buerger}. Via localization, the results transfer to embedded closed hypersurfaces.
 
\begin{lemma}[Hölder Spaces as Continuous Interpolation Spaces]\label{Interpol_ComO} $\phantom{x}$ \\
Let $m \in \N_{\geq 1}$ and $\theta \in (0,1)$ with $\theta m \notin \N$ and let $W \subset \R^d$ be an open subset with regular boundary (see \cite[Section 0.1, pages 2 and 3]{LunardiASG} for an explicit definition). Then, 
\begin{align*}
\big( C^0_b(\overline{W}), C^m_b(\overline{W}) \big)_{\theta} = h^{\theta m}(\overline{W})
\end{align*}
holds with equivalent norms, where $(\cdot,\cdot)_\theta$ denotes the continuous interpolation functor. 
\end{lemma}
 
\begin{lemma}[Embeddings of Hölder Spaces]\label{HS_Einbettung} $\phantom{x}$ \\
Let $W \subset \R^d$ be an open, bounded and convex subset. For any $s_1,s_2 \in \R_{> 0}$ with $s_1 \leq s_2$,
\begin{align*}
h^{s_2}_b(\overline{W},X) \hookrightarrow h^{s_1}_b(\overline{W},X)
\end{align*}
holds. For $X=\R^n$, the statement also holds if $W \subset \R^d$ is an open subset with regular boundary.
\end{lemma}

As a special case of Lemma \ref{HS_Einbettung}, the following remark holds.

\begin{remark}\label{Absch_Höldernorm}
Let $T \in (0,1]$ and let $\alpha_1,\alpha_2 \in (0,1)$ with $\alpha_1 < \alpha_2$. We have $h^{\alpha_2}([0,T],X) \hookrightarrow h^{\alpha_1}([0,T],X)$ with
\begin{align*}
\|f\|_{h^{\alpha_1}([0,T],X)} \leq 2 T^{\alpha_2-\alpha_1} \|f\|_{h^{\alpha_2}([0,T],X)} + \|f(0)\|_X
\end{align*}
for all $f \in h^{\alpha_2}\big([0,T],X\big)$. 
\end{remark}

\begin{lemma}[Embeddings of Hölder Spaces in Time and Space]\label{EinbettungE_1} $\phantom{x}$ \\
Let $T \in (0,\infty)$ and let $\alpha, \beta \in (0,1)$. Furthermore, let $M \subset \R^{d+1}$ be a $d$-dimensional $h^{2+\alpha}$-embedded closed hypersurface. Define $X \defr h^\alpha(M)$, $Y \defr h^{1+\alpha}(M)$ and $Z \defr h^{2+\alpha}(M)$. Then, there exists $\gamma \in (0,1)$ with $\gamma > \beta$ such that
\begin{align*}
h^{1+\beta}([0,T],X) \cap h^\beta([0,T],Z) \hookrightarrow h^\gamma([0,T],Y)
\end{align*}
is a continuous embedding. 
\end{lemma}

\begin{lemma}[Compact Embeddings of Hölder Spaces]\label{kompHölder} $\phantom{x}$ \\
Let $W \subset \R^d$ be an open, bounded and convex subset. For every ${s_1,s_2 \in \R_{>0} \setminus \N}$ with $s_1 < s_2$,
\begin{align*}
h^{s_2}(\overline{W}) \hookrightarrow h^{s_1}(\overline{W})
\end{align*}
is a compact embedding. 
\end{lemma}

\begin{proposition}[Pointwise Product in Hölder Spaces]\label{HS_Produkt} $\phantom{x}$ \\
Let $W \subset \R^d$ be an open, bounded and convex subset and let $s \in \R_{\geq 0}$. Furthermore, let $X_1, X_2, X$ be Banach spaces with a $\R$-bilinear operation $\cdot: X_1 \times X_2 \rightarrow X$ such that $\|u_1 \cdot u_2\|_X \lesssim \|u_1\|_{X_1} \|u_2\|_{X_2}$ holds for all $u_1 \in X_1$, $u_2 \in X_2$. \\
Then, with pointwise multiplication, $f \cdot g \in h_b^s(\overline{W},X)$ with
\begin{align*}
\| f \cdot g \|_{h^s(\overline{W},X)} \leq C \|f\|_{h^s(\overline{W},X_1)} \|g\|_{h^s(\overline{W},X_2)}
\end{align*}
holds for all $f \in h^s_b(\overline{W},X_1)$, $g \in h^s_b(\overline{W},X_2)$. For $X_1=X_2=\R^n$ and $X=\R$, the statement also holds if $W \subset \R^d$ is an open subset with regular boundary. 
\end{proposition}

\begin{proposition}[Composition of Hölder Functions]\label{HS_Verknüpfung} $\phantom{x}$ \\
Let $W_1 \subset \R^{d_1}$, $W_2 \subset \R^{d_2}$ be open, bounded and convex subsets, let $X$ be a Banach space and let $s \in \R_{\geq 0}$. Furthermore, let $\varphi \in h^s_b(\overline{W_1},\R^{d_2})$ such that $\varphi(\overline{W_1}) \subset \overline{W_2}$ holds and $\varphi: \overline{W_1} \rightarrow \R^{d_2}$ is Lipschitz continuous, i.e., there exists a constant $L \geq 0$ with 
\begin{align*}
\sup_{\substack{x, y \in \overline{W_1}\\x \neq y}} \frac{|\varphi(x)-\varphi(y)|}{|x-y|} \leq L. 
\end{align*}
Then, if $F \in h^s_b(\overline{W_2},X)$, we have $F \circ \varphi \in h^s_b(\overline{W_1},X)$. For $X=\R^n$, the statement also holds if $W_1 \subset \R^{d_1}$, $W_2 \subset \R^{d_2}$ are open subsets with regular boundaries.
\end{proposition}

\subsection{Composition Operators of Hölder Regular Functions}\label{Sec_CompOp}

In the following, let $W \subset \R^d$ be an open, bounded and convex subset, let $s \in \R_{\geq 0}$ and let $X,Y,Z$ be Banach spaces.

\begin{lemma}\label{compositionOP_lin}
Let $g: \overline{W} \rightarrow \mathcal{L}(Y,Z)$ and define $G(v): \overline{W} \rightarrow Z, \, \big(G[v]\big)(x) \defr g(x)\big[v(x)\big]$ for any function $v: \overline{W} \rightarrow Y$. 
If $g \in h^s(\overline{W},\mathcal{L}(Y,Z))$, then $G \in \mathcal{L}\big( h^s(\overline{W},Y), h^s(\overline{W},Z) \big)$ holds with 
\begin{align*}
\|G\|_{\mathcal{L}(h^s(\overline{W},Y),h^s(\overline{W},Z))} \lesssim \|g\|_{h^s(\overline{W},\mathcal{L}(Y,Z))}.
\end{align*}
\end{lemma}

\begin{proof}
The result is a consequence of Proposition \ref{HS_Produkt}. 
\end{proof}

\begin{proposition}\label{compositionOP_C0}
Let $U \subset Y$ be an open subset and $K \subset U$ a convex subset. Furthermore, let $f: U \rightarrow Z$ and define $F(u): \overline{W} \rightarrow Z, \, \big(F(u)\big)(x) \defr f\big(u(x)\big)$ for any function $u: \overline{W} \rightarrow U$. Then the following hold: 
\begin{enumerate}
\item[(i)]
If $f \in C^{\lfloor s \rfloor+1}(U,Z)$ with $f \in C^{\lfloor s \rfloor+1}_b(K,Z)$, then we have $F(u) \in h^s(\overline{W},Z)$ for all $u \in h^s(\overline{W},K)$. 
In addition, for any $R>0$ there exists a $C(R)>0$ such that 
\begin{align*}
\|F(u)\|_{h^s(\overline{W},Z)} \leq C(R)
\end{align*}
holds for all $u \in h^s(\overline{W},K)$ with $\|u\|_{h^s(\overline{W},Y)} \leq R$.
\item[(ii)]
If $f \in C^{\lfloor s \rfloor+2}(U,Z)$ with $f \in C^{\lfloor s \rfloor+2}_b(K,Z)$, then $F \in C^0\big( h^s(\overline{W},K), h^s(\overline{W},Z)\big)$. In particular, for any $R>0$ there exists a $C(R)>0$ such that we have
\begin{align*}
\|F(u_1)-F(u_2)\|_{h^s(\overline{W},Z)} \leq C(R) \|u_1-u_2\|_{h^s(\overline{W},Y)}
\end{align*}
for all $u_1,u_2 \in h^s(\overline{W},K)$ with $\|u_j\|_{h^s(\overline{W},Y)} \leq R$. Moreover, $F \in C^0_b\big( \mathcal{B}, h^s(\overline{W},Z)\big)$ holds for all subsets $\mathcal{B} \subset h^s(\overline{W},K)$ that are bounded in $h^s(\overline{W},Y)$.
\item[(iii)]
If $f \in C^{k+\lfloor s \rfloor+2}(U,Z)$ with $f \in C^{k+\lfloor s \rfloor+2}_b(K,Z)$, then $F \in C^k\big( h^s(\overline{W},V), h^s(\overline{W},Z)\big)$ and $F \in C^k_b\big( \mathcal{B},h^s(\overline{W},Z)\big)$ hold for any $k \in \N_{\geq 0}$, any open subset $V \subset K$ and any bounded subset $\mathcal{B} \subset h^s(\overline{W},V)$. 
\end{enumerate}
\end{proposition}

\begin{proof}
First, we prove the statements (i) and (ii) for $s \in [0,1)$, i.e. $\lfloor s \rfloor=0$:  Due to the mean value theorem and the convexity of $K$, we have 
\begin{align*}
\|F(u)\|_{h^s(\overline{W},Z)} &\leq \|f\|_{C^1(K,Z)}\big(1+\|u\|_{h^s(\overline{W},Y)}\big) \text{ and } \\
\|F(u_1)-F(u_2)\|_{h^s(\overline{W},Z)} &\leq \|f\|_{C^2(K,Z)}(1+R) \|u_1-u_2\|_{h^s(\overline{W},Y)} 
\end{align*}
for all $u \in h^s(\overline{W},K)$ and $u_1,u_2 \in h^s(\overline{W},K)$ with $\|u_j\|_{h^s(\overline{W},Y)} \leq R$. 
Boundedness of the function $F: \mathcal{B} \rightarrow h^s(\overline{W},Z)$ for a bounded subset $\mathcal{B} \subset h^s(\overline{W},K)$ follows directly from the estimate in (i). \\
The general statements (i) and (ii) for arbitrary $s \in \R_{\geq 0}$ follow by mathematical induction on $\lfloor s \rfloor$, using Lemma \ref{HS_Einbettung} and the fact that differentiability of $f$ and $u$ implies differentiability of $F(u)$ and we have $\partial_{x_i} \big(F(u)\big) = A(u)\big( \partial_{x_i}u \big)$ with $A(v): \overline{W} \rightarrow \mathcal{L}(Y,Z)$, $\big(A(v)\big)(x) \defr Df\big(v(x)\big)$ for any function $v: \overline{W} \rightarrow U$. Applying the induction hypothesis and Lemma \ref{compositionOP_lin} on $A$ conclude the inductive step. \\
We show the statement (iii) using mathematical induction: Assume that the claim is satisfied for a fixed $k \in \N_0$ and choose a function $f \in C^{k+\lfloor s \rfloor+3}(U,Z) \cap C^{k+\lfloor s \rfloor+3}_b(K,Z)$ as well as an open subset $V \subset K$ and a bounded subset $\mathcal{B} \subset h^s(\overline{W},V)$. Define $A(u): \overline{W} \rightarrow \mathcal{L}(Y,Z)$, $\big(A(u)\big)(x) \defr Df\big(u(x)\big)$. The induction hypothesis together with Lemma \ref{compositionOP_lin} yields 
\begin{align*}
F &\in C^k\big( h^s(\overline{W},V), h^s(\overline{W},Z)\big) \cap C^k_b\big( \mathcal{B}, h^s(\overline{W},Z)\big) \text{ and } \\
A &\in C^k\big( h^s(\overline{W},V), \mathcal{L}\big(h^s(\overline{W},Y), h^s(\overline{W},Z) \big)\big) \cap C^k_b\big( \mathcal{B}, \mathcal{L}\big( h^s(\overline{W},Y), h^s(\overline{W},Z) \big)\big).
\end{align*}
It remains to show that $F$ is Fréchet-differentiable with $DF=A$. \\
Fix $u_0 \in h^s(\overline{W},V)$. Due to $f \in C^{\lfloor s \rfloor+3}_b(V,Z)$,  the statement of (i) yields $D^2f(u_0+\theta h) \in h^s\big( \overline{W}, \mathcal{L}(Y,\mathcal{L}(Y,Z)) \big)$ with $\|D^2f(u_0+\theta h)\|_{h^s(\overline{W},\mathcal{L}(Y,\mathcal{L}(Y,Z)))} \leq C(u_0)$ for all $\theta \in [0,1]$ and $h \in h^s(\overline{W},Y)$ with $\|h\|_{h^s(\overline{W},Y)}$ sufficiently small. A Taylor expansion, the triangle inequality for integrals and Lemma \ref{compositionOP_lin} imply
\begin{align*}
&\big\| F(u_0+h) - F(u_0) - A(u_0)h \big\|_{h^s(\overline{W},Z)} \\
&\leq \int_0^1 (1-\theta) \| D^2f(u_0+\theta h) \|_{h^s(\overline{W},\mathcal{L}(Y,\mathcal{L}(Y,Z)))} \|h\|_{h^s(\overline{W},Y)}^2 \, \mathrm{d}\theta 
\leq C(u_0) \|h\|_{h^s(\overline{W},Y)}^2. \qedhere
\end{align*}
\end{proof}

We derive two corollaries from this main result: The first one reduces to the case of a compact subset $K \subset U$ and the second one deals with a finite dimensional setting. 

\begin{corollary}\label{compositionOP_compact}
Let $U \subset Y$ be an open subset and $K \subset U$ a compact and convex subset. Furthermore, let $f: U \rightarrow Z$ and define $F(u): \overline{W} \rightarrow Z, \, \big(F(u)\big)(x) \defr f\big(u(x)\big)$ for any function $u: \overline{W} \rightarrow U$. Then the following hold: 
\begin{enumerate}
\item[(i)]
If $f \in C^{\lfloor s \rfloor+1}(U,Z)$, then $F(u) \in h^s(\overline{W},Z)$ for all $u \in h^s(\overline{W},K)$. Moreover, for any $R>0$ there exists a constant $C(K,R)>0$ such that 
\begin{align*}
\|F(u)\|_{h^s(\overline{W},Z)} \leq C(K,R)
\end{align*}
holds for all $u \in h^s(\overline{W},K)$ with $\|u\|_{h^s(\overline{W},Y)} \leq R$.
\item[(ii)]
If $f \in C^{\lfloor s \rfloor+2}(U,Z)$, then $F \in C^0\big( h^s(\overline{W},K), h^s(\overline{W},Z) \big) \cap C^0_b\big( \mathcal{B},h^s(\overline{W},Z)\big)$ holds for all bounded subsets $\mathcal{B} \subset h^s(\overline{W},K)$. Moreover, for any $R>0$ there exists a constant $C(K,R)>0$ such that 
\begin{align*}
\|F(u_1)-F(u_2)\|_{h^s(\overline{W},Z)} \leq C(K,R) \|u_1-u_2\|_{h^s(\overline{W},Y)}
\end{align*}
holds for all $u_1, u_2 \in h^s(\overline{W},K)$ with $\|u_j\|_{h^s(\overline{W},Y)} \leq R$.
\item[(iii)]
If $f \in C^{k+\lfloor s \rfloor+2}(U,Z)$, then $F \in C^k\big( h^s(\overline{W},V), h^s(\overline{W},Z)\big) \cap C^k_b\big( \mathcal{B},h^s(\overline{W},Z)$ holds for any $k \in \N_{\geq 0}$, any open subset $V \subset K$ and any bounded subset $\mathcal{B} \subset h^s(\overline{W},V)$.
\end{enumerate}
\end{corollary}

\begin{proof} The statements (i) and (ii) follow directly from Proposition \ref{compositionOP_C0} using the compactness of $K$. Therefore, we only prove the statement (iii). \\
As $K$ is convex with $V \subset K$, also the convex hull $\text{conv } V \subset K$ of $V$ is a subset of $K$. Its interior $\tilde{V} \defr (\text{conv } V)^\circ$ therefore is an open and convex set with $\tilde{V} \subset K$. We then have $f \in C^{k+\lfloor s \rfloor+2}(U,Z)$ and $K \subset U$ compact, $\tilde{V} \subset K$. Thus, $f \in C^{k+\lfloor s \rfloor+2}_b(\tilde{V},Z)$ holds with the open and convex subset $\tilde{V} \subset Y$. Proposition \ref{compositionOP_C0}(iii) yields $F \in C^k\big( h^s(\overline{W},\tilde{V}), h^s(\overline{W},Z) \big)$ and $F \in C^k_b\big( \mathcal{B},h^s(\overline{W},Z)\big)$ for all bounded subsets $\mathcal{B} \subset h^s(\overline{W},\tilde{V})$. As $V \subset Y$ is open, $V \subset \tilde{V}$ holds and therefore $F \in C^k\big( h^s(\overline{W},V), h^s(\overline{W},Z) \big)$ and $F \in C^k_b\big(\mathcal{B},h^s(\overline{W},Z) \big)$ for all bounded subsets $\mathcal{B} \subset h^s(\overline{W},V)$ follows. \qedhere
\end{proof}

\begin{corollary}\label{compositionOP_finitedim}
Let $f: U \rightarrow \R^N$ for an open subset $U \subset \R^M$ and define $F(u): \overline{W} \rightarrow \R^N$, $\big(F(u)\big)(x) \defr f\big(u(x)\big)$ for any function $u: \overline{W} \rightarrow U$. Then the following hold:
\begin{enumerate}
\item[(i)]
If $f \in C^{\lfloor s \rfloor+1}(U,\R^N)$, then $F(u) \in h^s(\overline{W},\R^N)$ for all $u \in h^s(\overline{W},U)$. 
\item[(ii)]
If $f \in C^{k+\lfloor s \rfloor+2}(U,\R^N)$, then $F \in C^k\big( h^s(\overline{W},U), h^s(\overline{W},\R^N)\big) \cap C^k_b\big( \mathcal{B},h^s(\overline{W},\R^N)\big)$ holds for any $k \in \N_{\geq 0}$ and any bounded subset $\mathcal{B} \subset h^s(\overline{W},\mathcal{A})$ with $\mathcal{A} \subset \R^M$ closed and $\mathcal{A} \subset U$.
\end{enumerate}
\end{corollary}

\begin{proof}
For $U=\R^M$, the statements follow directly from Corollary \ref{compositionOP_compact}, as any bounded set in $\R^M$ can easily be enclosed in a compact set, that still remains a subset of $U$.
Now, assume $U \subsetneq \R^M$. For any $f \in C^l(U,\R^N)$ and any compact subset $K \subset U$, choose a cut-off function $\xi \in C^\infty(\R^m,\R)$ with $\xi \equiv 1$ on $K$, $\xi \equiv 0$ on $\R^M \setminus U$ and $0 \leq \xi \leq 1$. The results for $\tilde{f} \defr \xi f \in C^l(\R^M,\R^N)$ from the first part of the proof then can be transferred to $f$.
\end{proof}

\begin{remark}[Hölder Regularity for the Inverse of a Matrix]\label{h_Matrixinverse} $\phantom{x}$ \\
Let $W \subset \R^d$ be an open, bounded and convex subset and let $s \in \R_{\geq 0}$. The set of invertible matrices 
\begin{align*}
U \defr \{ A \in \R^{n \times n} \, | \, \det A \neq 0 \}
\end{align*}
is an open subset of $\R^{n \times n}$. For the matrix inversion mapping $f: U \rightarrow \R^{n \times n}, \, f(A) \defr A^{-1}$, we have $f \in C^1(U,\R^{n \times n})$ with 
\begin{align*}
Df(A)[H] = -f(A) \cdot H \cdot f(A)
\end{align*}
for all $A \in U$ and $H \in \R^{n \times n}$. Thus, we have $Df \in C^1(U,\mathcal{L}(\R^{n \times n},\R^{n \times n}))$ (see e.g. \cite[§2 Satz 2.7(ii)]{Ruz}) and then recursively, $f \in C^\infty(U,\R^{n \times n})$ follows. Corollary \ref{compositionOP_finitedim}(ii) thus implies 
\begin{align*}
(\cdot)^{-1} \in C^\infty \big( h^s(\overline{W},U), h^s(\overline{W},\R^{n \times n}) \big) \cap C^\infty_b \big( \mathcal{B}, h^s(\overline{W},\R^{n \times n}) \big)
\end{align*}
for the inversion $(\cdot)^{-1}$ of matrices with $\mathcal{B} \subset h^s(\overline{W},\mathcal{A})$ an arbitrary bounded subset and $\mathcal{A} \subset \R^{n \times n}$ closed with $\mathcal{A} \subset U$. In particular, for any $A \in h^s(\overline{W},\R^{n \times n})$ with $\det A \neq 0$ on $\overline{W}$, also $A^{-1} \in h^s(\overline{W},\R^{n \times n})$ holds. 
\end{remark}

\subsection{Results using Generators of Semigroups}

In this section, we state some results using the theory of semigroups. Again, we refer to \cite[Section 2.3.4]{Buerger} for the proofs. 

\begin{proposition}[Maximal Regularity]\label{L_bij_allg} $\phantom{x}$ \\
Let $A: \mathcal{D}(A) \subset X \rightarrow X$ generate an analytic $C^0$-semigroup in a Banach space $X$. Furthermore, let $\beta \in (0,1)$ and $T \in (0,1]$. We have 
\begin{enumerate}
\item[(i)]
$h^{1+\beta}([0,T],X) \cap h^\beta\big([0,T],\mathcal{D}(A)\big) \hookrightarrow C^1\big([0,T],\mathcal{D}_A(\beta)\big)$ and
\item[(ii)]
$L_T: h^{1+\beta}([0,T],X) \cap h^\beta\big([0,T],\mathcal{D}(A)\big) \rightarrow \big(h^\beta([0,T],X) \times \mathcal{D}(A)\big)_+, \, L_T[\rho] \defr \binom{\partial_t\rho - A\rho}{\rho(0)}$ is bijective with $\sup_{0<T\leq 1} \|L_T^{-1}\|_{\mathcal{L}} < \infty$. 
\end{enumerate}
\end{proposition}

\begin{lemma}[Improved Regularity for Preimages]\label{Neuanfang_Lemma} $\phantom{x}$ \\
Let $s_1,s_2 \in (0,2) \setminus \{1\}$ with $s_1 < s_2$. Let $M \subset \R^{d+1}$ be an $h^{2+s_2}$-embedded closed hypersurface and let $A: h^{2+s_i}(M) \rightarrow h^{s_i}(M)$ generate an analytic $C^0$-semigroup for both $i \in \{1,2\}$. Then, any $v \in h^{2+s_1}(M)$ with $Av \in h^{s_2}(M)$ already fulfills $v \in h^{2+s_2}(M)$.
\end{lemma}

\begin{definition}[(Uniform) Ellipticity]\label{def_elliptic} $\phantom{x}$ \\
Let $\Omega \subset \R^d$ be an arbitrary subset. A matrix valued function $A: \Omega \rightarrow \R^{n \times n}$ is called
\begin{enumerate}
\item[(i)] elliptic (or positive definite on $\Omega$), if 
\begin{align*}
\xi^\top A(x) \xi > 0
\end{align*}
holds for every $x \in \Omega$ and $\xi \in \R^n \setminus \{0\}$ and
\item[(ii)] uniformly elliptic, if there exists $C>0$ so that 
\begin{align*}
\xi^\top A(x) \xi \geq C |\xi|^2
\end{align*}
holds for every $x \in \Omega$ and $\xi \in \R^n$.
\end{enumerate}
If $\Omega$ is compact and $A$ is continuous on $\Omega$, the two properties coincide. 
\end{definition}

\begin{proposition}[Differential Operators as Generators]\label{HG_elliptic_semigroup_M} $\phantom{x}$ \\
Let $s \in (0,2) \setminus \{1\}$ and let ${M \subset \R^{d+1}}$ be an $h^{2+s}$-embedded closed hypersurface. Moreover, let $A \in \mathcal{L}\big( h^{2+s}(M), h^s(M) \big)$ be a symmetric, elliptic differential operator of second order, i.e. given a local parameterization $(\gamma,W)$ of $M$, 
\begin{align*}
Au \circ \gamma = a:D^2(u \circ \gamma) + b \cdot \nabla (u \circ \gamma) + c(u \circ \gamma)
\end{align*}
holds for every $u \in h^{2+s}(M)$, with $a \in h^s(\overline{W},\R^{d \times d})$, $b \in h^s(\overline{W},\R^d)$ and $c \in h^s(\overline{W},\R)$ such that the matrix $a$ is symmetric and positive definite on $\overline{W}$. Then, 
\begin{align*}
A: \mathcal{D}(A) \defr h^{2+s}(M) \subset h^s(M) \rightarrow h^s(M)
\end{align*}
generates an analytic $C^0$-semigroup. 
\end{proposition}

\subsection{Results for Hypersurfaces}

We show well-definedness of the parameterization of evolving immersed hypersurfaces used in Definition \ref{evolvHF_rho}. 

\begin{lemma}\label{I>0}
Let $\Sigma = \theta(M)$ be a $C^1$-immersed closed hypersurface. Then, we have
\begin{align*}
\inf_{p \in M} \inf_{\substack{v \in T_pM, \\|v|=1}} \big|\mathrm{d}_p\theta[v]\big| > 0.
\end{align*}
\end{lemma}

\begin{proof}
Let $d \defr \text{dim } M$ and choose a local parameterization $(\gamma,W)$ of $M$. In particular, $\gamma \in C^1(\overline{W},\R^{d+1})$ is an embedding with $\gamma(\overline{W}) \subset M$. Set 
\begin{align*}
v(\alpha,x) \defr \frac{\sum_i \alpha^i \partial_i \gamma(x)}{\left| \sum_i \alpha^i \partial_i \gamma(x) \right|} 
\end{align*}
for $\alpha \in \R^d \setminus \{0\}$ and $x \in \overline{W}$. Then $v(\alpha,x) \in T_{\gamma(x)}M$ holds with $|v(\alpha,x)|=1$ for every $\alpha \in \R^d \setminus \{0\}$ and $x \in \overline{W}$. Moreover, for $\beta \defr \frac{\alpha}{|\alpha|}$, we have 
\begin{align*}
v(\beta,x) 
= \frac{\sum_i \beta^i \partial_i \gamma(x)}{\left| \sum_i \beta^i \partial_i \gamma(x) \right|} 
= \frac{\frac{1}{|\alpha|}\sum_i \alpha^i \partial_i \gamma(x)}{\left| \frac{1}{|\alpha|} \sum_i \alpha^i \partial_i \gamma(x) \right|} 
= \frac{\sum_i \alpha^i \partial_i \gamma(x)}{\left| \sum_i \alpha^i \partial_i \gamma(x) \right|} 
= v(\alpha,x)
\end{align*}
for every $x \in \overline{W}$. So, with $\mathcal{S} \defr \{ \alpha \in \R^d, |\alpha|=1\}$,
\begin{align*}
\big\{ v \in T_{\gamma(x)}M, \, |v|=1 \big\}
= \big\{ v(\alpha,x) \, \big| \, \alpha \in \mathcal{S} \big\}
\end{align*}
and in particular 
\begin{align*}
\inf_{\substack{v \in T_{\gamma(x)}M, \\|v|=1}} \big|\mathrm{d}_{\gamma(x)}\theta[v]\big|
= \inf_{\alpha \in \mathcal{S}} \big| \mathrm{d}_{\gamma(x)}\theta\big[ v(\alpha,x) \big] \big|
\end{align*}
follows for every $x \in \overline{W}$. We have 
\begin{align*}
\big| \mathrm{d}_{\gamma(x)}\theta\big[v(\alpha,x)\big] \big|
= \left| \mathrm{d}_{\gamma(x)}\theta\left[\frac{\sum_i \alpha^i \partial_i \gamma(x)}{\left| \sum_i \alpha^i \partial_i \gamma(x) \right|}\right] \right|
= \left| \frac{\sum_i \alpha^i \mathrm{d}_{\gamma(x)}\theta[\partial_i \gamma(x)]}{\left| \sum_i \alpha^i \partial_i \gamma(x) \right|} \right|
= \frac{\left| \sum_i \alpha^i \partial_i (\theta \circ \gamma)(x)\right|}{\left| \sum_i \alpha^i \partial_i \gamma(x) \right|}
\end{align*}
for all $\alpha \in \mathcal{S}$ and $x \in \overline{W}$. Due to $\theta \in C^1(M,\R^{d+1})$, $\gamma \in C^1(\overline{W},\R^{d+1})$ with $\gamma(\overline{W}) \subset M$ and $\partial_i \gamma \neq 0$ on $\overline{W}$ for all $i=1,...,d$ by the immersion property of $\gamma$, thus
\begin{align*}
(\alpha,x) \mapsto \big| \mathrm{d}_{\gamma(x)}\theta\big[v(\alpha,x)\big] \big| \in C^0(\mathcal{S} \times \overline{W})
\end{align*}
follows. Because $\theta$ is an immersion, $\big| \mathrm{d}_{\gamma(x)}\theta\big[v(\alpha,x)\big] \big| > 0$ holds for all $(\alpha,x) \in \mathcal{S} \times \overline{W}$ and then compactness of $\mathcal{S} \times \overline{W}$ implies
\begin{align*}
\inf_{x \in \overline{W}} \inf_{\substack{v \in T_{\gamma(x)}M, \\|v|=1}} \big|\mathrm{d}_{\gamma(x)}\theta[v]\big| = \inf_{x \in \overline{W}} \inf_{\alpha \in \mathcal{S}} \big| \mathrm{d}_{\gamma(x)}\theta\big[ v(\alpha,x) \big] \big| > 0. 
\end{align*}
Finally, as $M$ is compact, it can be covered by finitely many local parameterizations $(\gamma,W)$ and therefore the claim follows. 
\end{proof}

\begin{lemma}\label{theta_imm}
Let $\Sigma = \bar{\theta}(M) \subset \R^{d+1}$ be a $C^2$-immersed closed hypersurface with unit normal $\nu_\Sigma$. Furthermore, let $\rho \in C^1(M,\R)$ with $\|\rho\|_{C^0(M,\R)}$ sufficiently small. Then, 
\begin{align*}
\theta_\rho: M \rightarrow \R^{d+1}, \phantom{x} \theta_\rho(p) \defr \bar{\theta}(p) + \rho(p)\nu_\Sigma(p)
\end{align*}
is an immersion. 
\end{lemma}

\begin{proof}
We have $\nu_\Sigma \in C^1(M,\R^{d+1})$ and thus $\theta_\rho = \bar{\theta}+\rho\nu_\Sigma \in C^1(M,\R^{d+1})$. For any local parameterization $(\gamma,W)$ of $M$, the domain $\overline{W} \subset \R^d$ is compact and hence
\begin{align*}
S_{(\gamma,W)} 
&\defr \sup_{x \in \overline{W}} \big\| \mathrm{d}_{\gamma(x)} \nu_\Sigma \big\|_{\mathcal{L}(T_{\gamma(x)}M,\R^{d+1})} \\
&\lesssim \sup_{x \in \overline{W}} \max_{i=1,...,d} \frac{\big| \mathrm{d}_{\gamma(x)} \nu_\Sigma \big( \partial_i \gamma(x) \big) \big|}{|\partial_i \gamma(x)|}
= \sup_{x \in \overline{W}} \max_{i=1,...,d} \frac{|\partial_i (\nu_\Sigma \circ \gamma)(x)|}{|\partial_i \gamma(x)|} < \infty
\end{align*}
holds. Because $M$ is compact, it can be covered by finitely many local parameterizations $(\gamma_l,W_l)_{l=1,...,L}$ and therefore
\begin{align*}
S \defr \sup_{p \in M} \big\| \mathrm{d}_p \nu_\Sigma \big\|_{\mathcal{L}(T_pM,\R^{d+1})} \leq \max_{l=1,...,L} S_{(\gamma_l,W_l)} < \infty
\end{align*}
follows. As the mean curvature $H=-\Div_\Sigma \nu_\Sigma$ is not the zero function on closed hypersurfaces, $S \neq 0$ holds. Further, Lemma \ref{I>0} implies
\begin{align*}
I \defr \inf_{p \in M} \inf_{\substack{v \in T_pM, \\|v|=1}} \big|\mathrm{d}_p\bar{\theta}[v]\big| > 0.
\end{align*}
Hence,
\begin{align*}
R \defr \frac{I}{2S} > 0
\end{align*}
is well-defined. Assume $\|\rho\|_{C^0(M)} \leq R$. 
For all $p \in M$ and $v \in T_pM$
\begin{align}\label{eq_theta_imm_perp}
\mathrm{d}_p\bar{\theta}[v], \mathrm{d}_p\nu_\Sigma[v] \in T_p\Sigma
\phantom{xx} \text{ and } \phantom{xx}
\nu_\Sigma(p) \perp T_p\Sigma
\end{align}
hold. Due to $\rho(p), \mathrm{d}_p\rho[v] \in \R$, we thus have
\begin{align*}
\big| \mathrm{d}_p \theta_\rho[v] \big|^2
&= \big| \mathrm{d}_p\bar{\theta}[v] + \mathrm{d}_p\rho[v] \nu_\Sigma(p) + \rho(p) \mathrm{d}_p\nu_\Sigma[v] \big|^2 \\
&= \big| \mathrm{d}_p\bar{\theta}[v] + \rho(p) \mathrm{d}_p\nu_\Sigma[v] \big|^2 + \big| \mathrm{d}_p\rho[v] \big|^2 \\
&\geq \big| \mathrm{d}_p\bar{\theta}[v] + \rho(p) \mathrm{d}_p\nu_\Sigma[v] \big|^2
\end{align*}
and then
\begin{align*}
\big| \mathrm{d}_p \theta_\rho[v] \big|
\geq \big| \mathrm{d}_p\bar{\theta}[v] + \rho(p) \mathrm{d}_p\nu_\Sigma[v] \big|
\geq \big| \mathrm{d}_p\bar{\theta}[v] \big| - R \big| \mathrm{d}_p\nu_\Sigma[v] \big|
\geq I - RS
= \frac{I}{2} 
> 0
\end{align*}
follows for $|v|=1$.
In particular, $\mathrm{d}_p\theta_\rho: T_pM \rightarrow \R^{d+1}$ is injective and therefore $\theta_\rho: M \rightarrow \R^{d+1}$ is an immersion.
\end{proof}

\providecommand{\bysame}{\leavevmode\hbox to3em{\hrulefill}\thinspace}
\providecommand{\MR}{\relax\ifhmode\unskip\space\fi MR }
\providecommand{\MRhref}[2]{%
  \href{http://www.ams.org/mathscinet-getitem?mr=#1}{#2}
}
\providecommand{\href}[2]{#2}


\begin{thebibliography}{10}

\bibitem{BarrettDeckelnickStyles}
John~W. Barrett, Klaus Deckelnick, and Vanessa Styles, \emph{Numerical analysis
  for a system coupling curve evolution to reaction diffusion on the curve},
  SIAM Journal on Numerical Analysis \textbf{55} (2017), no.~2, 1080--1100.

\bibitem{BoyerFabrie}
Franck Boyer and Pierre Fabrie, \emph{Mathematical {T}ools for the {S}tudy of
  the {I}ncompressible {N}avier-{S}tokes {E}quations and {R}elated {M}odels},
  vol. 183, Springer Science \& Business Media, 2012.

\bibitem{Buerger}
Felicitas B{\"u}rger, \emph{Interaction of {M}ean {C}urvature {F}low and a
  {D}iffusion {E}quation}, dissertation, Universität Regensburg, 2021,
  \url{https://epub.uni-regensburg.de/51215/}.

\bibitem{DeckelnickStyles_FEEA}
Klaus Deckelnick and Vanessa Styles, \emph{Finite element error analysis for a
  system coupling surface evolution to diffusion on the surface}, arXiv
  preprint arXiv:2104.04827 (2021).

\bibitem{ElliottGarckeKovacs}
Charles~M. Elliott, Harald Garcke, and Balázs Kovács, \emph{Numerical
  analysis for the interaction of mean curvature flow and diffusion on closed
  surfaces}, arXiv preprint arXiv:2202.03302 (2022).

\bibitem{ES}
Joachim Escher and Gieri Simonett, \emph{A center manifold analysis for the
  {M}ullins--{S}ekerka model}, Journal of Differential Equations \textbf{143}
  (1998), no.~2, 267--292.

\bibitem{Huisken}
Gerhard Huisken, \emph{Flow by mean curvature of convex surfaces into spheres},
  Journal of Differential Geometry \textbf{20} (1984), no.~1, 237--266.

\bibitem{KovacsLiLubich_AConvergentAlgorithmForForcedMCF}
Bal{\'a}zs Kov{\'a}cs, Buyang Li, and Christian Lubich, \emph{A convergent
  algorithm for forced mean curvature flow driven by diffusion on the surface},
  Interfaces and Free Boundaries \textbf{22} (2020), no.~4, 443--464.

\bibitem{KovacsLiLubichPowerGuerra}
Bal{\'a}zs Kov{\'a}cs, Buyang Li, Christian Lubich, and Christian~A.
  Power~Guerra, \emph{Convergence of finite elements on an evolving surface
  driven by diffusion on the surface}, Numerische Mathematik \textbf{137}
  (2017), no.~3, 643--689.

\bibitem{KovacsLubich_LinearlyImplicitFDOfSurfaceEvolution}
Bal{\'a}zs Kov{\'a}cs and Christian Lubich, \emph{Linearly implicit full
  discretization of surface evolution}, Numerische Mathematik \textbf{140}
  (2018), no.~1, 121--152.

\bibitem{LunardiASG}
Alessandra Lunardi, \emph{{Analytic Semigroups and Optimal Regularity in
  Parabolic Problems}}, Springer Science \& Business Media, 2012.

\bibitem{Mantegazza}
Carlo Mantegazza, \emph{{Lecture Notes on Mean Curvature Flow}}, vol. 290,
  Springer Science \& Business Media, 2011.

\bibitem{PozziStinner1}
Paola Pozzi and Bj{\"o}rn Stinner, \emph{Curve shortening flow coupled to
  lateral diffusion}, Numerische Mathematik \textbf{135} (2017), no.~4,
  1171--1205.

\bibitem{PozziStinner2}
\bysame, \emph{Elastic flow interacting with a lateral diffusion process: the
  one-dimensional graph case}, IMA Journal of Numerical Analysis \textbf{39}
  (2019), no.~1, 201--234.

\bibitem{PS}
Jan Pr{\"u}ss and Gieri Simonett, \emph{{Moving Interfaces and Quasilinear
  Parabolic Evolution Equations}}, vol. 105, Springer, 2016.

\bibitem{Ruz}
Michael Ruzicka, \emph{Nichtlineare {F}unktionalanalysis: Eine
  {E}inf{\"u}hrung}, Springer-Verlag, 2006.

\end{thebibliography}

\end{document}